\documentclass{amsart}
\usepackage{mathrsfs}
\usepackage{amsmath,amssymb}

\theoremstyle{plain}
\newtheorem{thm}{Theorem}[section]

\newtheorem{lem}[thm]{Lemma}
\newtheorem{cor}[thm]{Corollary}

\theoremstyle{definition}

\theoremstyle{remark}
\newtheorem{rmk}[thm]{Remark}
\newtheorem{ques}[thm]{Question}
\newtheorem{eg}[thm]{Example}
\newtheorem{con}[thm]{Conjecture}

\newcommand{\sA}{\mathscr{A}}
\newcommand{\sO}{\mathscr{O}}
\newcommand{\sW}{\mathscr{W}}

\newcommand{\bC}{\mathbb{C}}
\newcommand{\bD}{\mathbb{D}}
\newcommand{\bZ}{\mathbb{Z}}
\newcommand{\bR}{\mathbb{R}}
\newcommand{\B}{\mathcal{B}}
\newcommand{\D}{\mathcal{D}}

\newcommand{\dd}{\mathbb{D}}

\newcommand{\BH}{\mathcal{B(H)}}

\newcommand{\bN}{\mathbb{N}}
\newcommand{\cH}{{\mathcal{H}}}
\newcommand{\cK}{{\mathcal{K}}}
\newcommand{\cM}{{\mathcal{M}}}

\newcommand{\cA}{{\mathcal{A}}}
\newcommand{\lhdd}{\preccurlyeq}

\newcommand{\eps}{\varepsilon}
\newcommand{\ep}{\varepsilon}
\newcommand{\al}{\alpha}
\newcommand{\sig}{\sigma}
\newcommand{\iid}{\textup{i.i.d.}~}
\newcommand{\ind}{\operatorname{ind}}

\newcommand{\essran}{\operatorname{essran}}

\newcommand{\KH}{\mathcal{K(H)}}

\newcommand{\card}{\textup{card}}
\newcommand{\om}{\omega}
\newcommand{\Om}{\Omega}
\newcommand{\ran}{\textup{ran~}}
\newcommand{\rank}{\textup{rank}}
\newcommand{\Lat}{\textup{Lat}}

\newcommand{\sot}{\textsc{sot}}
\newcommand{\la}{\langle}
\newcommand{\ra}{\rangle}
\newcommand{\lam}{\lambda}
\newcommand{\bP}{\mathbb{P}}
\newcommand{\bE}{\mathbb{E}}

\begin{document}

\title[Random weighted shifts]{Random weighted shifts}

\author[G. Cheng]{Guozheng Cheng}
\address{Department of Mathematics\\ Sun Yat-sen University\\  Guangzhou 510275\\ P. R. China }
\email{chenggzh@mail.sysu.edu.cn}

\author[X. Fang]{Xiang Fang}
\address{Department of Mathematics \\
National Central University\\
Chungli, Taiwan }
\email{ xfang@math.ncu.edu.tw}

\author[S. Zhu]{Sen Zhu}
\address{Department of Mathematics\\Jilin University\\Changchun 130012\\P. R. China}
\email{zhusen@jlu.edu.cn}



\begin{abstract}
In this paper we initiate the study of a  fundamental yet untapped random model of non-selfadjoint, bounded linear operators acting on a separable complex Hilbert space.
We replace the weights $w_n=1$ in the classical unilateral shift $T$, defined as  $Te_n=w_ne_{n+1}$, where $\{e_n\}_{n=1}^\infty$ form an orthonormal basis of a complex Hilbert space, by a sequence of i.i.d. random variables $\{X_n\}_{n=1}^{\infty}$; that is, $w_n=X_n$. This paper answers  basic questions concerning such a model. We propose that this model can be studied in comparison with the classical Hardy/Bergman/Dirichlet spaces in function-theoretic operator theory.

We calculate the spectra and determine their fine structures   (Section \ref{S:spectral}).
We classify the samples up to four equivalence relationships (Section \ref{S:SampleClassify}).
 We introduce a family of random Hardy spaces and determine the growth rate of the coefficients of analytic functions in these spaces (Section \ref{S:random Hardy}).
 We compare them with three types of classical operators (Section \ref{S:vonNeumanIneq}); this is achieved in the form of generalized von Neumann inequalities.
 The invariant subspaces are shown to admit arbitrarily large indices and their semi-invariant subspaces model arbitrary contractions almost surely. We discuss a Beurling-type theorem   (Section \ref{S:InvSubSpa}).
 We  determine various non-selfadjoint algebras  generated by $T$ (Section \ref{S:algebra}).
 Their dynamical properties are clarified  (Section \ref{S:Dynamic}).
 Their iterated Aluthge transforms are shown to converge  (Section \ref{S:Aluthge}).

In summary, they provide a new random model from the viewpoint of probability theory, and they provide a new class of analytic functional Hilbert spaces from the viewpoint of operator theory.
The technical novelty in this paper is that the methodology used  draws from three (largely separate) sources:  probability theory, functional Hilbert spaces, and the approximation theory of bounded operators.
\end{abstract}

\date{\today}
\maketitle

\tableofcontents


\section{Introduction}

  The problems considered in the present paper answer basic questions concerning the random counterpart of a canonical, non-selfadjoint, bounded linear operator acting on a separable, complex Hilbert space. In doing so we hope to lay out a foundation for  further  investigation of such a  model, especially in comparison with the Hardy/Bergman/Dirichlet spaces in operator theory.
The background is as follows.
 Currently random matrix theory (RMT) has evolved into a remarkably sophisticated subject.
Random operator theory (ROT), especially for self-adjoint (and unbounded) operators, can be  traced back to the seminal work of P. Anderson \cite{Anderson} in 1958 in his study of localization of disordered systems. (In contrast, E. Wigner proposed his celebrated random matrix model in 1955 to model the nuclei of heavy atoms \cite{Wigner}.) Self-adjoint (and unbounded) ROT deals mostly with difference and differential operators so far, e.g., random Schr\"{o}dinger operators \cite{Random operator, car,klein,lang}; the main theme is, like RMT,  spectral analysis. This has met with great success in both mathematics and mathematical physics.

  On the other hand, non-selfadjoint (and bounded) ROT, fair to say, is still in its infancy, if it exists at all. Upon reviewing existing literature, it is perhaps easy  to conclude that there is a rather large gap in our current knowledge concerning the random theory  of  bounded non-selfadjoint operators acting on infinite dimensional Hilbert spaces. The latter indeed  form a focus of modern operator theory in the last several decades.
There are certainly quite a few references where ``randomness'' meets ``non-selfadjoint'' operators and we shall review some of them at the end of this introduction. In a nutshell, the results so far are usually either isolated or under a general framework;  it seems that no effort has ever been made for a systematic study of any concrete model.

    This paper represents some initial effort to study random operator theory in parallel to non-selfadjoint, bounded operator theory. Our main contribution  is perhaps to convince people, at least  us, that  such a pursuit is meritorious and our model is indeed meaningful and deserves further investigation.
We consider one of the most important examples in non-selfadjoint operator theory, corresponding to the Hardy space $H^2(\dd)$; namely, the unilateral shift:
$$Te_n=w_ne_{n+1}, \quad n =1,2,3, \cdots,$$
where the weights  are simply
 \begin{equation}\label{E:w1} w_n=1, \qquad n =1,2,3, \cdots, \end{equation} and $\{e_n\}_{n=1}^{\infty}$ is an orthonormal basis for a separable, complex Hilbert space $\cH$. This makes perfect sense for real Hilbert spaces, but we choose ``complex'' to take advantage of the toolkit for analytic functions.
The goal is then to study  the ensemble obtained by replacing $w_n=1$ with a sequence of i.i.d. random variables, i.e.,
\begin{equation}\label{E:wX}
w_n=X_n, \quad n =1,2,3, \cdots.
\end{equation}

  In the deterministic setting, it  is probably hard to overestimate the influence of the unilateral shift  (\ref{E:w1}) in non-selfadjoint operator theory, and in complex analysis, via its incarnation as the multiplication operator on the Hardy space $H^2(\mathbb{D})$ over the unit disk. For instance,  the classification of the lattice of invariant subspaces of $T$ in (\ref{E:w1}) leads to the celebrated Beurling theorem \cite{Beurling}, with more contribution by P. Halmos \cite{Halmos61} and P. Lax \cite{Lax} later. The Beurling-Halmos-Lax theorem provides the impetus for numerous subsequent research works for more than half a century and remains active as of today.  Another outstanding example is the study of the extension theory of the $C^*$-algebra generated by $T$, resulting in the so-called BDF theory \cite{BDF,BDF77}, which answered a question of M. Atiyah in algebraic topology concerning finding the homology realization dual to K-theory, which is a cohomology. The BDF theory eventually evolved into a far-reaching field of noncommutative algebraic topology.
Moreover, it is well known that the deterministic unilateral shift plays an important role in system theory and control theory \cite{Nik1, Nik2, Linear}. In the weighted case, J. A. Ball and V. Bolotnikov discussed related ideas at length in the setting of weighted Hardy spaces \cite{Ball}.
It is plausible that the model introduced in this paper has similar roles to play.

  In the random setting, however, the territory  is largely uncharted and it is not even obvious what questions are meaningful for  (\ref{E:wX}). Ideas directly borrowed from operator theory and RMT  need to be  scrutinized.
Eigenvalues, for instance, occupy a central role in RMT, yet operators on infinite dimensional spaces may have none. The role of eigenvalues is sometimes replaced by that of invariant subspaces, which is investigated in Section \ref{S:InvSubSpa} of this paper. Another analogy with eigenvalues in RMT is the distribution of zero sets of random analytic functions. This is a fruitful idea, with important contributions from many researchers. Good entry points for this direction include \cite{HKPV, 2005 Acta, Sodin}.  In the present paper, associated with the model $T=T(\omega)$ is a random Hardy space $H^2_{\mu}=H^2_{\mu}(\omega)$ (Section \ref{S:random Hardy}). Here the randomness is on the Hilbert space instead of individual functions. Obtaining a Blaschke product-type characterization of the zero sets of these random Hardy spaces should provide an interesting contrast with both the classical Hardy space and random analytic functions. This is one of our future goals.

  Concerning our methodology, most techniques used in this paper draw from three sources, usually practiced by three largely different groups of analysts:  (i) probability theory; (ii) operator theory on Hilbert spaces of analytic functions, as represented by Shields' influential survey \cite{shields} and the enormous body of research it inspired up to this day; (iii) approximation of Hilbert space operators, as represented in the two-volume monographs of Apostol-Fialkow-Herrero-Voiculescu \cite{AFHV, Herr89}. The latter might be the most technical part of this paper.

  Next  the new findings of this paper are described  in more details.
As the precise statements often require considerable preparation, we decide to choose, in this introduction,  only  results which can be explained in somehow plain language. The  technical and complete theorems are deferred to later sections. Moreover, Sections \ref{S:random Hardy}-\ref{S:Aluthge} treat problems which can be  read largely independently with only a few exceptions, mostly on $C^*$-algebras. We do have, however, a unifying criteria for how we select our problems; namely, this paper seeks to answer the most fundamental questions, subject to our personal judgement, that an operator theorist might want to ask upon first hearing about  $T=T(\omega)$. Specialized topics such as the distribution of zero sets and Carleson measures, clearly important, will be taken up in the future.

  In the following the weight sequence $\{X_n\}_{n=1}^{\infty}$ is always assumed to be nonnegative, bounded i.i.d. random variables, and we shall write $T  \ \sim \{X_n\}_{n=1}^{\infty}$ when the weight sequence of $T$ is given by $\{X_n\}_{n=1}^{\infty}$. Please see Section \ref{S:Section2} for more on notations and assumptions.

   Section \ref{S:spectral} calculates the spectra of $T$. In fact,  all fine structures of the spectra of interests to us are determined (Theorem \ref{T:SpectPic}). In particular,  they  have a thick boundary except for the degenerate case, meaning that the essential spectrum of a generic sample $T(\omega)$  is  an (deterministic) annulus:
$$ \sigma_e(T)=\{z\in\bC: r\leq |z|\leq R\},$$ where $r=\min\essran X_1$ and $R=\max\essran X_1$. Here $\essran$ denotes the  essential range.
Indeed we discover that four different notions of spectral radii will  show up naturally in this study:
\begin{equation}\label{E:fourradii}
r\le r_0 \le r_p  \le R,
\end{equation}
where $r_0=e^{\mathbb{E} (\ln  X_1) }$ and $r_p=\mathbb{E} (X_1^p)$ with $p \in [1, \infty).$ This stands in sharp contrast with  deterministic operator theory and RMT and brings up rich complications in complex analysis. For instance, one may define several notions of Carleson measures resting on discs of different radii. All four are equal to one in the deterministic case. As for RMT, it is interesting to compare this phenomenon with the study of the spectral radius of random matrices, which, in case of i.i.d. entries,  is often (conjecturally) comparable to the square root of the size of the matrix \cite{Bordenave}.

  Section \ref{S:SampleClassify} classifies the samples. By this we mean that, for the random model $T=T(\omega)$, we  have indeed a family of bounded operators indexed by each $\omega \in \Omega$, hence a natural question is to study the equivalence relationship among them. We call this the classification problem. In Section \ref{S:SampleClassify} the samples $T(\omega)$'s are classified up to four equivalence relationships:
\begin{itemize}
\item similarity (Theorem \ref{T:similarity}),
\item asymptotical similarity (Theorem \ref{T:similarity}),
\item unitary equivalence (Theorem \ref{T:ApprUnitEqui}),
\item approximate unitary equivalence (Theorem \ref{T:ApprUnitEqui}).
\end{itemize}
Moreover, algebraical equivalence takes more work, and is completed only in Section \ref{S:vonNeumanIneq} (Theorem \ref{T:AlgebraEquivalent}) after much preparation. The precise statements and proofs to these classification problems are rather technical but satisfacotry.
 The answer, in short, is that they usually exhibit rigidity up to unitary equivalence, and approximate unitary equivalence if $0\notin \text{ess ran} X_1$, but homogeneity  up to  approximate unitary equivalence if $0 \in \text{ess ran} X_1$. Then rigidity up to similarity:
$$\bP^2\{(\omega_1,\omega_2): T(\omega_1)\sim T(\omega_2)\} =0,$$ when $\{X_n\}$ are positive and $r<R$,
but homogeneity up to asymptotical similarity. Indeed, under mild technical restriction, there exists a deterministic bounded operator $A$ such that
$$\bP\{\om: T(\om) \sim_a A)\}=1.$$
We further characterize such $A$'s in terms of the spectral picture (Theorem \ref{T:similarity}).

Both as an application and as a tool in the above classification problems, we solve another fundamental problem concerning $T=T(\omega)$ (Theorem \ref{T:compacts}):
\begin{thm}
Let $T\sim\{X_n\}_{n=1}^{\infty}$. Then
$$\bP\big(\KH\subset C^*(T)\big)=\bP\big(C^*(T)\cap\KH \ne\{0\}\big)=\begin{cases}
1,& 0\notin\essran X_1,\\
0,& 0\in\essran X_1.
\end{cases}$$
\end{thm}

  Here, $\KH$ is the ideal of compact operators on $\cH$ and $C^*(T)$ is the $C^*$-algebra generated by $T.$ Although the seemingly straightforward statement, the proof of the above theorem relies on some quite technical toolkit which we develop during the  classification problems.

  Section \ref{S:random Hardy} introduces the so-called random Hardy spaces associated with $T=T(\omega)$. The construction is familiar to operator theorists as laid out in \cite{shields}. Namely, each  weighted shift $T(\om)$ is unitarily  equivalent to the multiplication operator $M_z$ on a weighted Hardy space $H_{\mu}^2=H_{\mu}^2(\om),$ where $\mu$ denotes the law of $ X_1.$

  The first question on $H_{\mu}^2(\om)$ is about the  membership; namely,  what functions are in  $H_{\mu}^2(\om)$? To explain this more precisely, first of all, we observe  a zero-one law:
for any analytic function $f,$
\begin{equation*}\label{E:???}
\bP (\omega: \ f\in H_{\mu}^2(\om))\in \{0, 1\}.
\end{equation*}
This is a consequence of the Hewitt-Savage zero-one law. So we naturally introduce a deterministic function space
$$H_*=\{f(z): \ \  f\in H_{\mu}^2 \quad \text{a.s.}\}.$$
We  observe that functions in $H_{\mu}^2(\om)$ almost surely live on a disk with radius $e^{\bE(\ln X_1)}$ (Lemma \ref{T:convergence radius}).   Interesting  examples, under mild restriction and the assumption that the convergence radius of $H^2_\mu$ is almost surely one, include \begin{equation}\label{E:na} f(z)=\sum_{n=1}^{\infty}\frac{1}{n^{{\al}}}z^n \notin H_{\ast}
\end{equation} for any $\al >0$, and \begin{equation}\label{E:nna}
f(z)=\sum_{n=1}^{\infty}\frac{1}{n^{n^{\al}}}z^n \in H_{\ast}\end{equation} if and only if $\al\geq \frac{1}{2}.$ This might be rather surprising from the viewpoint of deterministic operator theory on functional Hilbert spaces.
%
%
Indeed, we have the following  (Theorem \ref{T:LIL our case} and the examples after its proof).

\begin{cor}\label{C:oneovere}
Let  $\bE(\ln X_1)=0$ and $\bE\big((\ln X_1^2)^2\big)=\sig^2 \in (0, \infty),$ and let $f(z)=\sum_{n=0}^{\infty}a_nz^n \in H_*$ be any element in $H_*$. Then
\begin{equation}\label{E:small}\sup_{f \in H_\ast} \limsup\limits_{n\to\infty}|a_n|^{\frac{\sqrt{2}}{\sig \sqrt{n\ln\ln n}}} \in [e^{-1}, 1].\end{equation}  Moreover, each $a \in  [e^{-1}, 1]$ may be achieved.
\end{cor}


  An unexpected  phenomenon about    $H_\ast$ is the existence of ``large elements". That is,  its elements admit a dichotomy  (Lemma \ref{T:norm p}), according to $$\bE\big(||f||_{H^2_\mu}\big) <\infty \quad \text{or} \quad \bE\big(||f||_{H^2_\mu}\big)=\infty.$$   The latter means that although $f \in H^2_\mu$ almost surely, its norm in mean is infinite. We call them ``large  elements'' of $H_\ast$, which exist in abundance and behave differently from small elements, i.e., $\bE\big(||f||_{H^2_\mu}\big) <\infty$. They make the problem of even how to define Carleson measures intriguing.  This leads to a natural question:

  \textbf{Problem:} Does  $H_\ast$ admit a canonical Banach space norm?

  It is clearly a vector space. The straightforward choice  of norm $||f||_{H_\ast}\doteq \bE\big(||f||_{H^2_\mu}\big) $ fails because of the existence of large elements in $H_\ast$.

  Since $T=\{T(\omega)\}_{\omega \in \Omega}$ form a new class of bounded operators, in order to better understand them, we  compare them with well known classical operators. This is achieved in the form of generalized von Neumann inequalities in Section \ref{S:vonNeumanIneq}. The comparison is performed for
  \begin{itemize}
  \item a deterministic weighted shift (Theorem \ref{T:UnilateralAUE}),
  \item a bilateral weighted shift (Theorem \ref{T:BilateralAUE}),
  \item a normal operator (Theorem \ref{T:vonNeum-Normal}).
  \end{itemize}
The results are quite satisfactory but the details are rather technical and complicated, so here we mention only two  special cases. The first is the following (Corollary \ref{C:extreme}).

\begin{cor}\label{C:01}
Let  $T\sim\{X_n\}_{n=1}^\infty$ with $\essran X_1=[0,1]$.
If $A$ is a deterministic unilateral weighted shift, then $A\lhd T$ a.s. if and only if $\|A\|\leq 1$.
\end{cor}
\noindent Here $A\lhd B$ means that
$\|p(A,A^*)\|\leq \|p(B,B^*)\|$ for all polynomials $p(x,y)$ in two free variables. Roughly speaking, Corollary \ref{C:01} implies that among the set of contractive weighted shifts, $T$ is almost surely an extreme element in some sense. It is somehow unexpected that a similar result also holds for bilateral weighted shifts (Corollary \ref{C:bi-extreme}). It is perhaps somehow surprising to point out that the conclusion of Corollary \ref{C:01} fails if the interval $[0, 1]$ is replaced by any smaller closed subset of $[0, 1]$.

The second special case is that the comparison with  the standard unilateral shift  implies that if $R=1$, then almost surely
\begin{equation}\label{OpeFunc}
\qquad ||f||_{\infty}=||f(T)||_{\BH}, \qquad \forall f \in H^\infty(\bD).
\end{equation}
This implies that an $H^\infty$-functional calculus is available. Moreover, it can be viewed as a strong form of the von Neumann inequality. Section \ref{S:vonNeumanIneq} shows that much more can be said. The generalized von Neumann inequalities there are in fact some $C^*$-terminology in disguise. Arveson's notion of  ``algebraic equivalence''  (\cite{arveson}, P. 2) plays a role.  As an application, in Section \ref{S:algebra}  we describe when the $C^*$-algebra generated by $T=T(\omega)$ is GCR or simple  (Theorem \ref{T:GCR}).

   Section \ref{S:InvSubSpa} studies the lattice of invariant subspaces of $T$, a central topic in modern non-selfadjoint operator theory. For Hilbert spaces of analytic functions,  the paradigm is  the celebrated Beurling theorem  \cite{Beurling}, which concerns the case of the standard unilateral shift, with later contribution by P. Halmos \cite{Halmos61} and P. Lax \cite{Lax}, leading to what is known as the Beurling-Halmos-Lax theorem today. This has inspired a great amount of research, say, \cite{Bergman} for Bergman and \cite{1988Richter Dirichlet} for Dirichlet.
 Among other things, we show (Theorem \ref{T:ABFP} and Proposition \ref{C:Universal})
\begin{cor}\label{T:likeBergman}
Let  $T\sim\{X_n\}_{n=1}^\infty$ with  $0<r<R=1.$
\begin{enumerate}
\item[(i)]  $T$ is almost surely reflexive;
\item[(ii)]  $T$ has almost surely invariant subspaces with
an index of any finite number or infinity;
\item[(iii)] If $A$ is a strict contraction (that is, $\|A\|<1$), then  there exist almost surely $\mathcal{N}, \mathcal{M}\in\Lat(T)$ with $\mathcal{N}\subset \mathcal{M}$ such that the compression of $T$ to ${\mathcal{M}\ominus \mathcal{N}}$ is unitarily equivalent to $A$.
\end{enumerate}
\end{cor}
\noindent
These statements are characteristics of the Bergman shift. But (\ref{E:small}) suggests that the functional Hilbert spaces for $T$ are rather small, hence in interesting constrast with (iii) above. The proof relies on the powerful theory of the so-called class {$\mathbb{A}_{\aleph_0}$}, due to  Apostol-Bercovici-Foia\c{s}-Pearcy \cite{Apostol}. Our main contribution here is to discover a new class of operators in {$\mathbb{A}_{\aleph_0}$}.  It should be mentioned that the first example of an invariant subspace with codimension two is constructed by Hedenmalm in \cite{Hed93}. Then the case of arbitrary index is achieved by Hedenmalm-Richter-Seip in \cite{Hed96}.

  Subsection \ref{S:sub Beurling-type} discusses Beurling-type results, and this is the only place in this paper where we make (small) direct contribution to deterministic operator theory. Our  goal of proving a random version of Beurling's theorem is not achieved. The obstacles are on both the operator theory side and the probability theory side. We pin down two   conjectures (\ref{Con:quadratic} and \ref{con:growth})  to illustrate the difficulty. On the operator theory side, roughly speaking, we categorize most known results in deterministic operator theory as for concave operators with slow growing moment sequences (indeed linear-like growth). But
to treat the random model, one needs knowledge on convex operators with fast growing moment sequences (indeed half-exponential growth $e^{ \sqrt{n\ln \ln n}}$). In contrast, it is natural to wonder whether one can obtain Hilbert spaces of analytic functions with rapid coefficient growth via adjusting the weights in weighted Bergman spaces. For this direction, the work of Borichev-Hedenmalm-Volberg \cite{Hed04} provides some interesting examples. As for convex operators, to the best of our knowledge, there is no result in the literature so far. Also there is no result for moment sequences beyond quadratic growth.  This is indeed related to a well-known open problem in operator theory for Beurling-type theorems for Dirichlet-like spaces with index $\al >1$ in   (\ref{E:D al}). To spur further interests, we offer an easy result  (Lemma \ref{L:convex}) for convex operators with linear growth and a  conjecture (\ref{Con:quadratic}) for more general situations.
On the probability theory side,  problems of convergence of  summation of  dependent random variables are encountered. Two of them are:
\begin{equation}\label{E:powerseries}
c_1X_1+c_2X_1X_2+\cdots+c_nX_1X_2\cdots X_n+ \cdots,
\end{equation}
and for each $m$,
\begin{equation}\label{E:powerseries2}
c_1X_1\cdots X_m+c_2X_2 \cdots X_{m+1}+\cdots+c_nX_n \cdots X_{n+m-1}+ \cdots,
\end{equation}
where $X_n$'s are i.i.d. random variables. It is well known that summation of dependent random variables can be tricky and existing knowledge in probability theory, often depending on estimates of various mixing coefficients,  does not yield satisfactory answers for us, although the first one is  subject to martingale difference techniques  and the second one is the summation of a stationary series, or even orthogonal series if $\bE( X_1)=0$. But for us, the case of interests is when $X_1$ is strictly positive, and translating $X_1$ by a constant has unmanageable consequences for (\ref{E:powerseries}) and (\ref{E:powerseries2}). On the other hand, in view of their elegance,
%
 we plan to start a separate work to sort out their theory.

    Section \ref{S:algebra}  identifies various algebras generated by $T$, including the Banach algebra, the weakly closed algebra, the commutant, the double commutant,  the dual algebra (Theorem \ref{T:GenerAlgebra}), as well as some description of the $C^*$-algebras (Theorem \ref{T:GCR}) .
\begin{thm}\label{T:alge}
Let $T\sim\{X_n\}_{n=1}^\infty$ with $R=\max\essran X_1=1$.
\begin{enumerate}
\item[(i)] $\sA_T,$ the Banach algebra generated by $T,$ is isometrically isomorphic to the disk algebra $\mathcal{A}(\bD)$  almost surely.
\item[(ii)] $\sO_T=\sW_T$ almost surely, where $\sO_T$ is the dual algebra generated by $T$ and  $\sW_T$ is the weakly closed algebra generated by $T$ almost surely.
\item[(iii)] Moreover,  if $\bP(X_1=0)=0$, then $\sO_T=\sW_T=\{T\}'=\{T\}''$ is isometrically isomorphic to $H^\infty(\bD)$  almost surely. In particular, the multiplier algebra is isometrically isomorphic to $H^\infty(\bD)$  almost surely.
\end{enumerate}
\end{thm}
\noindent A word of caution:  a reader might feel that he/she is familiar with the form of the above results  from his/her  background in deterministic operator theory, but  this impression can be misleading. Theorem \ref{T:alge} is indeed quite different in that  the convergence radius of $H^2_\mu$ above is $$e^{\mathbb{E}(\ln X_1)},$$ hence strictly less than one (except for the degenerate case) (Lemma \ref{T:convergence radius}). This means that $H^\infty(\mathbb{D})$ above is in some sense much smaller than the familiar multiplier algebra $H^\infty(\mathbb{D})$ for the Hardy or Bergman space. From the viewpoint of function theory, the natural normalization is $\mathbb{E} ( \ln X_1)=0$, so $H^2_\mu$ can be viewed as a function space over the unit disk $\mathbb{D}$.
{In this setting the random Hardy space $H^2_\mu$ is often  contained in the disk algebra (Remark \ref{C:smallspace}). If so,  then the multiplier algebra should be a subalgebra of $A(\mathbb{D})$.} As mentioned before, the interplay between function theory and the four radii in (\ref{E:fourradii}) is intriguing and many fundamental problems are awaiting to be sorted out. This paper brings out some interesting phenomena but is probably only scratching the surface along this line.

Concerning the phenomenon that the radius of convergence for the multiplier algebra is typically larger in the random setting than that of the function space itself, a curious analogue of this appears in the setting of Dirichlet series; see the work of Hedenmalm-Lindqvist-Seip \cite{Hed97}.

On the other hand,  with the aid of results obtained in Section \ref{S:SampleClassify} and Section \ref{S:vonNeumanIneq}, we take a first step toward understanding the structure of the $C^*$-algebras generated by $T$ in Section \ref{S:algebra}. We show that the  $C^*$-algebras $C^*(T)$ are simple or GCR only for  trivial cases. This implies that their structure   is a nontrivial issue.

\begin{thm}
Let  $T\sim\{X_n\}_{n=1}^\infty.$
Then
\begin{enumerate}
\item[(i)] $\bP(C^*(T)\ \textup{is simple})=\begin{cases}
1,& \essran X_1=\{0\},\\
0,& \textup{otherwise};
\end{cases}$
\item[(ii)] $\bP(T\ \textup{is GCR})=\begin{cases}
1,& \card\big(\essran X_1\setminus\{0\}\big)\leq1,\\
0,& \textup{otherwise}.
\end{cases}$
\end{enumerate}
\end{thm}

  One may wonder whether the random model exhibits more chaotic behaviors than classical weighted shifts. This is indeed the case but at least basic questions can be neatly resolved. We clarify these dynamic theoretical issues in Section \ref{S:Dynamic} (Theorem \ref{T:Dynamics}). Topics treated include:
\begin{itemize}
\item supercyclicity,
\item  hypercyclicity,
\item Li-Yorke chaoticity,
\item topologically mixing property,
\item chaoticity, and
\item frequent hypercyclicity.
\end{itemize}
The statements are usually tidy, thanks to the well-rounded deterministic theory.  In particular, among other things, we show
\begin{thm}
Let  $T\sim\{X_n\}_{n=1}^\infty$ with $X_1$ being non-degenerate and $\bP(X_1=0)=0$.
Then\begin{enumerate}
\item[(i)] $
\bP(T^* ~\textup{is Li-Yorke chaotic})=\begin{cases}
1,& R> 1,\\
0,& R\leq 1,
\end{cases}  $ where $R=\max\essran X_1$;
\item[(ii)] $\bP(T^* ~\textup{is chaotic})=\bP(T^* ~\textup{is frequently hypercyclic})
=\begin{cases}
1,& \bE(\ln X_1)> 0,\\
0,& \bE(\ln X_1)\leq 0.
\end{cases}  $
\end{enumerate}
\end{thm}

   Recall that the classical unilateral shift in operator theory is not chaotic nor Li-Yorke chaotic.

  Since the weights of  $T=T(\omega)$ oscillate randomly, we seek to regularize them until they converge to the standard unilateral shift.   This is achieved in Section \ref{S:Aluthge}  with the Aluthge transform, defined as
$$\Delta(A)=|A|^{1/2}U|A|^{1/2},$$  with polar decomposition $A=U|A|$. The transform $\Delta(A)$ usually makes an operator $A$ closer to a normal operator. The convergence of iterated Aluthge transforms has received some attention \cite{Antezana,Jung,Jung3}. We show that the iteration in our case converges in the strong operator topology to the standard unilateral shift $S$ (Theorem \ref{T:AluthgeTrans}), that is, almost surely,
$$
\Delta^n(T)\overset{\sot}{\longrightarrow} e^{\bE(\ln X_1)}S\qquad \textup{as} \quad n\rightarrow\infty.
$$
The convergence is not in norm except for the degenerate case; this is related to the four radii  in (\ref{E:fourradii}).


  \textbf{Literature review.} Before we end this introduction we briefly survey some literature where ``randomness'' meets ``non-selfadjointness''.
First, there is a sizable literature on non-selfadjoint random Schr\"{o}dinger operators, originating from the works of Hatano-Nelson on non-hermitian Anderson model in 1990s \cite{Hat1, Hat2}.
These operators are usually unbounded and eigenvalues are still a focus of study.
Second, a natural candidate for non-selfadjoint ROT is random Toeplitz operators (as opposed to random Toeplitz matrices) for which we found only two groups of works. In \cite{CurtoMuhlyXia2, CurtoMuhlyXia1} Curto-Muhly-Xia   mentioned the term but their randomness is not in the sense of probability theory; it is just an adjective \cite{XiaPersonal}. Another group of work is due to Ke-Lai-Lee-Wong \cite{Wong1, Wong2}.
Their motivation is from problems in scientific computing (Raymond Chan's work on preconditioning Toeplitz systems). In \cite{Wong1, Wong2} they set up a very general framework for random Toeplitz operators in the sense of Skorohod's random linear operators \cite{Skorohod}, but they have no concrete models in mind \cite{WongPersonal}. Indeed \cite{Skorohod}  is one of the standard references; from its citations one can recover much existing literature on  abstract random operators. Needless to say, they usually differ in theme significantly from the present paper. A related framework for random linear functional can be found in \cite{Dudley}. Another good source for existing literature on random operator theory is Saadati's book \cite{Saadati} and its citations. Say, one can find a discussion on random compact operators there. Lastly, other relevant literature includes randomly normed spaces \cite{Randomly normed spaces},
probabilistic normed spaces \cite{Probabilstic  Normed  Spaces},
and probabilistic metric spaces \cite{Probabilistic metric spaces}.

\section{Notations and assumptions}\label{S:Section2}


  In this paper we consider only nonnegative  weights. This is a standard assumption in operator theory, although there is a subtlety here. Let $A,B$ be two unilateral (or bilateral) weighted shifts with weights $\{\lambda_i\}$ and $\{\mu_i\}$ respectively. A basic fact is that if $|\lambda_i|=|\mu_i|$ for all $i$, then by \cite[Proposition 1]{shields}, $A,B$ are unitarily equivalent. For the random model, the unitary operator used to implement the equivalence is by itself a random operator, hence a priori not a trivial object. This is again a good problem for further study.

   As usual, we let $\bN,\bR,\bC$ denote, respectively, the set of positive integers,  real numbers, and  complex numbers. For $z\in\bC$ and $\eps>0$, we denote $B(z,\eps)=\{\lambda\in\bC: |\lambda-z|<\eps\}$ and $\overline{B(z,\eps)}=\{\lambda\in\bC: |\lambda-z|\leq\eps\}$. $\cH$ will always denote a complex, separable, infinitely dimensional Hilbert space endowed with the inner product $\la\cdot,\cdot\ra$. We denote by $\BH$ the algebra of all bounded linear operators on $\cH$, and by $\KH$ the ideal of compact operators.

  Throughout this paper, we let $(\Omega,\mathcal{F},\bP)$ denote a probability space, although it seldom appears explicitly.
Then $\omega \in \Omega$ denotes a sample.
We write $\bE(X)=\int_\Omega Xd\bP$, if the integral exists. We let $\essran X$ denote the {\it essential range} of $X$. Then we write
$$
r=\min\essran X_1\ \ \text{and}\ \ R=\max\essran X_1.
$$
We usually assume that $R<\infty$ since we want to refrain the discussion to bounded operators only. The case $R=\infty$ is still interesting, but not treated in this paper since the questions to ask for unbounded operators are usually very different.

  We say that $X$ is   {\it non-degenerate } if $\essran X$ contains more than one point. Then all four radii in (\ref{E:fourradii}) are distinct except for the degenerate case.

Intuitively, to obtain less randomness one might consider independent random weights with less randomness further out.
More precisely, if $\essran X_1$ gets smaller, then the random weighted shift $T$ will have less randomness. In the extreme case $\card\essran X_1=1$, $T$ becomes a deterministic weighted shift. As we shall see later (Theorem \ref{T:UnilateralAUE}), if $\essran X_1=[0,1/2]$, then there will be less weighted shifts $A$ satisfying $A\lhd T$ a.s., which means that $T$ possesses less randomness. It is the same as the general case. It is worthwhile to explore further how the randomness of the weights $\{X_n\}$ effects the behaviors of $T$.

\section{The spectral picture}\label{S:spectral}

  In this section we  describe the spectral picture of the random weighted shifts and determine the fine parts. The results are summarized below. They will be repeatedly used throughout the paper. Notations, which are usually rather standard, will be explained in subsequent subsections as we proceed to the proofs.

\begin{thm}\label{T:SpectPic}
Let $T\sim \{X_n\}_{n=1}^{\infty}$ with $R>0$. Then
\begin{enumerate}
\item[(i)] $\|T\|=\|T\|_e=\gamma(T)=w(T)=R$ a.s.
\item[(ii)] $\sigma(T)= \sigma_\delta(T)=\overline{B(0,R)}$ a.s.
\item[(iii)] $\sigma_\pi(T)=\sigma_e(T)=\sigma_{lre}(T)=\{z: r\leq |z|\leq R\}$ a.s. and
$$\ind(T-\lambda)=-1, \quad\forall \lambda\in B(0,r)\quad a.s.$$
\item[(iv)] If $\bP(X_1=0)>0$, then almost surely
$$\sigma_p(T)=\sigma_p(T^*)= \{0\},  \ \text{and} \   \dim\ker T=\dim\ker T^*=\infty.$$
\item[(v)] If $\bP(X_1=0)=0$, then almost surely
$$\sigma_p(T)= \emptyset,\quad\sigma_p(T^*)=
B(0, e^{\bE(\ln X_1)}), \ \text{and} \ \dim\ker(T^*-\lambda)=1,\quad \forall \lambda\in\sigma_p(T^*).$$
\item[(vi)] $W(T)=B(0,R)$ a.s. and $W_e(T)=\overline{B(0,R)}$ a.s.
\end{enumerate}
\end{thm}

\subsection{The norm and spectrum}

We first introduce some notations.
For $A\in\BH$, we denote by $\sigma(A)$ the spectrum of $A$, and by $\gamma(A)$ the spectral radius of $A$, that is, $\gamma(A)=\max\{|z|: z\in\sigma(A)\}$.
Given a (finite or denumerable) uniformly bounded family $\{A_\alpha\}_{\alpha\in\Lambda}$
of operators such that $A_\alpha\in\B(\cH_\alpha)$ for $\alpha\in\Lambda$, we shall denote by
$\oplus_{\alpha\in\Lambda}A_\alpha$ the direct sum of the operators $A_\alpha$ acting on the orthogonal direct sum $\oplus_{\alpha\in\Lambda}\cH_\alpha$.
If $A\in\BH$ and $\upsilon$ is a cardinality number with $1\leq \upsilon\leq\infty$, then $A^{(\upsilon)}$ will
denote the operator $\oplus_{1\leq i\leq\upsilon} A$ acting on
$\cH^{(\upsilon)}$ (orthogonal direct
sum of $\upsilon$ copies of $\cH$).
 The following fact will be used repeatedly in this paper. We record it  for convenience without a proof.

\begin{lem}\label{L:iidRV}
Let $\{X_i\}_{i=1}^n$ be bounded, i.i.d. random  variables on $(\Omega,\mathcal{F},\bP)$. Then
$\|\Pi_{i=1}^n X_i\|_\infty=(\|X_1\|_{\infty})^n.$
\end{lem}

\begin{lem}\label{T:NormSpectral}
Let $T\sim \{X_n\}_{n=1}^{\infty}.$ Then, almost surely,
  $\|T\|=R$ and
 $\sigma(T)=\overline{B(0,R)}.$

\end{lem}

\begin{proof}
 We first observe that  $\|T\|\leq R$ a.s.
For any $0\leq \delta<R$ and  $n\geq 1$,
 $ \bP(\|T\|\leq \delta)=\bP(\sup_i X_i\leq \delta)
\leq\bP\big(\cap_{1\leq i\leq n} [X_i\leq \delta] \big)
  =\big(\bP(X_1\leq \delta)\big)^n.$
So $\bP(\|T\|\leq \delta)=0$ by letting $n\to\infty.$
 In view of Theorem 4 in \cite{shields}, it suffices to prove that $\gamma(T)=R$ a.s.
 For each $n\geq 1$, observe  that $T^n=\oplus_{i=1}^n T_i$, where
 $T_i$ is a random weighted shift with weights $\{\Pi_{k=i}^{n+i-1} X_{jn+k}\}_{j=0}^\infty$ which are  \iid for  $i \ge 1$. In view of  Lemma \ref{L:iidRV},
$\|T^n\|^{1/n}=R$ a.s., hence
$\gamma(T)=R$  a.s.
\end{proof}

  For $A\in\BH$, we denote by $\sigma_{p}(A)$ the point spectrum of $A$. The kernel of $A$ and the range of $A$ are denoted, respectively, as $\ker A$ and $\ran A$.

%
%
%

\begin{lem}\label{P:PointSpect}
Let $T\sim \{X_n\}_{n=1}^{\infty}.$
\begin{enumerate}
\item[(i)] If $\bP(X_1=0)>0$, then
$\sigma_p(T)=\sigma_p(T^*)= \{0\}$ a.s. and $$\dim\ker T=\dim\ker T^*=\infty\ \ a.s.$$
\item[(ii)] If $\bP(X_1=0)=0$, then almost surely, $\sigma_p(T)= \emptyset,$ and
$$\sigma_p(T^*)=\{0\}\cup \{z\in\bC: |z|<e^{\bE(\ln X_1)}\},\quad \dim\ker(T^*-\lambda)=1,\quad \forall \lambda\in\sigma_p(T^*).$$
\end{enumerate}
\end{lem}

\begin{proof}
(i) Observe that
$\bP\big(\liminf_{n\to\infty}[X_n>0]\big)=\bP\Big(\bigcup_{n=1}^{\infty}\bigcap_{k=n}^{\infty}[X_n>0]\Big)=0.$
Thus $T$ is  almost surely a direct sum of a family of nilpotent operators.
(ii) Set $\Omega'=\cap_{n=1}^\infty[ X_n>0]$. Then $\bP(\Omega')=1$. Thus $T$ is almost surely an injective unilateral weighted shift. Then   $\sigma_p(T)= \emptyset$ follows  from
Theorem 8 in \cite{shields}.
 Also, almost surely, $0\in\sigma_p(T^*)$ and $\dim\ker T^* =1$.
Now assume that $\lambda\in\bC\setminus\{0\}$. By the circular symmetry, we  assume that $\lambda>0$.
If $x\in\ker (T^*-\lambda)$, then   $x=\alpha e_1+\sum_{n=1}^\infty \frac{\alpha\lambda^n}{X_1\cdots X_n}e_{n+1}$ for some $\alpha\in\bC$. Thus $\lambda\in\sigma_p(T^*)$ if and only if
$\sum_{n=1}^\infty\frac{\lambda^{2n}}{(X_1\cdots X_n)^2}<\infty.$
If $\lambda<e^{\bE(\ln X_1)}$, then  a straightforward   application of the law of large numbers (LLN) shows that $\sum_{n=1}^\infty\frac{\lambda^{2n}}{(X_1\cdots X_n)^2}<\infty$ a.s. Next we need only consider $\lambda= e^{\bE(\ln X_1)}.$ Set $Y_n=X_n/\lambda$.   By Theorem 4.1.2 in  \cite{Durrett},
$\liminf\limits_{n\to\infty}\big(\Pi_{i=1}^n Y_i\big)=0$ a.s.  Thus,
$\sum_{n=1}^\infty\frac{\lambda^{2n}}{(X_1\cdots X_n)^2}=\infty$  a.s.
\end{proof}

%
%
%


  Next we consider the essential spectrum, the Wolf spectrum and the approximate point spectrum of $T$ (Lemma \ref{T:SpectralPicture}). We need two lemmas (Lemma \ref{L:r1} and Lemma \ref{L:quasiDiagonal}) with independent  interests.
For $A\in\BH$, let $m(A)=\inf_{x\in\cH,\|x\|=1}\|Ax\|$ and
$r_1(A)=\lim\limits_{n\rightarrow\infty}[m(A^n)]^{1/n}$. If $A$ is a weighted shifts with weights $\{\lambda_n\}_{n=1}^\infty$, then   $m(A)=\inf_n|\lambda_n|$. Also,  $m(A\oplus B)=\min \{m(A),m(B)\}$.
Recall that  $T^n=\oplus_{i=1}^n T_i$, where
  $T_i$ is a random weighted shift with weights $\{\Pi_{k=i}^{n+i-1} X_{jn+k}\}_{j=0}^\infty$ which are \iid random variables. So   $m(T^n)=r^n$ a.s. In particular, we have

 \begin{lem}\label{L:r1}
Let $T\sim \{X_n\}_{n=1}^{\infty}.$
Then $r_1(T)=r$ a.s.
\end{lem}

   The {\it essential spectrum}  and the {\it Wolf spectrum} of $A$ are defined, respectively, as
\[\sigma_{e}(A)=\{\lambda\in\bC: A-\lambda \textup{ is not Fredholm}
\}\] and \[\sigma_{lre}(A)=\{\lambda\in\bC: A-\lambda \textup{ is not semi-Fredholm}
\}.\] The reader is referred to \cite{Herr89} for more details.
 An operator $A\in\BH$ is said to be {\it quasidiagonal} if there exists a compact operator $K$ on $\cH$ such that $A+K$ is the direct sum of at most countable finite-rank operators. By Theorem 6.15 in \cite{Herr89}, if $A$ is quasidiagonal
then $\ind(A-\lambda)=0$ for $\lambda\notin\sigma_{lre}(A)$.
Observe that almost surely,
$\liminf\limits_{n\to\infty} X_n =r$ and  $\limsup\limits_{n\to\infty}X_n=R.$
If $r=0,$ then one has $\liminf\limits_{n\to\infty} X_n =0$ a.s.  It follows that

\begin{lem}\label{L:quasiDiagonal}
Let  $T\sim \{X_n\}_{n=1}^\infty.$
If $r=0$, then $T$ is almost surely quasidiagonal.
\end{lem}

\noindent Recall that the {\it approximate point spectrum} of $A\in\BH$ is the set
\[\sigma_\pi(A)=\{\lambda\in\bC: \lambda-A~ \textup{is not bounded below}\}\]
and
\[\sigma_\delta(A)=\{\lambda\in\bC: \lambda-A~ \textup{is not surjective}\}\]
is the {\it approximate defect spectrum} of $A$.

\begin{lem}\label{T:SpectralPicture}
Let  $T\sim \{X_n\}_{n=1}^\infty.$ Then $ \sigma_\delta(T)=\overline{B(0,R)}$ a.s. and $$\sigma_\pi(T)=\sigma_{lre}(T)=\sigma_e(T)=\{z\in\bC: r\leq |z|\leq R\}\quad a.s.$$
Furthermore, if $r>0$, then, almost surely,
$\ind(T-\lambda)=-1$ for $\lambda\in B(0,r).$
\end{lem}

\begin{proof}
Clearly, $0\in\sigma_\delta(T)$. Fix   $\lambda\in\bC$ with $0<|\lambda|\leq R$. By Lemma \ref{P:PointSpect}, $T-\lambda$ is almost surely injective. On the other hand, by Lemma \ref{T:NormSpectral}, $\lambda\in\sigma(T)$ a.s. It follows that $T-\lambda$ is almost surely not surjective. Hence,   $\sigma_\delta(T)=\sigma(T)=\overline{B(0,R)}$ a.s.
 The rest of the proof is divided into two cases.

  {\it Case 1.} $r>0$.
 By Lemma \ref{L:r1} and Theorem 6 in \cite{shields}, we have $\sigma_\pi(T)=\{z: r\leq |z|\leq R\}$ a.s. Hence, for $\lambda$ with $|\lambda|<r$, $T-\lambda$ is almost surely a semi-Fredholm operator and, by the continuity of the index function, $\ind(T-\lambda)=\ind T=-1$. So
$\sigma_{lre}(T)\subset\sigma_e(T)\subset\{z\in\bC: r\leq |z|\leq R\}$ a.s.
Fix a $\lambda$ with $r\leq |\lambda|\leq R$. By Lemma \ref{P:PointSpect}, $T-\lambda$ is almost surely injective. It follows that $\ran(T-\lambda)$ is almost surely not closed, since $\lambda\in\sigma_\pi(T)$ a.s. Hence $\lambda\in\sigma_{lre}(T)$ a.s. Thus $\sigma_{lre}(T)=\sigma_e(T)=\{z\in\bC: r\leq |z|\leq R\}$ a.s.

  {\it Case 2.} $r=0$.
 This means that $0\in\essran X_1$. Since $\sigma_{lre}(T)\subset\sigma_{\pi}(T)\cap\sigma_e(T)$, it suffices to prove $\overline{B(0,R)}\subset\sigma_{lre}(T)$ a.s.
Fix a $\lambda\in\bC$ with $0<|\lambda|\leq R$.
Set $\Omega_1=[\ran(T-\lambda)\ \textup{is norm closed}],$\
$\Omega_2=[\dim\ker(T-\lambda)=0],$\
$\Omega_3=[ T~ \textup{is quasidiagonal}].$
 If $\omega\in\cap_{i=1}^3\Omega_i$, then $T(\omega)-\lambda$ is injective, quasidiagonal and a semi-Fredholm operator.
Then $\ind(T(\omega)-\lambda)=0$. Hence $T(\omega)-\lambda$ is surjective and $\lambda\notin\sigma_\delta(T(\omega))$.
Since $\lambda\in\sigma_\delta(T)$ a.s., it follows that $\cap_{i=1}^3\Omega_i$ is a null set. On the other hand, by Lemma \ref{P:PointSpect} and Lemma \ref{L:quasiDiagonal}, we have $\bP(\Omega_2)=1=\bP(\Omega_3)$. This implies that $\bP(\Omega_1)=0$. Hence $\ran(T-\lambda)\ne \overline{\ran(T-\lambda)}$ a.s. and $\lambda\in\sigma_{lre}(T)$ a.s.
Choose a countable dense subset $\Gamma$ of $B(0,R)\setminus\{0\}$. Then $\Gamma\subset\sigma_{lre}(T)$ a.s.
It follows that $\overline{B(0,R)}\subset\sigma_{lre}(T)$ a.s.
\end{proof}

  For $A\in\BH$, the {\it essential norm} of $A$ is
$\|A\|_e=\inf_{K\in\KH}\|A+K\|.$ Here $\KH$ is the subset of $\BH$ consisting of all compact operators. Since $\|T\|_e\leq \|T\|$, and
$\|T\|_e\geq\max_{z\in\sigma_e(T)}|z|=\max_{z\in\sigma(T)}|z|=\|T\|$  a.s., we have

\begin{cor}
Let  $T\sim \{X_n\}_{n=1}^\infty.$ Then $\|T\|_e=R$ a.s.
\end{cor}


\subsection{The numerical range}
For $A\in\BH$, the {\it numerical range}   is the set
\[W(A)=\{\la Ax,x\ra: x\in\cH ~\textup{with}~ \|x\|=1\}.\]
 The  Toeplitz-Hausdorff theorem asserts that the numerical range of an operator is always convex \cite{Hausdorff,Toeplitz}. The convex hull of $\sigma(A)$ lies in the closure
of $W(A)$. The {\it numerical radius} of $A$ is
$w(A)=\sup_{z\in W(A)}|z|$. It is  known that $\gamma(A)\leq w(A)\leq \|A\|$ \cite[Problem 214]{HilbertProblem}. The {\it essential numerical range} of $A$ is the compact set
\[W_e(A)=\bigcap_{K\in\KH}\overline{W(A+K)}.\] It always holds that $\sigma_e(A)\subset W_e(A)\subset \overline{W(A)}$  \cite{Lancaster}. The reader is referred to \cite[Chapter 22]{HilbertProblem} for more   facts about the numerical range.

\begin{lem}\label{P:numericalRange}
Let  $T\sim \{X_n\}_{n=1}^\infty$ with $R>0$.   Then
$$W(T)=B(0,R)\   a.s. \ \ \ \text{and }\ \ \ W_e(T)=\overline{B(0,R)} \ a.s.$$
\end{lem}

\begin{proof}
Since $\gamma(T)\leq w(T)\leq \|T\|$, it follows from Lemma \ref{T:NormSpectral} that $w(T)=\|T\|=R$ a.s. Since $W(T)$ possesses the circular symmetry, it suffices to show $R\notin W(T)$ a.s.
By the discussion above,   $\|T\|=w(T)=R$ and if   $R\in W(T)$, then $R\in\sigma_p(T)$ a.s.
By Lemma \ref{P:PointSpect},  $R\notin W(T)$ a.s.
 By Lemma \ref{T:SpectralPicture}, $\overline{B(0,R)}$ $\subset W_e(T)$ a.s., hence   $W_e(T)=\overline{B(0,R)}$ a.s.
\end{proof}

\begin{rmk}
By Lemma \ref{T:NormSpectral} and Lemma \ref{P:numericalRange}, a random weighted shift is a.s. normaloid. An operator $A\in\BH$ is {\it normaloid} if $\|A\|=w(A)$
 (\cite{HilbertProblem}, P. 117).
\end{rmk}

\section{Classification of samples}\label{S:SampleClassify}

  This section is devoted to the classification of samples up to four typical equivalence relationships of Hilbert space operators, including similarity, asymptotical similarity, unitary equivalence, and approximate unitary equivalence.

\subsection{Similarity and asymptotical similarity}

Recall that two operators $A,B\in\BH$ are {\it similar}, denoted by $A\sim B$, if there exists an invertible $V\in\BH$ such that $AV=VB$. Here, the $``\sim"$ should be not confused with that occurring in $T\sim \{X_n\}_{n=1}^{\infty}.$
For $A\in\BH$, the {\it similarity orbit} of $A$ is defined to be the set $\mathcal{S}(A)=\{Z\in\BH: Z\sim A\}$.
Two operators $A,B\in\BH$ are said to be {\it asymptotically similar} if $\overline{\mathcal{S}(A)}=\overline{\mathcal{S}(B)}$, denoted by  $A\sim_a B$. The reader is referred to \cite{AFHV,Herr89} for more about similarity and asymptotical similarity of Hilbert space operators.

  The main result of this subsection is the following.
\begin{thm}\label{T:similarity}
Let  $T\sim \{X_n\}_{n=1}^\infty.$
\begin{enumerate}
\item[(i)] If $X_1$ is non-degenerate and $\bP(X_1=0)=0$, then $$\bP^2\big\{(\omega_1,\omega_2)\in\Omega^2: T(\omega_1)\sim T(\omega_2)\big\}=0;$$ indeed,
$$\bP(T\sim A)=0$$ for any deterministic operator $A.$
\item[(ii)] We have $$\bP^2\big\{(\omega_1,\omega_2)\in\Omega^2: T(\omega_1)\sim_a T(\omega_2)\big\}=1;$$ indeed, there exists a deterministic operator $A$ such that $$\bP(T\sim_a A)=1.$$
\item[(iii)] If $R>0$, then an operator $A$ satisfies $\bP(T\sim_a A)=1$ if and only if
\begin{itemize}
\item[(a)] $\sigma(A)=\overline{B(0,R)}$,
\item[(b)] $\sigma_{lre}(A)=\overline{B(0,R)}\setminus B(0,r)$ and
\item[(c)] $-\ind(A-\lambda)=1=\dim\ker (A-\lambda)^*\ \text{for all} \ \lambda\in B(0,r).$
\end{itemize}
\end{enumerate}
\end{thm}

  The following lemma is needed in the proof of  Theorem \ref{T:similarity}.

\begin{lem}\label{L:constantCase}
Let $\{X_n\}_{n=1}^\infty$ be i.i.d. random variables and $\lambda_i \in \essran X_1,\ i=1,\cdots,k$. If $\{n_j\}_{j=1}^\infty$ is a subsequence of $\bN$, then
$$\liminf_{j\to\infty} \big(\max_{1\leq i\leq k}|X_{n_j+i}-\lambda_i|\big)=0\ \ \ a.s.$$
\end{lem}


\begin{proof}
For $n\geq 1$, set $Y_n=\max\limits_{1\leq i\leq k}|X_{n+i}-\lambda_i|$. Then $\{Y_n\}_{n=1}^{\infty}$ are random variables  with the same distribution. By passing to a subsequence if necessary, we assume $n_j+k<n_{j+1}$ for  $j\geq 1,$
so $\{Y_{n_j}\}_{i=1}^\infty$ are \iid
For  $\eps>0$, note that
$\liminf\limits_{j\to\infty} [Y_{n_j}>\ep]=\bigcup_{N=1}^{\infty}\bigcap_{j=N}^{\infty}[Y_{n_j}>\ep].$ On the other hand,
since $\lambda_i \in \essran X_1$, we have $$\bP\left( Y_{1} <\eps \right)=\bP\Big(  \max_{1\leq i\leq k}|X_{i}-\lambda_i| <\eps\Big)=\prod_{1\leq i\leq k}\bP\left(|X_{i}-\lambda_i| <\eps\right)>0.$$
It follows   that
$\bP\big(\liminf\limits_{j\to\infty}[Y_{n_j}>\ep]\big)=0$ and the proof is complete.
\end{proof}

\begin{proof}[Proof of Theorem \ref{T:similarity}]
(i)  Since $\bP(X_1=0)=0$, we  set $Y_n=\ln X_n$.
Choose an at most countable dense subset $\{\lambda_j: j\in\Lambda\}$ of $\essran Y_1$.
Set \begin{equation}\label{(3.1)}
\Omega'=\bigcap_{s=1}^{\infty}\bigcap_{(j_1,\cdots,j_s)\subset\Lambda}\big[\liminf_{n\to\infty}\max_{1\leq i\leq s}|Y_{n+i}-\lambda_{j_i}| =0\big].
\end{equation}
By Lemma \ref{L:constantCase}, $\bP(\Omega')=1$.
 Fixing   $\omega_0\in \Omega'$,  we shall prove that $T\nsim T(\omega_0)$ a.s.
 Denote $\delta_1=\max\essran Y_1$. Since $Y_1$ is  non-degenerate,  we  choose $\delta_2 $ with $\delta_2<\delta_1$ such that $\bP(Y_1<\delta_2)>0$.
 For $i\geq 1$, let $r_i=Y_i(\omega_0)$. In view of (\ref{(3.1)}), we  choose $\{n_k\}_{k=1}^\infty$ with $n_k+2^k<n_{k+1}$ such that
\begin{equation}\label{(3.2)}r_{n_k+i}\leq \delta_2, \ \ \ \forall\ k\geq1,\ 1\leq i\leq 2^k.\end{equation}

\noindent  For each $s\geq 1$, set $\Omega_s=\big[\liminf\limits_{k\to\infty}\max\limits_{1\leq i\leq s}|Y_{n_k+i}-\delta_1| =0\big].$
Let $\Omega''=\cap_{s=1}^{\infty}\Omega_s$. Note that $\{Y_{n}\}_{n\geq 1}$ are \iid random variables and $\delta_1\in\essran Y_1$. By Lemma \ref{L:constantCase}, we have $\bP(\Omega'')=1$.
Now we show that $T(\omega)\nsim T(\omega_0)$ for any $\omega \in \Omega''$. Note that $T(\omega)$ and $T(\omega_0)$ are both injective. By Problem 90 in \cite{HilbertProblem}, it suffices to check that
either $\sup_{n\geq 1} \prod_{i=1}^n\frac{ X_i(\omega)}{ X_i(\omega_0)} =\infty$ or $\inf_{n\geq 1} \prod_{i=1}^n\frac{ X_i(\omega)}{ X_i(\omega_0)} =0$.
For each $k\geq 1$, there exists $i_0\geq 1$ with $2^{i_0}>k$ such that  $\max\limits_{1\leq j\leq k}|Y_{n_{i_0}+j}(\omega)-\delta_1|<\frac{\delta_1-\delta_2}{2}.$ It follows that $Y_{n_{i_0}+j}(\omega)>\frac{\delta_1+\delta_2}{2}$ for $1\leq j\leq k$. Then in view of (\ref{(3.2)}), we get
 \begin{eqnarray*}
\Big|\sum_{j=1}^{n_{i_0}+k} (Y_{j}(\omega)-r_j)\Big|+\Big|\sum_{j=1}^{n_{i_0}} (Y_{j}(\omega)-r_j)\Big|
&\geq&\Big|\sum_{j=n_{i_0}+1}^{n_{i_0}+k} (Y_{j}(\omega)-r_j)\Big|\\
 &>& \sum_{j=n_{i_0}+1}^{n_{i_0}+k} \big(\frac{\delta_1+\delta_2}{2}-\delta_2\big)=\frac{k(\delta_1-\delta_2)}{2}.
 \end{eqnarray*}
Then $\sup_{n\geq 1}|\sum_{i=1}^n (Y_i(\omega)-r_i)|\geq \frac{k(\delta_2-\delta_1)}{4}$. So $\sup_{n\geq 1}|\sum_{i=1}^n (Y_i(\omega)-r_i)|=\infty$.
Note that
$\sum_{i=1}^n (Y_i(\omega)-r_i)=\ln\Big(\frac{\prod_{i=1}^n X_i(\omega)}{\prod_{i=1}^n X_i(\omega_0)}\Big)$
and
$\exp\big(\sum_{i=1}^n (Y_i(\omega)-r_i)\big)= \prod_{i=1}^n\frac{ X_i(\omega)}{ X_i(\omega_0)}.$
So we have either $\sup_n\prod_{i=1}^n\frac{ X_i(\omega)}{ X_i(\omega_0)}=\infty$ or $\inf_n\prod_{i=1}^n\frac{ X_i(\omega)}{ X_i(\omega_0)}=0.$

  (ii) We assume that $X_1$ is non-degenerate, hence  $R>r$.
In view of Theorem \ref{T:SpectPic}, almost surely,
\begin{equation}\label{E:B}
\sigma(T)=\overline{B(0,R)},\ \ \ \sigma_{lre}(T)=\overline{B(0,R)}\setminus B(0,r)
\end{equation}
 and
 \begin{equation}\label{E:ind}
  -\ind(T-\lambda)=1=\dim\ker (T-\lambda)^*, \quad   \lambda\in B(0,r).
  \end{equation}
We need only construct an operator $A$ which has the spectral property as in (\ref{E:B}) and (\ref{E:ind}).
Then the so-called Similarity Orbit Theorem (\cite{AFHV}, Theorem 9.1) implies that  $T\sim_a A$ a.s.
If $r=0$, then let $A$ be a normal operator with $\sigma(A)=\overline{B(0,R)}$. It follows that  $\sigma_{lre}(A)=\overline{B(0,R)}$ and $A$ satisfies the requirements.
 If $r>0$, then choose a normal operator $D$ with $\sigma(D)=\overline{B(0,R)}\setminus B(0,r)$ and set $A=D\oplus S$, where $S$ is the unilateral shift.
It is easy to verify that $A$ satisfies the desired properties.

  (iii) This follows from a careful examination of the proof of (ii).
\end{proof}

\begin{rmk}
Let  $T\sim \{X_n\}_{n=1}^\infty.$
If $\bP(X_1=0)>0$, then the classification of samples by similarity is somewhat complicated. In the case that $X_1$ is discrete, as we shall see in Theorem \ref{T:ApprUnitEqui} (i), there exists $A$ such that $T\cong A$ a.s., which implies that
$T\sim A$ a.s.
\end{rmk}

\begin{ques}
 If $X_1$ is not discrete, then does it follow that $\bP(T\sim A)=0$ for any operator $A$?
\end{ques}

\subsection{Unitary equivalence and approximate unitary equivalence}

Given two operators $A$ and $B$, we use $A\cong B$ to denote that $A,B$ are unitarily equivalent, that is, $AV=VB$ for some unitary operator $V$.
$A$, $B$ are {\it approximately unitarily equivalent} (denoted $A\cong_a B$) if there exist a sequence of unitary operators $U_n$ such that $\lim_nU_n^*AU_n=B$ (\cite {Herr89}, Chapter 4). Given a $C^*$-algebra $\cA$ and two representations $\rho_1,\rho_2$ of $\cA$ on $\cH_1$ and $\cH_2$ respectively, if there exists a sequence of unitary operators $U_n:\cH_2\rightarrow\cH_1$ such that $\lim\limits_{n\to\infty}U_n^*\rho_1(Z)U_n=\rho_2(Z)$ for  $Z\in \cA$, then $\rho_1$ and $\rho_2$ are said to be {\it approximately unitarily equivalent}, denoted again as $\rho_1\cong_a\rho_2$.
 The main results of this subsection are summarized in the following.

\begin{thm}\label{T:ApprUnitEqui}
Let $T\sim\{X_n\}_{n=1}^{\infty}$. Then
\begin{enumerate}
\item[(i)] $\bP^2\big\{(\omega_1,\omega_2)\in\Omega^2: T(\omega_1)\cong T(\omega_2)\big\}=\begin{cases}
1,& \textup{if}\ X_1 \ \textup{is discrete and } \\
&\bP(X_1=0)>0,\\
1,& \textup{if} \ X_1\  \textup{is degenerate},\\
0,& \textup{otherwise}.
\end{cases}$
\item[(ii)] If $X_1$ is non-degenerate and $0\notin\essran X_1$, then 
$\bP(T\cong_a A)=0$ for any operator $A$.
\item[(iii)] If $0\in\essran X_1$, then there exists a unilateral weighted shift $W$ such that $T\cong_a W$ a.s.
\item[(iv)] If $0\in\essran X_1$, then an operator $A$ satisfies $T\cong_a A$ a.s. if and only if \begin{itemize}
\item[(a)] $T\approx A$ a.s.
and
\item[(b)] $C^*(A)$ contains no nonzero compact operators.
\end{itemize}
\end{enumerate}
\end{thm}

  In the above (iv), $\approx$ denotes the algebraical equivalence.
Two operators $A,B$ on Hilbert spaces are said to be {\it algebraically equivalent} (\cite{arveson}, P. 2) if there exists a $*$-isomorphism (that is, an isometric $*$-preserving isomorphism) from $C^*(A)$ onto $C^*(B)$ which carries
$A$ into $B$. Observe that  $A\cong_a B$ implies $A\approx B$. For $T$ a random weighted shift, more preparation in Section \ref{S:vonNeumanIneq} is needed before we can give a classification of the samples up to algebraical equivalence (Theorem \ref{T:AlgebraEquivalent}).



  Before we   give the proof of  Theorem    \ref{T:ApprUnitEqui}, we point out that its proof will  help us solve another fundamental problem: Describe the   compact operators in the $C^*$-algebra $C^*(T)$. That is, the following

\begin{thm}\label{T:compacts}
Let $T\sim\{X_n\}_{n=1}^{\infty}$. Then
$$\bP\big(\KH\subset C^*(T)\big)=\bP\big(C^*(T)\cap\KH \ne\{0\}\big)=\begin{cases}
1,& 0\notin\essran X_1,\\
0,& 0\in\essran X_1.
\end{cases}$$
\end{thm}

  The following lemma is not only needed in the proof of Theorem \ref{T:ApprUnitEqui}, but also has  significance in operator theory. Given $(\lambda_i)_{i=1}^n,(\mu_i)_{i=1}^n\in\bC^n$, we define
\[\|(\lambda_i)_{i=1}^n-(\mu_i)_{i=1}^n\|_{\max}=\max_{1\leq i\leq n}|\lambda_i-\mu_i|.\]

\begin{lem}\label{L:WSAppUniEqu}
Let $A$ be a unilateral weighted shift with nonnegative weights $\{\lambda_n\}_{n=1}^\infty$ and $B$ be a unilateral weighted shift with nonnegative weights $\{\mu_n\}_{n=1}^\infty$. Set $\lambda_0=\mu_0=0$. If for each $k\geq 1$
\[\liminf_{n\to\infty}\|(\lambda_i)_{i=n}^{n+k}-(\mu_i)_{i=0}^{k}\|_{\max}=0
=\liminf_{n\to\infty}\|(\mu_i)_{i=n}^{n+k}-(\lambda_i)_{i=0}^{k}\|_{\max},\]%
then $A\cong_a B$.
\end{lem}

\begin{proof}
  Assume that $\{e_i\}_{i=1}^\infty$ is an orthonormal basis of $\cH$, $\{f_i\}_{i=1}^\infty$ is an orthonormal basis of $\cK$ and
$Ae_i=\lambda_i e_{i+1},\quad Bf_i=\mu_i f_{i+1},$ for  $i\geq 1$.
Set $m_0=n_0=0$. For convenience, for nonnegative integers $i,j$ with $0\leq i<j<\infty$, we write
$\eta_{(i,j)}=(\lambda_i,\lambda_{i+1},\cdots,\lambda_j), \quad\xi_{(i,j)}=(\mu_i,\mu_{i+1},\cdots,\mu_j).$
By the hypothesis, for any $t>s\geq 0$, we have
\begin{equation}\label{(3.3)}
\liminf_{n\to\infty}\|\xi_{(s, t)}-\eta_{(s+n, t+n)}\|_{\max}=0
\end{equation}  and
\begin{equation}\label{(3.4)}\liminf_{n\to\infty}\|\eta_{(s, t)}-\xi_{(s+n, t+n)}\|_{\max}=0. \end{equation}

  Fix   $\eps>0$. In the following we shall choose two strictly increasing sequences $\{m_i\}_{i=1}^\infty$ and $\{n_i\}_{i=1}^\infty$ such that
\begin{enumerate}
\item[(a)] $ m_1-m_0=n_2-n_1$ and
$$m_{2k+2}-m_{2k+1}=n_{2k+4}-n_{2k+3}, \quad m_{2k+3}-m_{2k+2}=n_{2k+1}-n_{2k}, \quad \forall k\geq 0; $$
\item[(b)] $\lambda_{m_1}<\frac{\eps}{8}$, $\|\eta_{(m_0,m_1)}-\xi_{(n_1,n_2)}\|_{\max}<\frac{\eps}{8}$ and
$$\|\eta_{(m_{2k+1},m_{2k+2})}-\xi_{(n_{2k+3},n_{2k+4})}\|_{\max}<\frac{\eps}{4^k},\quad \forall k\geq 0,$$
$$\|\xi_{(n_{2k},n_{2k+1})}-\eta_{(m_{2k+2},m_{2k+3})}\|_{\max}<\frac{\eps}{4^k},\quad \forall k\geq 0.$$
\end{enumerate}

  {\it Step 1.} The choice of $m_1, n_1$ and $n_2$.
By the hypothesis, $\liminf\limits_{n\to\infty}\lambda_n=0$. Then there exists $m_1>1$ such that $\lambda_{m_1}<\frac{\eps}{8}.$
Applying (\ref{(3.4)}) to $\eta_{(m_0,m_1)}$, we  choose $n_1,n_2$ with $n_2>n_1> 1$ and $n_2-n_1=m_1-m_0$ such that
\begin{equation}\label{(3.6)}\|\eta_{(m_0,m_1)}-\xi_{(n_1,n_2)}\|_{\max}< \frac{\eps}{8}.\end{equation}

  {\it Step 2.} The choice of $m_2$ and $m_3$.
Applying (\ref{(3.3)}) to $\xi_{(n_0,n_1)}$, we  choose $m_2,m_3$ with $m_3> m_2>m_1$ and $m_3-m_2=n_1-n_0$ such that
$\|\xi_{(n_0,n_1)}-\eta_{(m_2,m_3)}\|_{\max}< \frac{\eps}{4^0}.$

  {\it Step 3.} The choice of $n_3$ and $n_4$.
Applying (\ref{(3.4)}) to $\eta_{(m_1,m_2)}$, we  choose $n_3,n_4$ with $n_4>n_3> n_2$ and $n_4-n_3=m_2-m_1$ such that
$\|\eta_{(m_1,m_2)}-\xi_{(n_3,n_4)}\|_{\max}< \frac{\eps}{4^0}.$

  {\it Step 4.}  The choice of $m_4$ and $m_5$.
Applying (\ref{(3.3)}) to $\xi_{(n_2,n_3)}$, we  choose $m_4,m_5$ with $m_5> m_4>m_3 $ and $m_5-m_4=n_3-n_2$ such that
$\|\xi_{(n_2,n_3)}-\eta_{(m_4,m_5)}\|_{\max}< \frac{\eps}{4}.$

  {\it Step 5.}  The choice of $n_5$ and $n_6$.
 Applying (\ref{(3.4)}) to $\eta_{(m_3,m_4)}$, we  choose $n_5,n_6$ with $n_6>n_5> n_4$ and $n_6-n_5=m_4-m_3$ such that
$\|\eta_{(m_3,m_4)}-\xi_{(n_5,n_6)}\|_{\max}< \frac{\eps}{4}.$

  We proceed as above and  choose recursively strictly increasing sequences $\{m_i\}_{i=1}^\infty$ and $\{n_i\}_{i=1}^\infty$ satisfying (a) and (b).

  Define $K_1\in\BH$ and $K_1'\in\B(\cK)$ as
\begin{equation*}K_1=-\sum_{i=1}^\infty \lambda_{m_i}e_{m_i+1}\otimes e_{m_i},\quad K_2=-\sum_{i=1}^\infty \mu_{n_i}f_{n_i+1}\otimes f_{n_i}.
\end{equation*}

  {\it Claim.} $\|K_1\|<4\eps$ and $\|K_2\|<4\eps$.

  In view of (\ref{(3.6)}) and (b), we have $\lambda_{m_1}<\frac{\eps}{8}$,
\[ \lambda_{m_2}<\mu_{n_0}+\frac{\eps}{4^0}=\eps,\quad
\mu_{n_1}<\lambda_{m_0}+\frac{\eps}{8}=\frac{\eps}{8},\quad\lambda_{m_3}<\mu_{n_1}+\frac{\eps}{4^0}<2\eps \] and
\[\mu_{n_2}<\lambda_{m_1}+\frac{\eps}{8}<\frac{\eps}{4},\quad\mu_{n_3}<
\lambda_{m_1}+\frac{\eps}{4^0}<2\eps, \quad\lambda_{m_4}<\mu_{n_2}+\frac{\eps}{4}<\frac{\eps}{2}; \] moreover,
\[|\lambda_{m_{2k+2}}-\mu_{n_{2k+4}}|<\frac{\eps}{4^k},\quad |\mu_{n_{2k}}-\lambda_{m_{2k+2}}|<\frac{\eps}{4^k},\quad\forall k\geq 0 \] and
\[|\lambda_{m_{2k+1}}-\mu_{n_{2k+3}}|<\frac{\eps}{4^k},\quad |\mu_{n_{2k+1}}-\lambda_{m_{2k+3}}|<\frac{\eps}{4^k} ,\quad\forall k\geq 0.\]
It follows that
\[ |\mu_{n_{2k}}-\mu_{n_{2k+4}}|\leq|\mu_{n_{2k}}-\lambda_{m_{2k+2}}|+|\lambda_{m_{2k+2}}-\mu_{n_{2k+4}}| <\frac{2\eps}{4^k},\quad \forall k\geq 0.\]
Likewise, we have for each $k\geq 0$ that
\[ |\mu_{n_{2k+1}}-\mu_{n_{2k+5}}|<\frac{2\eps}{4^k},\quad |\lambda_{m_{2k+2}}-\lambda_{m_{2k+6}}|<\frac{2\eps}{4^k},\quad |\lambda_{m_{2k+1}}-\lambda_{m_{2k+5}}|<\frac{2\eps}{4^k}.\]
Thus we conclude that
$\sup_{i\geq 1} \lambda_{m_i} <4\eps$ and $\sup_{i\geq 1} \mu_{n_i} <4\eps.$
Then $\|K_1\|<4\eps$, $\|K_2\|<4\eps$, $A+K_1$ is a unilateral weighted shift with weights $\{\lambda_k'\}$ and $B+K_2$ is a unilateral weighted shift with weights $\{\mu_k'\}$, where
\[\lambda_k'=\begin{cases}
 \lambda_k,& k\notin\{m_i:i\geq 1\},\\
 0,&  k\in\{m_i:i\geq 1\},
\end{cases}\quad \mu_k'=\begin{cases}
 \mu_k,& k\notin\{n_i:i\geq 1\},\\
 0,&  k\in\{n_i:i\geq 1\}.
\end{cases}\]
Denote $\cH_0=\vee\{e_i: m_0<i\leq m_1\}$, $\cH_0'=\vee\{f_i: n_1< i\leq n_2\}$ and
\[\cH_k=\vee\{e_i: m_{2k-1}< i\leq m_{2k}\} , \quad \cH_k'=\vee\{f_i: n_{2k+1}< i\leq n_{2k+2}\},\quad\forall k\geq 1,\]
\[\cK_k'=\vee\{e_i: m_{2k+2}< i\leq m_{2k+3}\}, \quad \cK_k=\vee\{f_i: n_{2k}< i\leq n_{2k+1}\},\quad\forall k\geq 0, \] where $\vee$ denotes closed linear span.
Then
\begin{enumerate}
\item[(c)]  $\dim\cH_k=\dim\cH_k'$, $\dim\cK_k=\dim\cK_k'$ for each $k\geq 0$
\item[(d)]  both $\cH_k$ and $\cK_k'$ reduce $\widetilde{A}:=A+K_1$ for each $k\geq 0$,
\item[(e)]   both $\cH_k'$ and $\cK_k$ reduce $\widetilde{B}:=B+K_2$ for each $k\geq 0$,
\item[(f)] $$A_0:=\widetilde{A}|_{\cH_0}=\sum_{i=1}^{m_1-1} \lambda_i e_{i+1}\otimes e_i, \quad A_0':=\widetilde{B}|_{\cH_0'}=\sum_{i=n_1+1}^{n_2-1} \mu_i f_{i+1}\otimes f_i,$$
$$A_k:=\widetilde{A}|_{\cH_k}=\sum_{i=m_{2k-1}+1}^{m_{2k}-1} \lambda_i e_{i+1}\otimes e_i, \quad A_k':=\widetilde{B}|_{\cH_k'}=\sum_{i=n_{2k+1}+1}^{n_{2k+2}-1} \mu_i f_{i+1}\otimes f_i,\quad\forall k\geq 1.$$
$$B_k':=\widetilde{A}|_{\cK_k'}=\sum_{i=m_{2k+2}+1}^{m_{2k+3}-1} \lambda_i e_{i+1}\otimes e_i,\quad B_k:=\widetilde{B}|_{\cK_k}=\sum_{i=n_{2k}+1}^{n_{2k+1}-1} \mu_i f_{i+1}\otimes f_i, \quad\forall k\geq 0.$$
\end{enumerate}
Thus, $$A+K_1=(\oplus_{i=0}^\infty A_i) \oplus(\oplus_{i=0}^\infty B_i'),\ \ \ \
B+K_2=(\oplus_{i=0}^\infty B_i) \oplus(\oplus_{i=0}^\infty A_i').$$

  By (b), we have $\|(\lambda_1,\cdots, \lambda_{m_1-1})-(\xi_{n_1+1},\cdots, \xi_{n_2-1})\|_{\max}<\eps/4.$
We can find $E_0$ on $\cH_0'$ with $\|E_0\|<\eps/8$ such that $A_0'+E_0\cong A_0$.
Likewise, using (b) again, for each $k\geq 1$, we can find $E_k$ on $\cH_k'$ with   $\|E_k\|<\frac{\eps}{4}$ so that $A_k'+E_k\cong A_k$.
Furthermore, for each $k\geq 0$, we can find $F_k$ on $\cK_k'$ with   $\|F_k\|<\frac{\eps}{4}$ so that $B_k'+F_k\cong B_k$. It follows that we can find an operator $K_3$ on $\cH$ and  $K_4$ on $\cK$ with $\|K_3\|<\frac{\eps}{4}$ and $\|K_4\|<\frac{\eps}{4}$ such that
$A+K_1+K_3\cong B+K_2+K_4.$
Note that $\|K_1+K_3\|<5\eps$ and $\|K_2+K_4\|<5\eps$.
It follows that there exists $K$ on $\cH$ with $\|K\|<10\eps$ such that $A+K\cong B$.
Since $\eps>0$ is arbitrary, we deduce that $A\cong_a B$.
\end{proof}

%

  For $e,f\in\cH$, let $e\otimes f$ denote the operator on $\cH$ defined by
$(e\otimes f)(x)=\la x,f\ra e$ for $x\in\cH.$ If $Z\in\B(\bC^n)$ and
$Z=\sum_{i=1}^{n-1}\lambda_i e_{i+1}\otimes e_i,$ where
$\{e_i\}_{i=1}^n$ is an orthonormal basis of $\bC^n$ and $\lambda_i\in\bC$ for $1\leq i\leq n-1$, then $Z$ is called a {\it truncated weighted shift} with weights $\{\lambda_i\}_{i=1}^{n-1}$ of order $n$. In the degenerate case that $n=1$, the truncated weighted shift of order $1$ means the zero operator on $\bC$.



\begin{proof}[Proof of Theorem \ref{T:ApprUnitEqui} ]

  (i) The proof  is divided into four cases.

  {\it Case 1.} $X_1$ is degenerate. Then $T\cong\lambda S$ a.s. for some $\lambda \in \mathbb{R}$, where $S$ is the unilateral shift. This case is trivial.

  {\it Case 2.} $ X_1$ is non-degenerate and $\bP(X_1=0)=0$. The conclusion follows from  Theorem \ref{T:similarity} (i).

  {\it Case 3.} $ X_1$ is non-degenerate, $\bP(X_1=0)>0$ and $X_1$ is discrete.
 Note that $\Gamma=\{a\in\bR: \bP(X_1=a)>0\}$ is at most countable.
For convenience, we  assume that $\Gamma=\{a_i:i=0,1,2,3,\cdots\}$ and $a_0=0$. Thus $\sum_{i=0}^\infty\bP(X_1=a_i)=1$.
We  choose $\Omega'\subset\Omega$ with $\bP(\Omega')=1$ such that $X_n(\omega)\in\{a_i: i\geq 0\}$ for  $n\geq 1$ and  $\omega\in\Omega'$. Then $T(\omega)$ is a unilateral weighted shift with weights in $\{a_i: i\geq 0\}$ for each $\omega\in\Omega'$.
 Suppose that $\eta_1,\eta_2,\eta_3,\cdots$ are all finite tuples with entries in $\{a_i: i\geq 1\}$. For each $i\geq 1$, assume that $\eta_i=(b_{i,1},\cdots, b_{i,n_i})$ and $A_i$ is a truncated weighted shift with weights $(b_{i,1},\cdots, b_{i,n_i})$. Set $A=0^{(\infty)}\oplus\Big(\bigoplus_{j=1}^\infty A_j^{(\infty)}\Big)$.
We shall prove that $T\cong A$ a.s.

  {\it Claim.} For each $k\geq 1$ and $k$-tuple $(i_1,i_2,\cdots,i_k)$ with entries in $\bN$,
$$\bP\Big( \limsup_{n\to\infty}\big[(X_{n},X_{n+1},\cdots, X_{n+k+1})=(0, a_{i_1},\cdots, a_{i_k}, 0)\big]\Big)=1.$$

  For each $n\geq 1$, set $E_n=[(X_{n},X_{n+1},\cdots, X_{n+k+1})=(0, a_{i_1},\cdots, a_{i_k}, 0)]$.
 For each $s\geq 0$, set $n_s=s(k+2)$. Then $\{E_{n_s}\}_{s=1}^\infty$ is a subsequence of $\{E_n\}_{n=1}^\infty$ and $\limsup\limits_{s\to\infty} E_{n_s}\subset\limsup\limits_{n\to\infty} E_n$. Note that $\{E_{n_s}\}_{s=1}^\infty$ are independent and, for each $s\geq 1$,
\begin{align*}
\bP(E_{n_s})&=\bP\Big((X_{s(k+2)+j})^{k+1}_{j=0}=(0, a_{i_1},\cdots, a_{i_k}, 0)\Big)\\
&=\bP\Big(X_{s(k+2)}=0\Big)\cdot\left(\prod_{j=1}^{k}\bP\Big(X_{s(k+2)+j}=a_{i_j}\Big)\right)\cdot\bP\Big(X_{s(k+2)+k+1}=0\Big)\\
&=\bP\Big(X_1=0\Big)\cdot\left(\prod_{j=1}^{k}\bP\Big(X_1=a_{i_j}\Big)\right)\cdot\bP\Big(X_{1}=0\Big)>0.\end{align*}
By Borel-Cantelli, we have $1=\bP\big(\limsup\limits_{s\to\infty} E_{n_s}\big)\leq\bP\big(\limsup\limits_{n\to\infty} E_n\big)$. This proves the claim.

  Similarly, one can prove
$\bP\big( \limsup\limits_{n\to\infty}[(X_{n},X_{n+1})=(0, 0) ]\big)=1.$ Then there exists $\Omega''\subset\Omega'$ with $\bP(\Omega'')=1$ such that, for each $\omega\in\Omega''$,  $k\geq 0$ and  $k$-tuple $(i_1,\cdots,i_k)$,
$$\bP\Big( \limsup\limits_{n\to\infty}\big[(X_{n},X_{n+1},\cdots, X_{n+k+1})=(0, a_{i_1},\cdots, a_{i_k}, 0)\big]\Big)=1.$$
Then for each $\omega\in\Omega''$,  $T(\omega)\cong A$.

  {\it Case 4.} $ X_1$ is non-degenerate, $\bP(X_1=0)>0$ and $X_1$ is not discrete.
 In this case,  by the proof of Lemma \ref{P:PointSpect} (i),  we  choose $\Omega_1\subset\Omega$ with $\bP(\Omega_1)=1$ such that there exists infinitely many $n$'s such that  $X_n(\omega)=0$ for each $\omega\in\Omega_1$.
Thus $T(\omega)$ is the direct sum of a sequence of irreducible truncated weighted shifts.
 Assume that $\Delta=\{a\in\bR: \bP(X_1=a)>0\}$. Then $\sum_{a\in\Delta}\bP(X_1=a)<1$
and $\bP(X_i\in\Delta)=\bP(X_1\in\Delta)<1$ for each $i\geq 1$. So
\[\bP(X_i\in\Delta,\forall i\geq 1)=\prod_{i\geq 1}\bP(X_i\in\Delta)=\prod_{i\geq 1}\bP(X_1\in\Delta)=0.\]
Thus we  choose $\Omega_2\subset\Omega_1$ with $\bP(\Omega_2)=1$ such that $\{X_i(\omega):i\geq 1\}\nsubseteq\Delta$ for each $\omega\in\Omega_2$.
Fix  $\omega_0\in\Omega_2$. We shall prove that $\bP(T\cong T(\omega_0))=0$. Since $\omega_0\in\Omega_2$, we  choose $i_0\geq 1$ such that $X_{i_0}(\omega_0)\notin\Delta$.
Assume that $\omega\in\Omega_2$ and $T(\omega)\cong T(\omega_0)$.
It follows that $|T(\omega)|\cong |T(\omega_0)|$. Since
$$|T(\omega)|=\textup{diag}\{X_1(\omega),X_2(\omega),X_3(\omega),\cdots\}$$ and $$|T(\omega_0)|=\textup{diag}\{X_1(\omega_0),X_2(\omega_0),X_3(\omega_0),\cdots\},$$ we deduce that there exists $i\geq 1$ such that $X_i(\omega)=X_{i_0}(\omega_0)$.

  For each $j$, note that $\bP(X_j=X_{i_0}(\omega_0))=\bP(X_1=X_{i_0}(\omega_0))=0$, since $X_{i_0}(\omega_0)\notin\Delta$.
This implies that
\begin{equation*}
\bP(T\cong T(\omega_0))=\bP\big( [T\cong T(\omega_0)]\cap \Omega_2\big)
\leq\bP\big( [\exists i\geq 1\ \textup{s.t.}\ X_i(\omega)=X_{i_0}(\omega_0)]\cap \Omega_2\big)=0.
\end{equation*}
Therefore we conclude that $\bP^2\big\{(\omega_1,\omega_2)\in\Omega^2: T(\omega_1)\cong T(\omega_2)\big\}=0$.

  (ii)
Since $0\notin \essran X_1$, it follows that $T$ is almost surely an irreducible, unilateral weighted shift and, by Theorem \ref{T:compacts}, we have $\KH\subset C^*(T)$ a.s. Then there exists $\Omega'\subset\Omega$ with $\bP(\Omega')=1$ such that $\KH\subset C^*(T(\omega))$ and $T(\omega)$ is irreducible for $\omega\in\Omega'$. In view of Exercise I.37 in \cite{David}, for any $\omega_1,\omega_2\in\Omega'$,
$T(\omega_1)\cong_a T(\omega_2)$ if and only if $T(\omega_1)\cong T(\omega_2).$
By Theorem \ref{T:ApprUnitEqui} (i), we have
\begin{equation*}
\bP^2\big\{(\omega_1,\omega_2)\in\Omega^2: T(\omega_1)\cong_a T(\omega_2)\big\}
=\bP^2\big\{(\omega_1,\omega_2)\in\Omega^2: T(\omega_1)\cong T(\omega_2)\big\}=0.
\end{equation*}
In particular, $\bP(T\cong_a A)=0$ for any operator $A$.

  (iii) Without loss of generality, we may assume that $\ran X_n$ $\subset$ $\essran X_1$ for $n\geq 1$.
Choose an at most countable dense subset $\Gamma=\{\alpha_i: i\in\Lambda\}$ of $\essran X_1$.
Then there are countably many finite tuples with entries in $\Gamma$, denoted as $\eta_1,\eta_2,\eta_3,\cdots$.
If $\eta_i=(\alpha_{i,1}, \cdots, \alpha_{i,k_i})$, then we write $W_{i}$ to denote the truncated weighted shift on $\bC^{k+1}$  with weights $(\alpha_{i,1}, \cdots, \alpha_{i,k_i})$. Set $W=\oplus_{i=1}^\infty W_{i}^{(\infty)}$. Note that $W\cong W^{(\infty)}$. It follows that $C^*(W)$
contains no  nonzero compact operator. Next we shall prove that $T\cong_a W$ a.s.

  Set \begin{equation}\label{(3.7)}\Omega'=\bigcap_{k\geq 1}\bigcap_{(m_1,\cdots,m_k)}\Big[\liminf_{n\to\infty} \|(X_{n+i})_{i=1}^k-(\alpha_{m_i})_{i=1}^k\|_{\max}=0\Big], \end{equation} where the second intersection is taken over all finite tuples with entries in $\Lambda$.  By Lemma \ref{L:constantCase}, we have $\bP(\Omega')=1$.
Choose $\omega_0\in\Omega'$. For $n\geq 1$, denote $\mu_n=X_n(\omega_0)$.
By the hypothesis, $\mu_n\in\essran X_1$ for $n\geq 1$. Denote $\mu_0=0$. So $\mu_0\in\essran X_1$.
 Note that $W$ is in fact a unilateral weighted shift. We let $\{\beta_j\}_{j=1}^\infty$ denote its weight sequence.
Then for any $i\geq 1$, the finite tuple $\eta_i$ appears infinitely many times in the weight sequence of $W$.
Denote $\beta_0=0$.
 Fix a $k\geq 1$. Since $\Gamma$ is dense in $\essran X_1$,
for any $\eps>0$, we  choose $m_0, m_1,\cdots, m_k\in\Lambda$ such that
$\|(\alpha_{m_i})_{i=0}^k-(\mu_{i})_{i=0}^k\|_{\max}<\eps.$
Since $(\alpha_{m_0},\alpha_{m_1},\cdots,\alpha_{m_k})$ appears infinitely many times in the weight sequence $\{\beta_i\}$ of $W$, we have
$\liminf\limits_{n\to\infty} \|(\beta_{n+i})_{i=0}^k-(\mu_{i})_{i=0}^k\|_{\max}<\eps.$
Since $\eps>0$ is arbitrary, we deduce that
\begin{equation}\label{(3.8)}\liminf\limits_{n\to\infty} \|(\beta_{n+i})_{i=0}^k-(\mu_{i})_{i=0}^k\|_{\max}=0. \end{equation}
On the other hand, since $\beta_{0},\beta_{1},\cdots,\beta_{k}\in\Gamma$, by (\ref{(3.7)}), we have
$\liminf\limits_{n\to\infty} \|(\mu_{n+i})_{i=0}^k-(\beta_{i})_{i=0}^k\|_{\max}=0. $
In view of Lemma \ref{L:WSAppUniEqu}, this coupled with (\ref{(3.8)}) implies that $T(\omega_0)\cong_a W$.
Noting that $\omega_0\in\Omega'$ is arbitrary, we obtain $T\cong_a W$ a.s.

  (iv) For the proof of this part, we need to mention an elementary fact, whose proof is an easy exercise: Let $A,B\in\BH$. \begin{enumerate}
\item[(a)] If $A\cong_a B$, then $C^*(A)\cap\KH=\{0\}$ if and only if $C^*(B)\cap\KH=\{0\}$.
\item[(b)] If $A\cong A^{(\infty)}$, then $C^*(A)\cap\KH=\{0\}$.
\end{enumerate}


%

  We resume the proof of (iv). Since $0\in\essran X_1$, by the proof of Theorem \ref{T:ApprUnitEqui} (iii), there exists a weighted shift $W$ such that
$T\cong_a W\cong W^{(\infty)}$ a.s. By the above fact, $C^*(W)$ contains no nonzero compact operators.
Assume that $A$ is an operator acting on some Hilbert space. 
If $T\cong_a A$ a.s., then $W\cong_a A$, which implies $W\approx A$ and,
by the fact again, $C^*(A)$ contains no nonzero compact operators.

  Conversely,  since $T\approx A$ a.s., we obtain  $W\approx A$. Then there exists a $*$-isomorphism $\varphi:C^*(W)\rightarrow C^*(A)$ such that $\varphi(W)=A$. It follows that $\|\varphi(Z)\|=\|Z\|$ for $Z\in C^*(W)$. Since neither $C^*(W)$ nor $C^*(A)$ contains nonzero compact operators, we deduce that $\rank~ \varphi(Z)=\rank~ Z=\rank~\textup{id}(Z)$ for  $Z\in C^*(W)$, where $\textup{id}$ is the identity representation of $C^*(W)$. By Theorem II.5.8 in \cite{David}, we get $\varphi\cong_a\textup{id}$. It follows that $A=\varphi(W)\cong_a\textup{id}(W)=W$. Hence $T\cong_a A$ a.s.
\end{proof}

  Next we give the proof of Theorem \ref{T:compacts}.
\begin{proof}[Proof of Theorem \ref{T:compacts}]
If $0\in\essran X_1$, then, by the proof of Theorem \ref{T:ApprUnitEqui} (iii),  there exists a weighted shift $W$ such that
$T\cong_a W\cong W^{(\infty)}$ a.s. In view of the fact in the proof of Theorem \ref{T:ApprUnitEqui} (iv), $C^*(T)$ almost surely contains no nonzero compact operators.
 If $0\notin\essran X_1$, then $\inf_i X_i>0$ a.s. So $T$ is almost surely a left-invertible unilateral weighted shift.
It is not hard to verify that each left-invertible unilateral weighted shift $A$ on $\cH$ has $\KH$ contained in $ C^*(A)$. So we conclude the proof.\end{proof}

%

\noindent The next result is for a curious reader. Let $W$ be the deterministic unilateral weighted shift defined in the proof of Theorem \ref{T:ApprUnitEqui} (iii).
Then $W\cong W^*$. Thus it follows that

\begin{cor}\label{T:Kakutani}
Let  $T\sim\{X_n\}_{n=1}^{\infty}.$
If $0\in\essran X_1$, then $T\cong_a T^*$ a.s.
\end{cor}

\begin{rmk}
The random weighted shift in Corollary \ref{T:Kakutani} is almost surely an approximate Kakutani shift.
A unilateral weighted shift $A$ with nonnegative weights $\{\lambda_i\}_{i=0}^\infty$ is said to be {\it approximate Kakutani} if for each $n\geq 1$ and any $\eps>0$ there exists an index $n_0\geq n$ such that
$\lambda_{n_0}<\eps$ and $1\leq k\leq n\Longrightarrow |\lambda_k-\lambda_{n_0-k}|<\eps.$ It was proved that an irreducible unilateral weighted shift $A$ is approximate Kakutani if and only if $A\cong_a A^*$ (Theorem 2.4 in \cite{GuoJiZhu}). Approximate Kakutani shifts, as a generalization of the Kakutani shift (\cite{Rickart}, P. 282), originally arose in the study of complex symmetric operators \cite{2013Gar}. There are several other equivalent characterizations (\cite{GuoJiZhu}, Theorem 2.4).
\end{rmk}

\section{Random Hardy spaces associated with $T$}\label{S:random Hardy}
  For each $\om\in\Om,$   $T(\om)$ is the weighted shift with weight sequence $\{X_n(\om)\}.$ By \cite{shields}, $T(\om)$ on $\cH$ is unitarily equivalent to $M_z$, the multiplication by the coordinate function on $H_{\mu}^2(\om)$ which is given by
\begin{equation}\label{E:H mu 2}
H^2_{\mu}(\om)=\Big\{f(z)=\sum_{n=0}^{\infty}a_nz^n:\ \|f\|^2(\om)=\sum_{n=0}^{\infty}|a_n|^2w_n^2(\om)<\infty\Big\},
\end{equation}
where $w_0=1,$  $w_n=X_1\cdots X_{n}\ (n\geq 1),$ and  $\mu$ denotes the law of $X_1.$ Then $H^2_{\mu}$ is what we call the random Hardy space associated with $T$. A starting point for analyzing elements in $H^2_{\mu}$ follows from the Hewitt-Savage zero-one law: For any analytic function $f,$
\begin{equation}\label{E:two case}
\bP(\|f\|<\infty)\in \{0, 1\}.
\end{equation}
Here we adopt the convention that $\|f\|=\infty$ if $f$ is not in the Hilbert space. Moreover, in this section, we assume that $X_1$ is non-degenerate so the four radii in (\ref{E:fourradii}) are all different.
Let $\gamma(H^2_{\mu})$  be the convergence radius of $H^2_{\mu}.$
By Hewitt-Savage,
 $\gamma(H^2_{\mu})=\liminf\limits_{n\to\infty}\sqrt[n]{w_n(\om)}$ is a constant a.s.  Then a direct application of  the strong law of large numbers yields  the following constant.

\begin{lem}\label{T:convergence radius}
 We have $
\gamma(H^2_{\mu})=e^{\bE(\ln X_1)}
$ a.s.
\end{lem}

\subsection{Membership without moment conditions}

  This section is devoted to the study  the membership of  $f\in H_{\mu}^2$ a.s.  By (\ref{E:two case}), it is reasonable to introduce a deterministic space $H_{\ast}$ as below:
$$
H_{\ast}=\Big\{f(z)=\sum_{n=0}^{\infty}a_nz^n:\ f\in H_{\mu}^2\ \mbox{a.s.}\Big\}.
$$

  As a first step, we have the following two lemmas.

\begin{lem}\label{L:expect to series}
Let $f(z)=\sum_{n=0}^{\infty}a_nz^n$ be an analytic function.  If
 one of the following conditions holds,
\begin{itemize}
\item[(i)] for $0<p<2,$
$
\sum_{n=0}^{\infty}|a_n|^p\|X_1\|_p^{pn}<\infty;
$
\item[(ii)] for $2\leq p<\infty,$
$
\sum_{n=0}^{\infty}|a_n|^2\|X_1\|_p^{2n}<\infty,
$
\end{itemize}
then $f\in H_{\ast}.$
\end{lem}
\begin{proof}
It is sufficient to observe  that
  if $0<p<2,$ then
\begin{equation}\label{E:48}
\Big(\sum_{n=0}^{\infty}|a_n|^2\|X_1\|_p^{2n}\Big)^{\frac{p}{2}}\leq \bE(\|f\|^p)\leq \sum_{n=0}^{\infty}|a_n|^p\|X_1\|_p^{pn};
\end{equation}
 and if $2\leq p<\infty,$ then
\begin{equation}\label{E:49}
\sum_{n=0}^{\infty}|a_n|^p\|X_1\|_p^{pn}\leq \bE(\|f\|^p)\leq \Big(\sum_{n=0}^{\infty}|a_n|^2\|X_1\|_p^{2n}\Big)^{\frac{p}{2}} .
\end{equation}
They   follow from the Minkowski  inequality.
\end{proof}

  Now we discuss the membership in terms of the convergence radius.
\begin{lem}\label{T:norm p}
Suppose that   $f(z)=\sum_{n=0}^{\infty}a_nz^n$ is an analytic function with $\gamma(f)$ as its convergence radius.
\begin{itemize}
\item[(a)]
 If $\gamma (f)<e^{\bE(\ln X_1)}, $ then $f\notin H_{\ast};$
\item[(b)] If $\gamma (f)>e^{\bE(\ln X_1)}, $ then $f \in H_{\ast}.$
\item[(c)] Let $1\leq p< \infty.$ Then
   $\bE(\|f\|^p)=\infty$ if $\gamma(f)<\|X_1\|_p,$ and
   $\bE(\|f\|^p)<\infty$ if $\|X_1\|_p<\gamma(f).$
\end{itemize}
\end{lem}

\begin{rmk}  If we choose $p=1$ and $\gamma(f)=\liminf\limits_{n\to\infty}\frac{1}{\sqrt[n]{|a_n|}}$ strictly between $e^{\bE(\ln X_1)}$ and $\|X_1\|_1,$  then this proposition implies that  that $f\in H_{\ast}$ is a ``large" element in the following sense: $$\bE(\|f\|)=\infty.$$
\end{rmk}

\smallskip

\begin{proof}[Proof of Lemma \ref{T:norm p}]
By Lemma  \ref{T:convergence radius},  $f\notin H_{\ast}$ if $\gamma (f)<e^{\bE(\ln X_1)}.$  If $\gamma (f)>e^{\bE(\ln X_1)},$
by the strong law of large numbers, almost surely,
$
\lim\limits_{n\to\infty}(X_1^2\cdots X_{n}^2)^{\frac{1}{n}}=e^{\bE(\ln X_1^2)},
$
which,  together with $\gamma (f)>e^{\bE(\ln X_1)},$ implies that
$
\|f\|^2=\sum_{n=1}^{\infty}|a_n|^2(X_1\cdots X_{n})^2<\infty$ \ a.s. The proof of part (c) follows from   (\ref{E:48}) and (\ref{E:49}).
\end{proof}

  Some quick examples follow.
Suppose that $\|X_1\|_1=1$  and
   $f(z)=\sum_{n=0}^{\infty}a_nz^n$ with
$\gamma(f)=1,\ \ \sum_{n=0}^{\infty}|a_n|^2=\infty.$
By  (\ref{E:48}) we have
$
 \bE(\|f\|)=\infty.
$
If we  require that
$\gamma(f)=1,\ \ \sum_{n=0}^{\infty}|a_n|<\infty,$
then also by (\ref{E:48}),
$ \bE(\|f\|)<\infty.$
In particular, when $\alpha >1$,
\begin{equation}\label{E:confusing}
f(z)=\sum_{n=1}^{\infty}\frac{1}{n^{\al}}z^n\in H_{\ast}\end{equation} is a small element, that is, $ \bE(\|f\|)<\infty.$ (Please note the normalization  $\|X_1\|_1=1$.)
Next comes an  interesting example  to illustrate the critical values in Lemma \ref{T:norm p}.

\begin{lem}\label{L:Example 1/2}
Let $\bE(\ln X_1)=0$, $\bE\big((\ln X_1)^2\big) \in (0, \infty)$ and
\begin{equation}\label{E:n n al}
f(z)=\sum_{n=1}^{\infty}\frac{1}{n^{n^{\al}}}z^n,\quad \al>0.
\end{equation}
 Then $f\in H_{\ast}$ if and only if $\al\geq \frac{1}{2}.$
\end{lem}

\begin{rmk}\label{R:sum al}
Observe that
$
\gamma(f)=\liminf\limits_{n\to\infty}\frac{1}{\sqrt[n]{|a_n|}}=1
$
if and only if $\al<1.$
Corollary \ref{C:n al} will tell us that $f(z)=\sum_{n=1}^{\infty}\frac{1}{n^{\al}}z^n\notin H_{\ast}$ for  $\al>0$ under the same conditions as in Lemma \ref{L:Example 1/2}. This is not to be confused with (\ref{E:confusing}).
\end{rmk}

\begin{proof}[Proof of Lemma \ref{L:Example 1/2}]
We recall a well-known variant of the strong law of large numbers \cite{Resnick}:
{\it
If $\{Y_n\}_{n\geq 1}$ are i.i.d. random variables  with mean $\mu$ and finite variance. Then, almost surely,
$
\lim\limits_{n\to\infty}\frac{S_n-n\mu}{\sqrt{n(\ln n)^{\delta}}}=0
$
for any $\delta>1,$ where $S_n=Y_1+\cdots+Y_n.$}
So    $
\lim\limits_{n\to\infty}\frac{\ln X_1^2+\cdots+\ln X_n^2}{\sqrt{n(\ln n)^{\delta}}}\to 0\ \text{a.s.}
$
Hence,
$
\frac{1}{n^{2n^{\al}}}X_1^2\cdots X_n^2\leq \frac{e^{\ep\sqrt{n(\ln n)^{\delta}}}}{n^{2n^{\al}}}
$
for any $\ep>0$ when $n$ is large enough.
If $\al\geq \frac{1}{2}$ and $1<\delta<2,$  then $\sum_{n=1}^{\infty}\frac{e^{\ep\sqrt{n(\ln n)^{\delta}}}}{n^{2n^{\al}}}<\infty.$
This implies that
$
\|f\|^2=\sum_{n=0}^{\infty}|a_n|^2X_1^2\cdots X_n^2<\infty,
$
as desired.
If $\al<\frac{1}{2},$ as a direct application  of Theorem \ref{T:LIL our case} below, $f\notin H_{\ast}.$
\end{proof}


\subsection{Membership with moment conditions}
In this subsection we   consider the membership of $H_{\ast}$ under a moment condition.
\begin{thm}\label{T:LIL our case}
Let $\bE(\ln X_1)=0$, $\bE\big((\ln X_1^2)^2\big)=\sig^2 \in (0, \infty),$ and  $f(z)=\sum_{n=0}^{\infty}a_nz^n$ be an analytic function.
\begin{itemize}
\item[(i)] If
$\limsup\limits_{n\to\infty}|a_n|^{\frac{\sqrt{2}}{\sig \sqrt{n\ln\ln n}}}< \frac{1}{e},$
then
$f\in H_{\ast}.$
\item[(ii)] If
$\liminf\limits_{n\to\infty}|a_n|^{\frac{\sqrt{2}}{\sig\sqrt{n\ln\ln n}}}> \frac{1}{e}$ \ or \
$\limsup\limits_{n\to\infty}|a_n|^{\frac{\sqrt{2}}{\sig\sqrt{n\ln\ln n}}}> 1,$
then
$f\notin H_{\ast}.$
\item[(iii)] We have
\begin{equation}\label{E:eachasup}\sup_{f \in H_\ast} \limsup\limits_{n\to\infty}|a_n|^{\frac{\sqrt{2}}{\sig \sqrt{n\ln\ln n}}} \in [e^{-1}, 1]
\end{equation} and each $a \in [e^{-1}, 1]$ may be achieved.
\item[(iv)] If $f \in H_\ast$, then \begin{equation}\label{E:eachainf} \liminf\limits_{n\to\infty}|a_n|^{\frac{\sqrt{2}}{\sig \sqrt{n\ln\ln n}}} \in [0, e^{-1}]\end{equation} and each $a \in [0, e^{-1}]$ may be achieved.
\end{itemize}
\end{thm}

\begin{rmk}\label{R:both possible}
 Theorem \ref{T:LIL our case} is the best possible in its form in the following sense: at the end of the proof of Theorem \ref{T:LIL our case} we shall use examples (\ref{Ex:Divergence}, \ref{Ex:Convergence}) to illustrate that for
$0\leq \liminf\limits_{n\to\infty}|a_n|^{\frac{\sqrt{2}}{\sig \sqrt{n\ln\ln n}}}\leq \frac{1}{e}$
and
$\frac{1}{e}\leq \limsup\limits_{n\to\infty}|a_n|^{\frac{\sqrt{2}}{\sig \sqrt{n\ln\ln n}}}\leq 1,$
  both $f\in H_{\ast}$ and  $f\notin H_{\ast}$ are possible.  Hence, for this range, one needs finer tests instead of mere $\limsup$ and $\liminf$ as above. Moreover, a moment of thought of examples (\ref{Ex:Divergence}, \ref{Ex:Convergence})  reveals (iii) and (iv).
\end{rmk}

  In essence Theorem \ref{T:LIL our case} concerns the convergence of   \begin{equation}\|f\|^2=|a_0|^2+\sum_{n=1}^{\infty}|a_n|^2X_1^2\cdots X_n^2,\end{equation} which    is an infinite summation of  dependent random varaibles. Its general theory poses  an interesting yet elegant problem in   probability theory.

%
  Now the promised example in  Remark \ref{R:sum al} is an easy corollary.
\begin{cor}\label{C:n al}
Under the conditions in Theorem \ref{T:LIL our case},  for each $\al>0,$
\begin{equation}
f(z)=\sum_{n=1}^{\infty}\frac{1}{n^{\al}}z^n\notin H_{\ast}.
\end{equation}
\end{cor}


  Recall that $\D_{\al}$  consists of the analytic functions $f(z)=\sum_{n=0}^{\infty}a_nz^n$ over the unit disk satisfying
$
\|f\|_{\al}^2=\sum_{n=0}^{\infty}(n+1)^{\al}|a_n|^2<\infty
$ \cite{shields}.
Theorem \ref{T:LIL our case} suggests that $H_*$ is a rather small space when compared with the standard Hardy space $H^2(\mathbb{D})$ or its cousins.

\begin{rmk}\label{C:smallspace}
{Under the conditions in Theorem \ref{T:LIL our case},
if $ \limsup\limits_{n\to\infty}|a_n|^{\frac{\sqrt{2}}{\sig \sqrt{n\ln\ln n}}} <1
$,
then
$
f \in \D_{\al}
$
 for each $\al>1.$ In particular, $f \in A(\mathbb{D})$. } Recall that $\D_\al$ is an algebra contained in the disk algebra $\D_\al \subset A(\mathbb{D})$  if $\al >1$ \cite{shields}. We conjecture that the condition $ \limsup\limits_{n\to\infty}|a_n|^{\frac{\sqrt{2}}{\sig \sqrt{n\ln\ln n}}} <1
$ is unnecessary.
\end{rmk}

\begin{proof}[Proof of Theorem \ref{T:LIL our case}]
We prove (i) and (ii) here only. To alleviate the notations,  we let    $Y_n=X_n^2$ and $c_n=|a_n|^2$.
(i) Applying the law of iterated logarithm \cite{Hartman-Wintner}, we have
$\limsup_{n\to\infty}(Y_1\cdots Y_{n})^{\frac{1}{\sqrt{2\sig^2n\ln\ln n }}}=e
$
almost surely.
If
$\limsup\limits_{n\to\infty}c_n^{\frac{1}{\sqrt{2\sig^2n\ln\ln n}}}< \frac{1}{e},$
then
we easily have
$
\sum_{n=1}^{\infty}c_nY_1\cdots Y_n<\infty$ a.s.
as desired.
(ii)
If   $\liminf\limits_{n\to\infty}|a_n|^{\frac{\sqrt{2}}{\sqrt{\sig^2n\ln\ln n}}}> \frac{1}{e}$, then the law of iterated logarithm  yields that, almost surely,  there exists a subsequence $\{n_k\}$
such that
$\lim_{k\to\infty}(Y_{1}\cdots Y_{n_{k}})^{\frac{1}{\sqrt{2\sigma^2n_{k}\ln\ln n_{k}}}}=e.$
So we easily have
$
\sum_{n=1}^{\infty}c_nY_1\cdots Y_n=\infty
$
a.s.
as desired.
If $\limsup\limits_{n\to\infty}|a_n|^{\frac{\sqrt{2}}{\sqrt{\sig^2n\ln\ln n}}}> 1,$
then there is a subsequence $\{n_k\}$ such that
$c_{n_k}^{\frac{1}{\sqrt{2\sig^2n_k\ln\ln n_k}}}>1.$
%
Recall that the random walk $S_n=\ln Y_1+\cdots+\ln Y_n$ satisfies  that
$
-\infty=\liminf\limits_{n\to\infty}S_n
<\limsup\limits_{n\to\infty}S_n=\infty$ a.s. (\cite{Durrett}, Proposition 4.1.2)
Together with the law of the iterated logarithm, one has that
 $0\leq \limsup\limits_{n\to\infty}\frac{S_n}{\sqrt{2\sig^2n\ln\ln n}}\leq 1$ a.s.
 For the subsequence $\{n_k\}$, this yields that
\begin{equation*}
\limsup_{k\to\infty} (Y_1\cdots Y_{n_k})^{\frac{1}{\sqrt{2\sig^2n_k\ln\ln n_k}}}=e^{\limsup\limits_{k\to\infty} \frac{\ln Y_1+\cdots+\ln Y_{n_k}}{\sqrt{2\sig^2n_k\ln\ln n_k}}}\geq 1\ \text{a.s.}
\end{equation*}
So almost surely there exists a subsequence $\{n_{k_j}\}$ such that
$
c_{n_{k_j}}Y_1\cdots Y_{n_{k_j}}
\geq  \lam^{\sqrt{2\sig^2n_{k_j}\ln\ln n_{k_j}}}
$
for some $\lam>1$ when $j$ is large enough, which ensures that $
\sum_{n=1}^{\infty}c_nY_1\cdots Y_n=\infty $ a.s.

\end{proof}

  The rest of this subsection is devoted to constructing examples as promised in Remark \ref{R:both possible}. This is indeed the main technical challenge we face in the present subsection.   Recall that now we have
\begin{equation*}
0\leq \liminf\limits_{n\to\infty}c_n^{\frac{1}{\sqrt{2\sig^2n\ln\ln n}}}\leq \frac{1}{e}
\ \text{and}\
\frac{1}{e}\leq \limsup\limits_{n\to\infty}c_n^{\frac{1}{\sqrt{2\sig^2n\ln\ln n}}}\leq 1.
\end{equation*}
In fact, for the divergence example, we need only to consider the ``smallest" case, i.e.,
\begin{equation*}
 \liminf\limits_{n\to\infty}c_n^{\frac{1}{\sqrt{2\sig^2n\ln\ln n}}}= 0
\ \text{and}\
\limsup\limits_{n\to\infty}c_n^{\frac{1}{\sqrt{2\sig^2n\ln\ln n}}}=\frac{1}{e};
\end{equation*}
and for  convergence, it is enough to treat the ``largest" case, i.e.,
\begin{equation*}
 \liminf\limits_{n\to\infty}c_n^{\frac{1}{\sqrt{2\sig^2n\ln\ln n}}}= \frac{1}{e}
\ \text{and}\
\limsup\limits_{n\to\infty}c_n^{\frac{1}{\sqrt{2\sig^2n\ln\ln n}}}=1.
\end{equation*}

\begin{eg}[Divergence]\label{Example:Divergence}
 Let
$
\bP(Y_1=e)=\bP(Y_1=\frac{1}{e})=\frac{1}{2}
$
and
\begin{equation}\label{Ex:Divergence}
c_n=
\begin{cases}
0,& n=n_k=2^{2^k};\\
\frac{1}{e^{\sqrt{2n\ln\ln n}}},& \text{other} \ n\geq 3.
\end{cases}
\end{equation}
Then
$
\sum_{n=3}^{\infty}c_nY_1\cdots Y_n=\infty$ a.s.

  To verify that this is the example we want, we need two results which might not be familiar to all, so we record them below for the reader's convenience.

\begin{thm}[\cite{1946Feller}]\label{T:upper and lower}
Assume that $\{X_n\}$ are i.i.d. random variables with mean 0,  variance 1,  distribution function $F,$ and
\begin{equation}\label{E:moment}
\int_{|t|>x}t^2dF(t)=O\Big(\frac{1}{\ln\ln x}\Big),\quad x\to\infty.
\end{equation}
For any increasing (unbounded) sequence $\{\varphi_n\},$ the divergence (convergence) of the series
$
\sum_n\frac{\varphi_n}{n}e^{-\frac{\varphi_n^2}{2}},
$
is a necessary and sufficient condition that, with probability one, the inequality
$
S_n>n^{\frac{1}{2}}\varphi_n
$
be satisfied for infinitely (only finitely) many $n.$
\end{thm}

\begin{thm}[\cite{Gut subsequence}]\label{T:sup<1}
Let $\{n_k\}$ be a strictly increasing subsequence of the positive integers such that
$
\limsup\limits_{k\to\infty}\frac{n_k}{n_{k+1}}<1
$
and let
\begin{equation}\label{E:ep ast}
\ep^{\ast}=\inf\Big\{\ep>0:\ \sum_{k=3}^{\infty}\frac{1}{(\ln n_k)^{\frac{\ep^2}{2}}}<\infty\Big\}.
\end{equation}
Furthermore, let $\{X_n\}_{n=1}^{\infty}$ be i.i.d. random variables. Set $S_n=\sum_{k=1}^nX_k$
and suppose that $\bE (X_1)=0$ and $\bE (X_1^2)=\sig^2<\infty.$ Then
\begin{equation}\label{E:39}
\limsup_{k\to\infty}\frac{S_{n_k}}{\sqrt{\sig^2n_k\ln\ln n_k}}=\ep^{\ast}\ a.s.\
\text{and}\
\liminf_{k\to\infty}\frac{S_{n_k}}{\sqrt{\sig^2n_k\ln\ln n_k}}=-\ep^{\ast}\  a.s.
\end{equation}
In particular, if $\ep^{\ast}=0,$ then
$
\lim\limits_{k\to\infty}\frac{S_{n_k}}{\sqrt{n_k\ln\ln n_k}}=0 \ a.s.
$
For the converse, suppose that $\ep^{\ast}>0.$ If
$
\bP\left(\limsup\limits_{k\to\infty}\frac{S_{n_k}}{\sqrt{n_k\ln\ln n_k}}<\infty\right)>0,
$
then $\bE \big(X_1^2\big)<\infty$ and $\bE\big(X_1\big)=0.$
\end{thm}

  Next we point out how to use the above results to verify our example. We apply  Theorem \ref{T:upper and lower} to   $\{\ln Y_n\}_{n=1}^{\infty}$ with
$
\varphi_n=\sqrt{2\ln\ln n}.
$
Hence, almost surely, there exists a subsequence $\{n_i\}$ such that
$\ln Y_1+\cdots+\ln Y_{n_i}>\sqrt{2n_i\ln\ln n_i}.$
%
On the other hand, Theorem \ref{T:sup<1} implies that
$
\limsup\limits_{k\to\infty}\frac{\ln Y_1+\cdots+\ln Y_{n_k}}{\sqrt{2n_k\ln\ln n_k}}=0$ a.s.
So the subsequence
$\{n_i\}$ is different from  $\{n_k\}$ except for finite terms, which yields that
$
a_{n_i}=\frac{1}{e^{\sqrt{2n_i\ln\ln n_i}}}
$
when $i$ is large enough.
Hence
$a_{n_i}Y_1\cdots Y_{n_i}=\frac{1}{e^{\sqrt{2n_i\ln\ln n_i}}}Y_1\cdots Y_{n_i}
>1,$
which yields the desired result.
\end{eg}

\begin{eg}[Convergence]\label{Example:Convergence}
Let  $\bP(Y_1=e)=\bP(Y_1=\frac{1}{e})=\frac{1}{2},$ and
\begin{equation}\label{Ex:Convergence}
c_n=
\begin{cases}
\frac{1}{(e+1)^{\sqrt{2n\ln\ln\ln n}}},& n=n_k=3^{3^{3^k}};\\
\frac{1}{e^{b_n\sqrt{2n\ln\ln n}}},& \text{other}\ n\geq 16,
\end{cases}
\end{equation}
where
$
b_n=\sqrt{1+\frac{a\ln\ln\ln n}{2\ln\ln n}} \ (a>3).
$
Then
$
\sum_{n=16}^{\infty}c_nY_1\cdots Y_n<\infty$ a.s.

  To verify that this example is sufficient for the purpose, we need a  (well-known) variant  of the law of the iterated logarithm. Again we record it  below for the convenience of the reader.

\begin{thm}[\cite{Gut subsequence}]\label{T:nlog k}
Let $\{n_k\}$ be a strictly increasing subsequence of the positive integers such that
$
\limsup\limits_{k\to\infty}\frac{n_k}{n_{k+1}}<1.
$
Furthermore, let $\{X_n\}_{n=1}^{\infty}$ be i.i.d. random variables. Set $S_n=\sum_{k=1}^nX_k$
and suppose that $\bE (Y_1)=0$ and $\bE (Y_1^2)=\sig^2<\infty.$ Then
\begin{equation*}
\limsup_{k\to\infty}\frac{S_{n_k}}{\sqrt{2\sig^2n_k\ln k}}=1 \  \
\textup{and} \
\liminf_{k\to\infty}\frac{S_{n_k}}{\sqrt{2\sig^2n_k\ln k}}=-1 \ \text{a.s.}
\end{equation*}
For the converse, if
$
\bP\left(\limsup\limits_{k\to\infty}\frac{|S_{n_k}|}{\sqrt{n_k\ln\ln k}}<\infty\right)>0,
$
then $\bE (Y_1^2)<\infty$ and $\bE (Y_1)=0.$
\end{thm}

  Now we explain how the example can be verified. By Theorem \ref{T:nlog k},
$
\limsup\limits_{k\to\infty}(Y_1\cdots Y_{n_{k}})^{\frac{1}{\sqrt{2n_{k}\ln k}}}=e$ almost surely,
which yields  that for small $\ep>0$, almost surely,
$
c_{n_k}Y_1\cdots Y_{n_{k}}
\leq c^{\sqrt{2n_k\ln\ln\ln n_k}}
$ for some $c \in (0, 1).$
Next we shall show that $\sum\limits_{n\neq n_k}c_{n}Y_1\cdots Y_{n}<\infty$ a.s., which follows from  Lemma \ref{L:b_n con} below.

\begin{lem}\label{L:b_n con}
Let $\{X_n\}_{n=1}^{\infty}$ be nonnegative i.i.d. random variables  and $\{\ln X_n\}_{n=1}^{\infty}$ satisfy the conditions in Theorem \ref{T:upper and lower}. Let
$
a_n=\frac{1}{e^{b_n\sqrt{2n\ln\ln n}}} (n\geq 16)
$
with
$$
b_n=\sqrt{1+\frac{a\ln\ln\ln n}{2\ln\ln n}}\ (a>3).
$$
Then
$
\sum_{n=16}^{\infty}a_nX_1\cdots X_n<\infty$ a.s.

\end{lem}
\begin{proof}
Let $S_n=\ln X_1+\cdots+\ln X_n.$ By Levy \cite{Levy},
$\bP[S_n\geq b_n\sqrt{2\ln\ln n},\ \mbox{i.o.}]=0.$
So Theorem \ref{T:upper and lower} implies that
$\sum_n\frac{\varphi_n}{n}e^{-\frac{\varphi_n^2}{2}}<\infty$
with
$
\varphi_n=b_n\sqrt{2\ln\ln n},
$
which yields that
$
\sum_{n=16}^{\infty}\frac{\varphi_n'}{n}e^{-\frac{\varphi_n'^2}{2}}<\infty
$
with
$
\varphi'_n=\Big(b_n-\frac{1}{n^{\al}}\Big)\sqrt{2\ln\ln n}
$
for any $ \al>0.$ By Theorem \ref{T:upper and lower} again,  we see that
$
\bP\Big(\Big[S_n\geq \Big(b_n-\frac{1}{n^{\al}}\Big)\sqrt{2\ln\ln n} \ \text{i.o.}\Big]\Big)=0,
$
that is, almost surely,
\begin{equation}\label{E:S_n-b_n<}
S_n<\Big(b_n-\frac{1}{n^{\al}}\Big)\sqrt{2\ln\ln n}
\end{equation}
 when $n$ is large enough.
 In particular, if we choose $0<\al<\frac{1}{2},$ then
\begin{equation}\label{E:sum e^..}
\sum_{n=16}^{\infty}\big(e^{-\frac{1}{n^{\al}}}\big)^{\sqrt{2n\ln\ln n}}<\infty.
\end{equation}
Now
\begin{eqnarray*}
a_nX_1\cdots X_n&=&\Big(a_n^{\frac{1}{\sqrt{2n\ln\ln n}}} (X_1\cdots X_n)^{\frac{1}{\sqrt{2n\ln\ln n}}}\Big)^{\sqrt{2n\ln\ln n}}\\
&=&\Big(e^{\frac{S_n}{\sqrt{2n\ln\ln n}}-b_n} \Big)^{\sqrt{2n\ln\ln n}}
< \big(e^{-\frac{1}{n^{\al}}}\big)^{\sqrt{2n\ln\ln n}},
\end{eqnarray*}
where $``<"$ follows from (\ref{E:S_n-b_n<}).
So (\ref{E:sum e^..}) yields
$
\sum_{n=16}^{\infty}a_nX_1\cdots X_n<\infty
$
 a.s., as desired.
\end{proof}
\end{eg}

\section{ von Neumann-type inequalities}\label{S:vonNeumanIneq}

  This section seeks to understand the random operator $T=T(\omega)$ by comparing it with classical, deterministic operators. This is done for (1)   weighted shifts, (2)   bilateral weighted shifts, and (3) normal operators. We expand (\ref{OpeFunc}) and offer comparison
 in terms of von Neumann-type inequalities. In a somehow narrow sense, only comparision for (3) should be called the von Neumann-type, and it enriches our understanding of the operator-function relationship. On the other hand, interests in (1) is obvious, and our results suggest that the random shifts are often extremal elements among the collection of all weighted shifts. Interests in (2) is mostly out of curiosity, hence the treatment will be brief. It is indeed curious to us such a comparison as for (2) can be performed quite satisfactorily, although with considerable technicality.

 As an application of the above comparison, the classification problem for algebraical equivalence, left open in Section \ref{S:SampleClassify}, is accomplished in Subsection \ref{Subs:Al}.

  For a Hilbert space operator $A$, the von Neumann inequality asserts that
\[\|p(A)\|\leq \max_{|z|\leq\|A\|}|p(z)|\] for all polynomials $p(z)$ in one variable. If $N$ is a normal operator with $\sigma(N)=\overline{B(0,\|A\|)}$, then the above inequality can be re-written as
\begin{equation}\label{(4.1)}\|p(A)\|\leq \|p(N)\|\end{equation}for all polynomials $p(z)$ in one variable.
We show that the normal operator $N$ in (\ref{(4.1)}) can be replaced by the random weighted shifts.
For convenience, if $A,B$ are two operators, we write $A \lhdd B$ to denote that
$\|p(A)\|\leq \|p(B)\|$ for all polynomials $p(z)$ in one variable.

\begin{lem}
Let $T\sim \{X_n\}_{n=1}^\infty.$  Then a Hilbert space operator $A$ satisfies $A\lhdd T$ a.s. if and only if $\|A\|\leq R$.
\end{lem}

\begin{proof}
By Lemma \ref{T:NormSpectral},   the necessity is obvious.
Conversely,
fix a polynomial $p(z)$ in one variable. Using the von Neumann inequality, we obtain
$\|p(A)\| \leq \sup\{|p(z)|: |z|\leq R\}.$
In view of  Lemma \ref{T:NormSpectral} again and the spectral mapping theorem, we obtain $\|p(A)\|\leq\|p(T)\|$ a.s.
\end{proof}

  We shall show in this section that much more can be achieved for the random model. In fact, we shall prove some inequalities like (\ref{(4.1)}) for polynomials in two free variables. Given two operators $A,B$, we write $A\lhd B$ to denote that
$\|p(A,A^*)\|\leq \|p(B,B^*)\|$ for all polynomials $p(x,y)$ in two free variables.
The aim of this section is to determine for a random weighted shift $T$ when an operator $A$ satisfies $A\lhd T$ or $T\lhd A$.
When $A$ lies in the class of weighted shifts or the class of normal operators, we give a complete answer.
As an application, we are able to classify the samples of $T$ up to algebraical equivalence.

  Now we introduce some notations.
 Let $A$ be a bilateral weighted shift with weights $\{\lambda_n\}_{n\in\bZ}$. For each $n\geq 1$,
the {\it $n$-spectrum} of $A$, denoted by $\Sigma_n(A)$, is defined to be the closure (in the usual
topology on $\bR^n$) of the set
$$\{(|\lambda_{i+1}|, |\lambda_{i+2}|, \cdots, |\lambda_{i+n}|) : i\in\bZ\}.$$
Similarly, if $B$ is a unilateral weighted shift with weights $\{\mu_n\}_{n=1}^{\infty}$, then the {\it $n$-spectrum} of $B$ is defined to be
$ \Sigma_n(B)=\text{closure of} \ \{(|\mu_{i}|, |\mu_{i+1}|, \cdots, |\mu_{i+n-1}|) : i\geq 1\}.$

\begin{rmk}\label{R:Modulus}
If $A$ is a (unilateral or bilateral) weighted shift with weights  $\{\lambda_n\}$, then
$|A|$ is a diagonal operator with its diagonal entries being $\{|\lambda_n|: n\}$. Thus $\sigma(|A|)=\overline{\{|\lambda_n|: n\}}=\Sigma_1(A)$.
\end{rmk}

\subsection{Unilateral weighted shifts}
The main result of this subsection is the following theorem.

\begin{thm}\label{T:UnilateralAUE}
Let $T\sim \{X_n\}_{n=1}^\infty$,  $\Gamma=\essran X_1,$  $\Gamma_1=\Gamma\setminus\{0\}$ and $A$ be a unilateral weighted shift with nonnegative weights $\{a_i\}_{i=1}^\infty$.
 \begin{enumerate}
\item[(i)] If $0\in\Gamma$, then $\bP(A \lhd T)=\begin{cases}
1,& \Sigma_1(A)\subset\Gamma,\\
0,& \Sigma_1(A)\nsubseteq\Gamma.
\end{cases}$
\item[(ii)] If $0\notin\Gamma$, then $\bP(A \lhd T)=\begin{cases}
1,& \textup{if}~ \card~\Gamma=1\ \textup{and}\ \Sigma_1(A)=\Gamma,\\
0,& \textup{otherwise}.
\end{cases}$
\item[(iii)] If $0\notin\Gamma$ or $0$ is an accumulation point of $\Gamma$, then
$T\lhd A $ a.s. if and only if $ \{0\}\times\Gamma^n\subset\Sigma_{n+1}(A)$ for $n\geq 1$;
\item[(iv)] If $0$ is an isolated point of $\Gamma$, then $T \lhd A$ a.s. if and only if one of the following three holds:
\begin{itemize}
    \item[(a)] $a_1=0$ and $\{0\}\times\Gamma_1^n\times\{0\}\subset\Sigma_{n+2}(A)$ for $n\geq 1;$
  \item[(b)]   $a_1\ne0$ and $\{0\}\times\Gamma_1^n\times\{0\}\subset\Sigma_{n+2}(A)$ for  $n\geq 0;$
     \item[(c)]  $\exists m\geq 1$ s.t. $(a_1,\cdots, a_{m+1})\in\Gamma_1^{m}\times\{0\}$, $$\{0\}\times\Big[\Gamma_1^m\setminus\{(a_1,\cdots,a_m)\}\Big]\times\{0\}\subset\Sigma_{m+2}(A),\ \ \textup{and}$$
    \[\{0\}\times\Gamma_1^n\times\{0\}\subset\Sigma_{n+2}(A), \ \ \forall n\geq 0, n\ne m.\]
   \end{itemize}
    \end{enumerate}
\end{thm}

\begin{rmk}
Let  $T\sim\{X_n\}_{n=1}^\infty$ with $0\in\essran X_1$. By Theorem \ref{T:ApprUnitEqui} (iii), there exists $W$ such that $W\cong_a T$ a.s.
It follows that $$\bP(T\lhd A)=\begin{cases}
1,& W\lhd A,\\
0,& W\ntriangleleft A,
\end{cases}\ \ \textup{and} \ \ \ \bP(A\lhd T)=\begin{cases}
1,& A\lhd W,\\
0,& A\ntriangleleft W.
\end{cases}$$
On the other hand, if $0\notin\essran X_1$, then it is possible that there exists an operator $A$ such that (Example \ref{E:CounterEg})
\begin{equation}\label{E:inzeroone}0<\bP(T\lhd A)<1.
\end{equation}
\end{rmk}

  To prove Theorem \ref{T:UnilateralAUE}, we need to make some preparation.

\begin{lem}\label{L:Homomor}
If $A,B$ are two operators on Hilbert spaces, then $A\lhd B$ if and only if there exists a unital $*$-homomorphism $\varphi: C^*(B)\rightarrow C^*(A)$ such that $\varphi(B)=A$.
\end{lem}

\begin{proof}
If $\varphi: C^*(B)\rightarrow C^*(A)$ is a $*$-homomorphism such that $\varphi(B)=A$, then for each polynomial $p(\cdot,\cdot)$ in two free variables, $\varphi(p(B^*,B))=p(A^*,A)$. Since a $*$-homomorphism is always a contraction, we have $\|p(A^*,A)\|\leq \|p(B^*,B)\|$. This proves the sufficiency.
Conversely, since  operators of the form $p(B^*,B)$ are dense in $C^*(B)$,  if $A\lhd B$, then   the map $p(B^*,B)\mapsto p(A^*,A)$ extends to a desired unital $*$-homomorphism from $ C^*(B)$ to $C^*(A)$.
\end{proof}

\noindent In particular, if $A\lhd B$, then $\varphi(|B|)=|A|$, where $\varphi$ is the map from above.   Moreover, $\varphi$ maps invertible operators to invertible ones. Thus we have the following corollary which will be used in the proof of Theorem \ref{T:UnilateralAUE}.

\begin{cor}\label{L:SpecDomina}
If $A,B$ are two operators and $A\lhd B$, then $\sigma(A)\subset\sigma(B)$ and $\sigma(|A|)\subset\sigma(|B|)$.
\end{cor}

\begin{rmk}
The relation $\lhd$ has something to do with the notion of operator-valued spectrum \cite{Hadwin, Herr89}. The {\it operator-valued spectrum} $\Sigma(A)$ of $A\in\BH$
is the set of all those operators $Z$ acting on some separable Hilbert space such that $A$ is the norm limit of a
sequence $\{A_n\}$  in $\BH$ with $A_n\cong Z\oplus B_n$ (for suitable operators $B_n$).
In terms of the operator-valued spectrum, Hadwin \cite{Hadwin} characterized approximate unitary equivalence of operators. Note that if $Z\in\Sigma(A)$, then
$Z\lhd A$. In general, the converse does not hold. However, if $C^*(A)\cap\KH=\{0\}$, then by Voiculescu's theorem (\cite{Voicu} or Theorem 4.21 in \cite{Herr89}), an operator $Z$ lies in $\Sigma(A)$ if and only if $Z\lhd A$. Thus, by Theorem \ref{T:compacts}, if $T\sim\{X_n\}$ with $0\in\essran X_1$, then almost surely an operator $Z$ satisfies $Z\lhd T$ if and only if $Z\in\Sigma(T)$.
\end{rmk}

\begin{lem}\label{L:truncated}
Let $\{e_i\}_{i=1}^\infty$ be an orthonormal basis of $\cH$ and $A,B\in\BH$ be defined by
$$Ae_i=\lambda_i e_{i+1}, \quad  Be_i=\mu_i e_{i+1},\qquad \forall i\geq 1,$$
where $\{\lambda_i\}_{i=1}^\infty,$ $\{\mu_i\}_{i=1}^\infty$ are bounded sequences of nonnegative real numbers.
Assume that $n\geq 1$ and $p(x,y)$ is a polynomial in two free variables $x,y$ with degree $l$.
If $\lambda_i=\mu_i$ for $1\leq i\leq n+l$, then $p(A^*,A)Q_n= p(B^*,B)Q_n$, where $Q_n$ is the orthogonal
projection of $\cH$ onto $\vee\{e_i: 1\leq i\leq n\}$.
\end{lem}

\noindent The proof of the above lemma, although seemingly technical, is an easy verification, hence omitted.

\begin{lem}\label{C:URSDomiUniWS}
Let $A$ be a unilateral weighted shift with nonnegative weights and $B$ be a unilateral weighted shift with nonnegative weights $\{b_i\}_{i=1}^\infty$. If
 $(0,b_1,b_2,\cdots,b_s)$ $\in$ $\Sigma_{s+1}(A)$ for  $s\geq 1$, then $B\lhd A$ .
\end{lem}

\begin{proof}
Assume that $\{e_i\}_{i=1}^{\infty}$ is an orthonormal basis of $\cH$ and $Be_i=b_i e_i$ for $i\geq 1$.
Assume that $\{f_i\}_{i=1}^{\infty}$ is an orthonormal basis of $\cK$ and $Af_i=a_i f_{i+1}$ for $i\geq 1$.
Fix a polynomial $p(x,y)$ in two free variables and assume that degree $p(x,y)=l$.
It suffices to prove
$\|p(B^*,B)\|\leq \|p(A^*,A)\|.$
For each $n\geq 1$, denote by $Q_n$ the orthogonal projection of $\cK$ onto $\vee\{e_i: 1\leq i\leq n\}$.
 It suffices to prove that
$\|p(B^*,B)Q_n\|\leq \|p(A^*,A)\|$ for each $n\geq 1$.
 Fix  $n\geq 1$ and set $s=n+l$. By the hypothesis, $(0,b_1,b_2,\cdots,b_s)\in\Sigma_{s+1}(A)$. For any $\eps>0$, we can find $m_0\in\bN$ such that
\[\|(0,b_1,\cdots,b_s)-(a_j)_{j=m_0}^{m_0+s}\|_{\max}<\eps.\]
Then there exists a finite-rank operator $K$ on $\cK$ with $\|K\|<\eps$ such that $A+K$ is a
unilateral weighted shift satisfying $(A+K)f_j=a_j'f_{j+1}$ for $j\geq 1$, where
\[a_j'=\begin{cases}
0,& j=m_0,\\
b_{j-m_0},& m_0+1\leq j\leq m_0+s,\\
a_j,& j<m_0\ \textup{or}\ j>m_0+s.
\end{cases}\]
Then $A+K=A_1\oplus A_2$, where $A_1$ is a truncated weighted shift and $A_2$ is a unilateral weighted shift on $\widetilde{\cK}=\vee\{f_j: j\geq m_0+1\}$ with weights
$$(b_1,\cdots,b_{s}, a_{m_0+s+1},  a_{m_0+s+2},\cdots ). $$ By Lemma \ref{L:truncated}, we have
$p(B^*,B)Q_n\cong p(A_2^*,A_2)Q_n'$, where $Q_n'$ is the orthogonal projection of $\widetilde{\cK}$ onto $\vee\{f_j: m_0+1\leq j\leq m_0+n\}$. It follows that
\begin{align*}
\|p(B^*,B)Q_n\|&=\|p(A_2^*,A_2)Q_n'\|\leq \|p(A_2^*,A_2)\| \\
&\leq \|p(A^*+K^*,A+K)\|.
\end{align*}
Since $\eps>0$ is arbitrary, the proof is complete.
\end{proof}

\begin{lem}\label{P:NSpectrum}
Let  $T\sim\{X_n\}_{n=1}^\infty$ and $\Gamma=\essran X_1$. Then, almost surely, $\Sigma_n(T)=\Gamma^n$ for  $n\geq 1$.
\end{lem}

\begin{proof}
For each $n\geq 1$ and  $\omega\in\Omega$, note that
$$\Sigma_n(T(\omega))=\text{closure of}\ \{(X_{k+i}(\omega))_{i=1}^n: k\geq 0 \}.$$ For each $i\geq 1$,   $\ran X_i\subset\Gamma$ a.s.
Thus $\Sigma_n(T)\subset\Gamma^n$ a.s.
Conversely,
choose a denumerable dense subset $\{\alpha_i: i\in\Lambda\}$ of $\essran X_1$ and set $$\Omega'=\bigcap_{n\in\bN}\bigcap_{(i_1,\cdots,i_n)\subset
\Lambda}\Big[\liminf_{k\to\infty}\|(X_{k+j})_{j=1}^n-(\alpha_{i_j})_{j=1}^n\|_{\max}=0\Big].$$  By Lemma \ref{L:constantCase}, we have $\bP(\Omega')=1$.
 Fix  $n\geq 1$ and  $\omega\in\Omega'$. It suffices to prove that $\Gamma^n\subset\Sigma_n(T(\omega))$. Denote $\lambda_j=X_j(\omega)$ for $j\geq 1$. Thus $T(\omega)$ is a unilateral weighted shift with weights $\{\lambda_j\}$ and
$\overline{\{(\lambda_{k+j})_{j=1}^n: k\geq 0\}}=\Sigma_n(T(\omega))$.
Choose $i_1,\cdots,i_n\in\Lambda$. By the definition of $\Omega'$, we have
\[\liminf_{k\to\infty}\|(\lambda_{k+j})_{j=1}^n-(\alpha_{i_j})_{j=1}^n\|_{\max}=0.\]
Thus $(\alpha_{i_j})_{j=1}^n\in\text{closure of}\ \{(\lambda_{k+j})_{j=1}^n: k\geq 0\}$. It follows that  $(\alpha_{i_j})_{j=1}^n\in\Sigma_n(T(\omega))$. Since $\{\alpha_i:i\in\Lambda\}$ is dense, we obtain $\Gamma^n\subset\Sigma_n(T(\omega))$. This ends the proof.
\end{proof}

\begin{proof}[Proof of Theorem \ref{T:UnilateralAUE} (i)]
If $\Sigma_1(A)\subset\Gamma$, it follows that $a_i\in\Gamma$ for $i\geq 1$. Note that $0\in\Gamma$. Then $(0, a_1,\cdots,a_s)\in\Gamma^{s+1}$ for  $s\geq 1$.
In view of Lemma \ref{P:NSpectrum}, for each $s\geq 1$, we have $(0, a_1,\cdots,a_s)\in\Sigma_{s+1}(T)$ a.s.
By Lemma \ref{C:URSDomiUniWS}, we obtain $A\lhd T$ a.s.
Assume that $\Sigma_1(A)\nsubseteq\Gamma$. In view of Lemma \ref{P:NSpectrum}, there exists $\Omega'\subset\Omega$ with $\bP(\Omega')=1$ such that $\Sigma_1(T(\omega))=\Gamma$ for  $\omega\in\Omega'$.
Thus  $\Sigma_1(A)\nsubseteq\Sigma_1(T(\omega))$ for  $\omega\in\Omega'$.  By Corollary \ref{L:SpecDomina} and Remark \ref{R:Modulus}, $A\ntriangleleft T(\omega)$ for $\omega\in\Omega'$. This ends the proof.
\end{proof}

\begin{cor}
Let  $T\sim\{X_n\}_{n=1}^\infty$ and
$A$ be a truncated weighted shift.
If $0\in\essran X_{1}$,
then $A\lhd T$ a.s. if and only if $\sigma(|A|)\subset\essran X_{1}$.
\end{cor}

\begin{proof}
Note that the zero operator  on an infinite dimensional Hilbert space is a unilateral weighted shift and, by Theorem \ref{T:UnilateralAUE} (i), $0\lhd T$ a.s. Then $A\lhd T$ a.s. if and only if $E=A\oplus 0\lhd T$ a.s.
Since $E$ is also a unilateral weighted shift, using Theorem \ref{T:UnilateralAUE} (i) again, we deduce that
$A\lhd T$ a.s. if and only if $\Sigma_1(E)\subset\essran X_1$, that is, $\sigma(|A|)\subset\essran X_1$.
\end{proof}


\begin{cor}\label{C:extreme}
Let  $T\sim\{X_n\}_{n=1}^\infty$ with $\essran X_1=[0,1]$.
If $A$ is a deterministic unilateral weighted shift, then $A\lhd T$ a.s. if and only if $\|A\|\leq 1$.
\end{cor}

\noindent The above corollary should be  compared with Corollary \ref{C:bi-extreme}.

\begin{lem}\label{L:AUE}
Let $A_i\in\B(\cH_i)$ be irreducible, $i=1,2$. If $\cK(\cH_2)\subset C^*(A_2)$, then $A_1\lhd A_2$ if and only if
either $A_1\cong A_2$ or $A_1\oplus A_2\cong_a A_2$.
\end{lem}

\begin{proof}
The sufficiency is obvious. For the necessity,
we choose a unital $*$-homomorphism $\varphi: C^*(A_2)\rightarrow C^*(A_1)$ such that $\varphi(A_2)=A_1$.
 If $\varphi|_{\cK(\cH_2)}\ne 0$, then by Corollary 5.41 in \cite{Douglas}, there exists unitary $V:\cH_1\rightarrow\cH_2$ such that $\varphi(Z)=V^*ZV$ for $Z\in C^*(A_2)$. Thus $V^*A_2V=A_1$, that is, $A_1\cong A_2$.
 If $\varphi|_{\cK(\cH_2)}=0$, then $\rank~\big(\varphi(Z)\oplus Z\big)=\rank~ Z$ for  $Z\in C^*(A_2)$. By Theorem II.5.8 in \cite{David}, this implies that
$\varphi\oplus \textup{id}\cong_a\textup{id}$, where $\textup{id}(\cdot)$ is the identity representation of $C^*(A_2)$. It follows that $A_1 \oplus A_2\cong_a A_2$.
\end{proof}

\begin{lem}\label{P:UnilLeftInvert}
Let  $T\sim\{X_n\}_{n=1}^\infty$  with $0<r<R.$ If $A$ is an irreducible operator with $0\in\sigma_p(A)\cup\sigma_p(A^*)$, then $A\ntriangleleft T$ a.s.
\end{lem}

\begin{proof}
In view of Theorem \ref{T:compacts}, there exists $\Omega'\subset\Omega$ with $\bP(\Omega')=1$ such that
$\KH\subset C^*(T(\omega))$ for $\omega\in\Omega'$. Hence $T(\omega)$ is irreducible for $\omega\in\Omega'$.
By Lemma \ref{T:SpectralPicture}, we assume that $\ind T(\omega)=-1$ for $\omega\in\Omega'$.
In view of Theorem \ref{T:ApprUnitEqui} (i), we have $A\ncong T$ a.s. Thus we  assume that $A\ncong T(\omega)$ for  $\omega\in\Omega'$.
Fix  $\omega\in\Omega'$. If $A\lhd T(\omega)$, then by Lemma \ref{L:AUE}, we have either $A\cong T(\omega)$ or $T(\omega)\oplus A\cong_a T(\omega)$. By what we have just assumed, the former does not hold. So $T(\omega)\oplus A\cong_a T(\omega)$.
Thus $T(\omega)\oplus A, T(\omega)$ are both Fredholm and, by Theorem 1.13 (iv) of \cite{Herr89}, we have
$\dim\ker A+\dim\ker T(\omega)=\dim\ker T(\omega)$ and $\dim\ker A^*+\dim\ker T(\omega)^*=\dim\ker T(\omega)^*.$ Since $\max\{\dim\ker T(\omega),\dim\ker T(\omega)^*\}<\infty$, it follows that $\dim\ker A=\dim\ker A^*=0,$ a contradiction. Thus we obtain
$A \ntriangleleft T$ a.s.\end{proof}

\begin{proof}[Proof of Theorem \ref{T:UnilateralAUE} (ii)] The proof is divided into two cases.
 {\it Case 1.} $\card ~\Gamma=1$.
 Assume that $\Gamma=\{\lambda\}$. Thus $T$ is almost surely a unilateral weighted shift with weights $(\lambda,\lambda,\lambda,\cdots)$.
If $\Sigma_1(A)=\Gamma$, then the weights of $A$ are $(\lambda,\lambda,\lambda,\cdots)$, which implies that $A\lhd T$ a.s.
If $\Sigma_1(A)\ne \Gamma$, then it follows that $\sigma(|A|)\nsubseteq \Gamma=\sigma(|T|)$ a.s. By Corollary \ref{L:SpecDomina}, we obtain $A\ntriangleleft T$ a.s.
 {\it Case 2.} $\card ~\Gamma>1$.
 Note that $0\in\sigma_p(A^*)$. If $A$ is irreducible, then the result follows  from Lemma \ref{P:UnilLeftInvert},
It remains to deal with the case that $A$ is reducible.
In this case, $A$ is the direct sum of an irreducible truncated weighted shift $A_1$ and a unilateral weighted shift.
Thus $0\in\sigma_p(A_1)$ and $A_1\lhd A$. By Lemma \ref{P:UnilLeftInvert}, $A_1 \ntriangleleft T$ a.s. Therefore we conclude that $A \ntriangleleft T$ a.s.
\end{proof}

  Let $\{A_i: 1\leq i\leq n\}$ be a commuting family of normal operators on $\cH$.
Denote by $C^*(A_1,A_2,\cdots,A_n)$ the $C^*$-algebra generated by
$A_1,A_2,\cdots, A_n$ and the identity $I$. The {\it joint spectrum} of the $n$-tuple
$(A_1,A_2,\cdots, A_n)$ is defined as the set of $n$-tuples of scalars
$(\lambda_1,\lambda_2,\cdots,\lambda_n)$ such that the ideal of
$C^*(A_1,A_2,\cdots,A_n)$ generated by $A_1-\lambda_1, A_2-\lambda_2,
A_n-\lambda_n$ is different from $C^*(A_1,A_2,\cdots,A_n)$ (Definition
3.1.13, \cite{Hormander90}). We let $\sigma(A_1,A_2,\cdots,A_n)$ denote the joint spectrum
of $(A_1,A_2,\cdots,A_n)$.

\begin{lem}\label{L:JointSpec}
For each $1\leq k\leq n$, we assume that
$$A_k=\textup{diag}\{a_1^{(k)}, a_2^{(k)}, a_3^{(k)},\cdots \} \text{~~~and~~~}
B_k=\textup{diag}\{b_1^{(k)}, b_2^{(k)}, b_3^{(k)},\cdots \}$$
with respect to some orthonormal basis $\{e_i\}_{i=1}^\infty$ of $\cH$.
\begin{enumerate}
\item[(i)] $\sigma(A_1,\cdots,A_n)=\text{closure of}\ \{(a_i^{(1)}, \cdots, a_i^{(n)}): i\geq 1 \}$, where the closure is taken
in the usual topology on $\bC^n$.
\item[(ii)] If there exists a unital $*$-homomorphism $$\varphi: C^*(A_1,\cdots,A_n)\rightarrow C^*(B_1,\cdots,B_n)$$ such that $\varphi(A_i)=B_i$ for $1\leq i\leq n$, then $\sigma(B_1,\cdots,B_n)\subset\sigma(A_1,\cdots,A_n)$.
\end{enumerate}
\end{lem}

\begin{proof}
By Theorem 3.1.14 in \cite{Hormander90}, we have
\begin{align*}
\sigma(A_1,\cdots, A_n)=\{(\nu(A_1),\cdots, \nu(A_n)): &~\nu ~\mbox{is a multiplicative linear}\\
 &~~~\mbox{functional on}~ C^*(A_1,\cdots, A_n)\}.
\end{align*}

  (i) For $i\geq 0$, define $\nu_i(E)=\alpha_i$ if $E\in C^*(A_1,\cdots, A_n)$ and
$E=\textrm{diag}\{\alpha_1, \alpha_2, \alpha_3,\cdots \}$ with respect to $\{e_i\}$. Then
each $\nu_i$ is a multiplicative linear functional on the $C^*$-algebra $C^*(A_1,\cdots, A_n)$.
Moreover, $\{\nu_i: i\geq 1\}$ is dense in the maximal ideal space of
$C^*(A_1,\cdots, A_n)$. Then we obtain
\begin{eqnarray*}
\sigma(A_1,\cdots, A_n)&=&\text{closure of}\ \{(\nu_i(A_1),\cdots, \nu_i(A_n)): i\geq 1\}\\
&=&\text{closure of}\ \{( a_i^{(1)},\cdots,a_i^{(n)}): i\geq
1\}.
\end{eqnarray*}

  (ii) Note that if $\nu$ is a multiplicative linear functional on $C^*(B_1,\cdots, B_n)$, then $\nu\circ\varphi$ is a multiplicative linear functional on $C^*(A_1,\cdots, A_n)$ and $\nu(B_i)=\nu\circ\varphi(A_i)$ for $1\leq i\leq n$.
Then it follows  that $\sigma(B_1,\cdots,B_n)\subset\sigma(A_1,\cdots,A_n)$.
\end{proof}

\begin{lem}\label{P:TruncUniWeigh}
Let $A$ be a unilateral weighted shift with nonnegative weights $\{a_i\}_{i=1}^{\infty}$ and $B$ be a truncated weighted shift with positive weights $\{b_i\}_{i=1}^s$. Then $B\lhd A$ if and only if either $(b_1,b_2,\cdots,b_s,0)=(a_1,a_2,\cdots, a_s,a_{s+1})$ or
$(0,b_1,b_2,\cdots,b_s,0)\in\Sigma_{s+2}(A).$
\end{lem}

\begin{proof}
For convenience, we only deal with the case that $s=2$. The proof for the general case is similar.
We assume that $A=\sum_{i=1}^\infty a_i e_{i+1}\otimes e_i$ and $ B=b_1f_2\otimes f_1+ b_2f_3\otimes f_2,$
where $\{e_i\}_{i=1}^\infty$ is an orthonormal basis of $\cH$ and $\{f_1,f_2,f_3\}$ is an orthonormal basis of $\bC^3$.
For sufficiency, we fix a polynomial $p(z,w)$ in free variables $z,w$.
 If $(b_1,b_2,0)=(a_1,a_{2}, a_{3})$, then $\vee\{e_{i}: 1\leq i\leq 3\}$ reduces $A$ and the restriction of $A$ to it is unitarily equivalent to $B$. Thus $\|p(B,B^*)\|\leq \|p(A,A^*)\|$.
 We assume that $(0,b_1,b_2,0)\in\Sigma_{4}(A).$
Then for any $\eps>0$, there exists $i\in\bN$ such that
$$\max\{|a_i-0|, |a_{i+1}-b_1|, |a_{i+2}-b_2|, |a_{i+3}-0|\}<\eps.$$
Set \begin{align*}
K_\eps&=-a_i e_{i+1}\otimes e_i-(a_{i+1}-b_1)e_{i+2}\otimes e_{i+1}\\
&\quad -(a_{i+2}-b_2)e_{i+3}\otimes e_{i+2}-a_{i+3} e_{i+4}\otimes e_{i+3}.
\end{align*} Then $\|K_\eps\|<\eps$, $\cM_\eps=\vee\{e_{i+1}, e_{i+2}, e_{i+3}\}$ reduces $A+K_\eps$ and
$(A+K_\eps)|_{\cM_\eps}\cong B.$ Then $\|p(B,B^*)\|\leq\|p(A+K_\eps,A^*+K_\eps^*)\|.$
Since $\eps>0$ is arbitrary, we conclude that $\|p(B,B^*)\|\leq\|p(A,A^*)\|$.

  Now we show the necessity. We choose   $\varphi: C^*(A)\rightarrow C^*(B)$ according to Lemma \ref{L:Homomor}
Then $\varphi(|A^*|)=|B^*|$ and $\varphi(|A^k|)=|B^k|$ for $k=1,2,3.$
Note that $(|B^*|,|B|,|B^2|,|B^3|)$ and $(|A^*|,|A|,|A^2|,|A^3|)$ are two commuting families of diagonal operators. By Lemma \ref{L:JointSpec} (ii), $\sigma(|B^*|,|B|,|B^2|,|B^3|)\subset \sigma(|A^*|,|A|,|A^2|,|A^3|)$.
In view of Lemma \ref{L:JointSpec} (i), a calculation shows that $$(0,b_1,b_1b_2,0)\in\sigma(|B^*|,|B|,|B^2|,|B^3|)$$ and $\sigma(|A^*|,|A|,|A^2|,|A^3|)$ is the closure of
\[\{(0,a_1,a_1a_2, a_1a_2a_3)\}\cup\{(a_i, a_{i+1},a_{i+1}a_{i+2},a_{i+1}a_{i+2}a_{i+3}): i\geq 1\}.\] Then either
 $(0,b_1,b_1b_2,0)= (0,a_1,a_1a_2, a_1a_2a_3)$ or $$(0,b_1,b_1b_2,0)\in\text{closure of}\ \{(a_i, a_{i+1},a_{i+1}a_{i+2},a_{i+1}a_{i+2}a_{i+3}): i\geq 1\}.$$
If the former holds, then $(b_1,b_2,0)=(a_1,a_2,a_{3})$. Next we assume the latter holds.

  For any $\eps>0$, set
$$\delta=\frac{\eps\cdot\min\{1,b_1,b_1b_2,b_1^2b_2\}}{3(1+\|A\|)(1+\|A\|^2)(1+\|A\|\cdot\|B\|)}.$$
Thus there exists $i\in\bZ$ such that
$$\max\{|a_i-0|, |a_{i+1}-b_1|, |a_{i+1}a_{i+2}-b_1b_2|, |a_{i+1}a_{i+2}a_{i+3}-0|  \}<\delta.$$
We shall prove that
\begin{equation}\label{(4.2)}\max\{|a_i-0|, |a_{i+1}-b_1|, |a_{i+2}-b_2|, |a_{i+3}-0|\}<\eps. \end{equation}
Compute to see
\begin{align*}
|a_{i+2}-b_2|&=\frac{|b_1a_{i+2}-b_1b_2|}{b_1}\\
&\leq \frac{|b_1a_{i+2}-a_{i+1}a_{i+2}|}{b_1}+\frac{|a_{i+1}a_{i+2}-b_1b_2|}{b_1}\\
&\leq \frac{ \|A\|\delta}{b_1}+\frac{\delta}{b_1}=\frac{(\|A\|+1)\delta}{b_1}<\eps
\end{align*} and
\begin{align*}
|a_{i+3}|&=\frac{b_1b_2a_{i+3}}{b_1b_2}\\
&\leq \frac{|b_1b_2a_{i+3}-b_1a_{i+2}a_{i+3}|}{b_1b_2}+\frac{|b_1a_{i+2}a_{i+3}-a_{i+1}a_{i+2}a_{i+3}|}{b_1b_2}\\
&\quad+\frac{a_{i+1}a_{i+2}a_{i+3}}{b_1b_2}\\
&\leq\frac{|b_2-a_{i+2}|\cdot\|A\|\cdot\|B\|}{b_1b_2}+\frac{|b_1-a_{i+1}|\cdot\|A\|^2}{b_1b_2}+\frac{\delta}{b_1b_2} \\
&<\frac{(\|A\|+1)\delta\cdot\|A\|\cdot\|B\|}{b_1^2b_2}+\frac{\delta\|A\|^2}{b_1b_2}+\frac{\delta}{b_1b_2}\\
&\leq \frac{\eps}{3}+\frac{\eps}{3}+\frac{\eps}{3}=\eps.\end{align*}
Therefore we obtain (\ref{(4.2)}).
Since $\eps>0$ is arbitrary, we deduce that $(0,b_1,b_2,0)\in\Sigma_{4}(A)$.
\end{proof}

  The following is straightforward now; it is needed for the proof of Theorem \ref{T:UnilateralAUE}.

\begin{cor}\label{C:TruncUniWeigh}
Let $A$ be a unilateral weighted shift with nonnegative weights $\{a_i\}_{i=1}^{\infty}$. Then $0\lhd A$ if and only if either $a_1=0$ or $(0,0)\in\Sigma_{2}(A).$
\end{cor}

\begin{lem}\label{P:URSDomiUniWS}
Let $A$ be a unilateral weighted shift with nonnegative weights $\{a_i\}_{i=1}^\infty$ and $B$ be a unilateral weighted shift with positive weights $\{b_i\}_{i=1}^\infty$. Then $B\lhd A$ if and only if either $a_i=b_i$ for  $i\geq 1$ or
 $(0,b_1,b_2,\cdots,b_s)\in\Sigma_{s+1}(A)$ for  $s\geq 1$.
\end{lem}

\begin{proof}
 If $a_i=b_i$ for  $i\geq 1$, then  $A\cong B$ and $B\lhd A$.
If $(0,b_1,b_2,\cdots,b_s)\in\Sigma_{s+1}(A)$ for  $s\geq 1$, then by Lemma \ref{C:URSDomiUniWS}, $B\lhd A$.
 On the other hand, the proof of necessity is similar to that of Lemma \ref{P:TruncUniWeigh} and is just somehow notationally more complicated.
\end{proof}

\begin{proof}[Proof of Theorem \ref{T:UnilateralAUE} (iii)]
By Lemma \ref{P:NSpectrum}, there exists $\Omega'\subset\Omega$ with $\bP(\Omega')=1$ such that  $\Sigma_n(T(\omega))=\Gamma^n$ for   $n\geq 1$ and $\omega\in\Omega'$.
If $\{0\}\times\Gamma^n\subset\Sigma_{n+1}(A)$ for  $n\geq 1$,
then for each $\omega\in\Omega'$,
$$(0,X_1(\omega),\cdots, X_n(\omega))\subset\{0\}\times\Gamma^n\subset\Sigma_{n+1}(A),\ \forall n\geq 1.$$ By Lemma \ref{C:URSDomiUniWS}, we have $T(\omega)\lhd A$. Furthermore, $T\lhd A$ a.s.
Since $T\lhd A$ a.s., we can choose $\Omega''\subset\Omega'$ with $\bP(\Omega'')=1$ such that $T(\omega)\lhd A$ for $\omega\in\Omega''$.
Assume that $A$ is a unilateral weighted shift with weights $\{a_i\}_{i\in\bN}$. The rest of the proof is divided into two cases.

  {\it Case 1.} $0\notin\Gamma$.
 For each $\omega\in\Omega''$, since $0\notin\Gamma$, $T(\omega)$ is a unilateral weighted shift with positive weights;
 since $T(\omega)\lhd A$, by Lemma \ref{P:URSDomiUniWS}, we have either $X_i(\omega)=a_i$ for  $i\geq1$ or $(0,X_1(\omega),\cdots,X_s(\omega))\in\Sigma_{s+1}(A)$ for  $s\geq 1$.

  {\it Claim 1.} There exists no $\omega\in\Omega''$ such that $X_i(\omega)=a_i$ for  $i\geq 1$.

  If not, we assume that $\omega_0\in\Omega''$ and $X_i(\omega_0)=a_i$ for  $i\geq 1$. Then $a_i\geq\min\Gamma>0$ for  $i$.
So $A$ is left-invertible and $(0,X_1(\omega),\cdots,X_n(\omega))\notin\Sigma_{n+1}(A)$ for  $n\geq 1$ and  $\omega\in\Omega''$. Hence $X_i(\omega)=a_i$ for  $i\geq 1$ and  $\omega\in\Omega''$. Noting that $\{X_i\}$ are \iid random variables, it follows that $\Gamma=\essran X_1=\{a_1\}$, contradicting the hypothesis that $\card~\Gamma>1$.
This proves Claim 1, which implies that $(0,X_1(\omega),\cdots,X_n(\omega))\in\Sigma_{n+1}(A)$ for  $n\geq 1$ and  $\omega\in\Omega''$. Fix  $n\geq 1$. Arbitrarily choose $(\alpha_{i})_{i=1}^n\subset\Gamma$. Then, given $\eps>0$, we have
\[\bP(|X_i-\alpha_i|<\eps, 1\leq i\leq n)=\Pi_{i=1}^n\bP(|X_i-\alpha_i|<\eps)>0.\]
Thus there exists $\omega\in\Omega''$ such that $ |X_i(\omega)-\alpha_i|<\eps$ for $1\leq i\leq n$. Since $\eps$ is arbitrary, we deduce that $(0,\alpha_1,\cdots\alpha_n)\in\Sigma_{n+1}(A)$.
We conclude that $\{0\}\times\Gamma^n\subset\Sigma_{n+1}(A)$ for  $n\geq 1$.

  {\it Case 2.} $0$ is an accumulation point of $\Gamma$.
Choose a countable dense subset $\Delta=\{\alpha_i: i\in\Lambda\}$ of $\Gamma\setminus\{0\}$. Since $0$ is an accumulation point of $\Gamma$, $\Delta$ is also dense in $\Gamma$. There are countably many finite tuples with entries in $\Delta$, denoted as $\eta_1,\eta_2,\eta_3,\cdots$.
For each $i\geq 1$, we assume that $\eta_i=(\alpha_{i,1}, \cdots, \alpha_{i,{n_i}})$, and write $E_{i}$ to denote the truncated weighted shift on $\bC^{n_i+1}$ with weights $(\alpha_{i,1}, \cdots, \alpha_{i,{n_i}})$.
Set $E=\oplus_{i=1}^\infty E_{i}$. By the proof of Theorem \ref{T:ApprUnitEqui} (iii),  $T\cong_a E^{(\infty)}$ a.s. It follows that $E\lhd T$ a.s. Thus $E\lhd A$.
 Note that $E=\oplus_{i=1}^\infty E_{i}$, so we obtain $E_i\lhd A$ for  $i\geq 1$.
In view of Lemma \ref{P:TruncUniWeigh}, for each $i\geq 1$, we have either $(\alpha_{i,1}, \cdots, \alpha_{i, {n_i}},0)=(a_1,\cdots,a_{n_i+1})$ or
$(0, \alpha_{i,1}, \cdots, \alpha_{i,{n_i}},0)\in\Sigma_{n_i+2}(A).$

  {\it Claim 2.} $(0, \alpha_{i,1}, \cdots, \alpha_{i,{n_i}},0)\in\Sigma_{n_i+2}(A)$ for  $i\geq 1$.

  If there exists some $j$ such that  $(\alpha_{j,1}, \cdots, \alpha_{j,{n_j}},0)=(a_1,\cdots,a_{n_j+1})$, then $$(\alpha_{i,1}, \cdots, \alpha_{i,{n_i}},0)\ne (a_1,\cdots,a_{n_i+1})$$ for any $i$ with $i\ne j$, which implies that
$(0, \alpha_{i,1}, \cdots, \alpha_{i,{n_i}},0)\in\Sigma_{n_i+2}(A)$ for $i\ne j.$
We choose $\{\eps_{k}\}_{k=1}^\infty\subset\Delta$ with $\eps_k\rightarrow 0$.
Thus $(0, \alpha_{j,1}, \cdots, \alpha_{j,{n_j}},\eps_k, 0)\in\Sigma_{n_j+3}(A)$, which implies that
$(0, \alpha_{j,1}, \cdots, \alpha_{j,{n_j}},\eps_k)\in\Sigma_{n_j+2}(A)$.
Letting $k\rightarrow\infty$  proves the claim.
Since $\{\eta_i: i\geq 1\}$ are all finite tuples with entries in $\Delta$ which is dense in $\Gamma$, it follows from Claim 2 that $\{0\}\times\Gamma^n\subset\Sigma_{n+1}(A)$ for  $n\geq 1$.\end{proof}

\begin{proof}[Proof of Theorem \ref{T:UnilateralAUE} (iv)]
When $\Gamma=\{0\}$, by Corollary \ref{C:TruncUniWeigh}, the result is clear. Next we assume $\card~\Gamma>1$.
 Choose an at most countable dense subset $\Delta=\{\alpha_i: i\in\Lambda\}$ of $\essran X_1$.
Since $0$ is an isolated point of $\Gamma$, it follows that $0\in\Delta$ and $\Delta\setminus\{0\}$ is dense in $\Gamma_1$.
 Suppose $\eta_1,\eta_2,\eta_3,\cdots$ are all finite tuples with entries in $\Delta$.
For each $i\geq 1$, we assume that $\eta_i=(\alpha_{i,1}, \cdots, \alpha_{i,{n_i}})$, and write $E_{i}$ to denote the truncated weighted shift on $\bC^{n_i+1}$ with weights $(\alpha_{i,1}, \cdots, \alpha_{i,{n_i}})$.
Set $E=\oplus_{i=1}^\infty E_{i}$. By the proof of Theorem \ref{T:ApprUnitEqui} (iii), $T\cong_a E^{(\infty)}$ a.s. Thus $T\lhd A$ a.s. if and only if $E\lhd A$.
 There are infinitely many $i$'s such that $(0,0)=(a_{i,s},a_{i,s+1})$ for some
$s$ with $1\leq s \leq n_i-1$. Thus $E$ can be rewritten as
\[E=0^{(\infty)}\oplus (\oplus_{j=1}^\infty F_j^{(k_j)}),\]
where each $F_j$ is a truncated weighted shift of order $\geq 2$ with weights in $\Delta\setminus\{0\}$.
For each $j\geq 1$, we assume that $F_j$'s weights are $\xi_j=(\beta_{j,1}, \cdots, \beta_{j,{m_j}})$.
Then   $\xi_1,\xi_2,\xi_3,\cdots$ are all finite tuples with entries in $\Delta\setminus\{0\}$.
 Set $F=0\oplus(\oplus_{i=1}^\infty F_i)$.
Thus   $T\lhd A$ a.s. if and only if $F\lhd A$. Note that the latter
means that
$0\lhd A $ and $F_j\lhd A$ for $ j\geq 1.$
In view of Lemma \ref{P:TruncUniWeigh} and Corollary \ref{C:TruncUniWeigh}, $T\lhd A$ a.s. if and only if
(1) $a_1=0$ or $(0,0)\in\Sigma_2(A)$, and (2) for each $j\geq 1$, either $(\beta_{j,1}, \cdots, \beta_{j,{m_j}},0)=(a_1,\cdots,a_{m_j+1})$ or $(0,\beta_{j,1}, \cdots, \beta_{j,{m_j}},0) \in\Sigma_{m_j+2}(A)$.
 The rest of the proof is divided into two cases.

  {\it Case 1.} $a_1=0$. Then there exists no $j\geq 1$ such that
$$(\beta_{j,1}, \cdots, \beta_{j,{m_j}},0)=(a_1,\cdots,a_{m_j+1}),$$ since each $\beta_{j,i}$ is positive.
In this case, $T\lhd A$ a.s. if and only if
$$(0,\beta_{j,1}, \cdots, \beta_{j,{m_j}},0) \in\Sigma_{m_j+2}(A),\ \ \forall j\geq 1.$$
Since $\{\xi_j: j\geq 1\}$ are all finite tuples with entries in $\Delta\setminus\{0\}$, the latter implies that $\{0\}\times \Gamma_1^{n}\times \{0\}\subset\Sigma_{n+2}(A)$ for  $n\geq 1$.

  {\it Case 2.} $a_1\ne 0$.
 If $(0,\beta_{j,1}, \cdots, \beta_{j,{m_j}},0) \in\Sigma_{m_j+2}(A)$ for  $j\geq 1$ or, equivalently, $\{0\}\times \Gamma_1^{n}\times \{0\}\subset\Sigma_{n+2}(A)$ for  $n\geq 1$, then $T\lhd A$ a.s. if and only $(0,0)\in \Sigma_{2}(A)$.
On the other hand, if there exists $k\geq 1$ such that $(0,\beta_{k,1}, \cdots, \beta_{k,{m_k}},0) \notin\Sigma_{m_k+2}(A)$, then
$$(\beta_{k,1}, \cdots, \beta_{k,{m_k}},0)=(a_1,\cdots,a_{m_k+1});$$ hence
$(\beta_{j,1}, \cdots, \beta_{j,{m_j}},0)=(a_1,\cdots,a_{m_j+1})$ for  $j\geq 1$ with $j\ne k$.
So $$(0,\beta_{j,1}, \cdots, \beta_{j,{m_j}},0) \in\Sigma_{m_j+2}(A),\ \ \forall j\geq 1, j\ne k.$$ In this case, $T\lhd A$ a.s. if and only if $(0,0)\in \Sigma_{2}(A)$,
$$\{0\}\times \Big[\Gamma_1^{m_k}\setminus\{(a_1,\cdots,a_{m_k})\}\Big]\times \{0\}\subset\Sigma_{m_k+2}(A)$$
and $$\{0\}\times \Gamma_1^{n}\times \{0\}\subset\Sigma_{n+2}(A),\ \ \forall n\geq 1, ~n\ne m_k.$$
Thus the proof is complete.
\end{proof}

  Now we apply Theorem \ref{T:UnilateralAUE} to two examples.

\begin{eg}\label{E:ThreeCounterEg}
Let $T\sim \{X_n\}_{n=1}^\infty$  with
$$\bP(X_1=0)=\bP(X_1=1)=\bP(X_1=2)=\frac{1}{3}.$$ Thus $0$ is an isolated point of $\essran X_1$. Denote $\Gamma=\essran X_1$ and $\Gamma_1=\Gamma\setminus\{0\}$.
Thus $\Gamma=\{0,1,2\}$ and $\Gamma_1=\{1,2\}$.

  Assume that $\eta_1,\eta_2,\eta_3,\cdots$ are all finite tuples with entries in $\{1,2\}$. For each $i\geq 1$, assume that $\eta_i=(a_{i,1},\cdots, a_{i,n_i})$ and $W_i$ is a truncated weighted shift with weights $(a_{i,1},\cdots, a_{i,n_i})$. Denote by $W_0$ the zero operator on $\bC$.
 We define four unilateral weighted shifts as  $$A_1=  W_1\oplus W_2 \oplus W_3\oplus \cdots,\quad A_2=W_0\oplus W_1\oplus W_2 \oplus W_3\oplus \cdots $$ and
$$A_3=W_0\oplus W_1\oplus W_0\oplus W_2 \oplus W_3\oplus W_4\oplus \cdots,\ A_4=W_1\oplus W_0\oplus W_2 \oplus W_3\oplus  W_4\oplus\cdots. $$
Then one can check that \begin{enumerate}
\item[(i)] the first one in the weight sequence of $A_1$ is not $0$ and $(0,0)\notin\Sigma_2(A_1)$;
\item[(ii)] the first one in the weight sequence of $A_2$ is $0$ and $\{0\}\times\Gamma_1^n\times\{0\}\subset\Sigma_{n+2}(A_2)$ for  $n\geq 1;$
\item[(iii)] $\{0\}\times\Gamma_1^n\times\{0\}\subset\Sigma_{n+2}(A_3)$ for  $n\geq 0$;
\item[(iv)] if $\{\lambda_i\}_{i=0}^\infty$ is the weight sequence of $A_4$, then $$(\lambda_1,\cdots\lambda_{n_1+1})=(a_{1,1},\cdots, a_{1,n_1},0)\in\Gamma_1^{n_1}\times\{0\},$$ $$\{0\}\times\Big[\Gamma_1^{n_1}\setminus\{(\lambda_1,\cdots\lambda_{n_1})\}\Big]\times\{0\}\subset\Sigma_{n_1+2}(A_4),\ \ \textup{and}$$
    \[\{0\}\times\Gamma_1^n\times\{0\}\subset\Sigma_{n+2}(A_4), \ \ \forall  n\geq 0, n\ne n_1.\]
\end{enumerate}
Thus, by Theorem \ref{T:UnilateralAUE} (iv), $A_1$ does not satisfy any one of (a), (b) and (c). So $T \ntriangleleft A_1$ a.s.
Furthermore, $A_2$, $A_3$ and $A_4$ satisfy respectively (a), (b) and (c) of Theorem \ref{T:UnilateralAUE} (iv). Then almost surely we have
$ T\lhd A_i$ for $i=2,3,4.$ On the other hand, we note that $\Sigma_1(|A_i|)\subset\Gamma$ for $1\leq i\leq 4$. By Theorem \ref{T:UnilateralAUE} (i), we obtain $A_i\lhd T$ a.s. for $1\leq i\leq 4$.
\end{eg}

\begin{eg}\label{E:CounterEg}
Let  $T\sim\{X_n\}_{n=1}^\infty$ with
$$\bP(X_1=1)=\frac{1}{2}=\bP(X_1=2).$$
There exist countably infinitely many finite tuples with entries in $\{1,2\}$, denoted as $\eta_1,\eta_2,\eta_3,\cdots$.
For each $i\geq 1$, assume that $\eta_i=(a_{i,1},\cdots, a_{i,n_i})$ and $A_i$ is a truncated weighted shift with weights $(a_{i,1},\cdots, a_{i,n_i})$.
Denote $\Lambda=\{i\in\bN: a_{i,1}=2\}$ and
set $A=  \oplus_{i\in\Lambda} A_i.$ It is easy to see \begin{equation}\label{(4.3)}\{0\}\times\{2\}\times\{1,2\}^n\subset\Sigma_{n+2}(A), \ \ \forall n\ge 0\end{equation} and \begin{equation}\label{(4.4)}\{0\}\times\{1\}\times\{1,2\}^n\cap\Sigma_{n+2}(A)=\emptyset,\ \ \forall n\ge 0.\end{equation}

  Without loss of generality, we assume that $\ran X_i\subset\{1,2\}$ for  $i\geq 1$.
For each $\omega\in\Omega$, in view of Lemma \ref{P:URSDomiUniWS}, $T(\omega)\lhd A$ if and only if $$(0,X_1(\omega),\cdots,X_s(\omega))\in\Sigma_{s+1}(A),\ \ \forall s\geq 1.$$ In view of (\ref{(4.3)}) and (\ref{(4.4)}),  the latter implies that $X_1(\omega)=2$.
Thus $$\bP(T\lhd A)=\bP(X_1=2)=1/2.$$
\end{eg}

In view of Theorem \ref{T:UnilateralAUE} (ii), when $T$ is a left-invertible random weighted shift, one can not expect a unilateral weighted shift $A$ to satisfy that $A\lhd T$ a.s. So we turn to consider a weaker relation.
A {\it hereditary polynomial} is a polynomials $p(z,w)$ in two free variables of the form $$p(z, w)=\sum_{0\leq i,j\leq n} \alpha_{i,j}z^i{w}^j.$$
We write $A\lhd_h B$ to denote that $\|p(A,A^*)\|\leq \|p(B,B^*)\|$ for all hereditary polynomials $p(z,w)$.

\begin{lem}\label{T:vonNeum-US}
Let  $T\sim\{X_n\}_{n=1}^\infty$ with $r<R=1.$
If $S$ is the unilateral shift, then
\begin{enumerate}
\item[(i)] $T \lhd_h S$ a.s.
\item[(ii)] If $0\in\essran X_1$, then $S \lhd_h T$ a.s.; otherwise, $S \ntriangleleft_h T$ a.s.
\end{enumerate}
\end{lem}

\begin{proof}
(i) By the hypothesis, we have $\|T\|=1$ a.s., that is, $T$ is almost surely a contraction. Thus, almost surely, $T$ can be dilated to an isometry $A$. For example, $A$ can be chosen as
\[A=\begin{bmatrix}
T&&&& \\
\sqrt{I-T^*T}& 0&&\\
&I&0&&\\
&&I&0&\\
&&&\ddots&\ddots
\end{bmatrix}.\] If $p(z,w)$ is a hereditary polynomial, then it is easy to verify that
$p(A,A^*)$ is a dilation of $p(T,T^*)$ and $\|p(T,T^*)\|\leq\|p(A,A^*)\|$. It follows that $T\lhd_h A$.
On the other hand, since $A$ is an isometry, we have $A\lhd S$ a.s. So $T\lhd_h S$ a.s.

  (ii) If $0\in\essran X_1$, then by Theorem \ref{T:UnilateralAUE} (i), we have $S\lhd T$ a.s. Therefore $S\lhd_h T$ a.s.
 We assume that $0\notin\essran X_1$.
For each $n\geq 1$, set $p_n(z,w)=z^n{w}^{n+1}-z^{n+1}{w}^{n+2}$. Then $p_n(T,T^*)$ is a backward unilateral random weighted shift with weights
$\{Y_i\}_{i\geq 1}$, where
\[Y_i=\begin{cases}
0,& 1\leq i\leq n,\\
(X_1\cdots X_n)^2X_{n+1},&  i=n+1,\\
(1-X_{i-n-1}^2)(X_{i-n}\cdots X_{i-1})^2X_{i},& i>n+1.
\end{cases}\]
Also we note that $\|p_n(S,S^*)\|=1$. If $\|p_n(S,S^*)\|\leq \|p_n(T,T^*)\|$, then $\|p_n(T,T^*)\|\geq 1$. Note that $1\geq X_i\geq \min\essran X_1>0$ for  $i\geq 1$. This implies that
$(X_1\cdots X_n)^2X_{n+1}\geq 1$ or, equivalently, $(X_1\cdots X_n)^2X_{n+1}=1$. Then
\begin{align*}
\bP(S\lhd_h T)&\leq \bP(\|p_n(S,S^*)\|\leq \|p_n(T,T^*)\|)\\
&\leq \bP(X_1=\cdots =X_n=1)=\bP(X_1=1)^n.
\end{align*}
The proof is complete as $n\to \infty$.

\end{proof}

\begin{rmk} When $T$ is a left-invertible random weighted shift, it is somewhat difficult (for us) to
determine which unilateral weighted shifts $A$ satisfy $A\lhd_h T$ a.s.
\end{rmk}

\subsection{Bilateral weighted shifts}

This subsection is devoted to describing when a bilateral weighted shift $A$ satisfies $A\lhd T$ or $T\lhd A$, where $T$ is a random weighted shift.   The main result of this subsection is the following

\begin{thm}\label{T:BilateralAUE}
Let  $T\sim\{X_n\}_{n=1}^\infty$
and $A$ be a bilateral weighted shift with nonnegative weights. Then
\begin{enumerate}
\item[(i)] $\bP(A \lhd T)=\begin{cases}
1,& \Sigma_1(A)\subset\Gamma,\\
0,& \Sigma_1(A)\nsubseteq\Gamma,
\end{cases}$ where $\Gamma=\essran X_1.$
    \item[(ii)] If $0\notin\Gamma$, then $T\lhd A$ a.s. if and only if $\{0\}\times\Gamma^n\subset\Sigma_{n+1}(A)$ for  $n\geq 1$.
\item[(iii)] If $0$ is an accumulation point of $\Gamma$, then $T\lhd A$ a.s. if and only if $\Gamma^n\subset\Sigma_n(A)$ for  $n\geq 1$.
    \item[(iv)] If $0$ is an isolated point of $\Gamma$, then $T\lhd A$ a.s. if and only if $\{0\}\times\Gamma_1^{n}\times\{0\}\subset\Sigma_{n+2}(A)$ for  $n\geq 0$, where $\Gamma_1=\Gamma\setminus\{0\}$.
        \end{enumerate}
\end{thm}

\begin{cor}\label{C:bi-extreme}
Let $T\sim \{X_n\}_{n=1}^\infty$ with $\essran X_1=[0,1]$.
If $A$ is a bilateral weighted shift, then  $A\lhd T$ a.s. if and only if $\|A\|\leq 1$.
\end{cor}

  To give the proof of Theorem \ref{T:BilateralAUE}, we need to make some preparation.
For the convenience of the reader, we separate a few ingredients to stress the similarity with the unilateral case and list them as results below. Their proofs are not hard by modifying the previous arguments. We only indicate how they can be proved.

\begin{lem}\label{P:TruncWeigh}
Let $A$ be a bilateral weighted shift with nonnegative weights $\{a_i\}_{i\in\bZ}$ and $B$ be a truncated weighted shift with positive weights $\{b_i\}_{i=1}^s$. Then $B\lhd A$ if and only if
$(0,b_1,b_2,\cdots,b_s,0)\in\Sigma_{s+2}(A).$
\end{lem}

\noindent  The proof of the above is similar to that of Lemma \ref{P:TruncUniWeigh}.

\begin{cor}\label{C:Tep}
Let $A$ be a bilateral weighted shift with nonnegative weights $\{a_i\}_{i\in\bZ}$. Then $0\lhd A$ if and only if
$(0,0)\in\Sigma_{2}(A).$
\end{cor}

\noindent The proof of the above  is  straightforward.

\begin{lem}\label{L:BilatTruncated}
Let $\{e_i\}_{i=1}^\infty$ be an orthonormal basis of $\cH$ and $\{f_i\}_{i=\in \bZ}$ be an
orthonormal basis of $\cK$. Assume that $$Ae_i=\lambda_i e_{i+1},\quad \forall i\geq 1,$$
and $$Bf_i=\mu_i f_{i+1},\quad \forall i\in \bZ, $$
where $\{\lambda_i\}_{i\in\bN},$ $\{\mu_i\}_{i\in\bZ}$ are bounded sequences of nonnegative real numbers.
Assume that $s\geq 1$, $t\in\bZ$ and $p(x,y)$ is a polynomial in two free variables $x,y$ with degree $l<s$.
\begin{enumerate}
\item[(i)] If $(\lambda_{s+i})_{i=-l}^{n+l}=(\mu_{t+i})_{i=-l}^{n+l}$, then $p(A^*,A)Q_n\cong p(B^*,B)Q_n'$, where $Q_n$ is the orthogonal
projection of $\cH$ onto $\vee\{e_i: s\leq i\leq s+n\}$ and  $Q_n'$ is the orthogonal
projection of $\cK$ onto $\vee\{f_i: t\leq i\leq t+n\}$.
\item[(ii)] If $(0,\lambda_{1},\lambda_{2},\cdots,\lambda_{n+l})=(\mu_{i})_{i=t}^{t+n+l},$ then $p(A^*,A)L_n\cong p(B^*,B)L_n'$, where $L_n$ is the orthogonal
projection of $\cH$ onto $\vee\{e_i: 1\leq i\leq n\}$ and  $L_n'$ is the orthogonal
projection of $\cK$ onto $\vee\{f_i: t+1\leq i\leq t+n\}$.
\end{enumerate}
\end{lem}

\noindent Everything in the above lemma, although seemingly technical,  can be verified directly.

\begin{lem}\label{P:URSDominated}
Let $A$ be a bilateral weighted shift with nonnegative weights $\{a_i\}_{i\in\bZ}$ and $B$ be a unilateral weighted shift with positive weights $\{b_i\}_{i=1}^\infty$. Then $B\lhd A$ if and only if
 $(0,b_1,b_2,\cdots,b_s)\in\Sigma_{s+1}(A)$ for  $s\geq 1$.
\end{lem}

\noindent The proof of the above can be modified from that of Lemma \ref{C:URSDomiUniWS} and Lemma \ref{P:TruncUniWeigh}.

  Now the preparation is done, and we are ready to give

\begin{proof}[Proof of Theorem \ref{T:BilateralAUE}] (i)
If $\Sigma_1(A)\nsubseteq\Gamma$, then it follows from Lemma \ref{P:NSpectrum} that $\Sigma_1(A)\nsubseteq\Sigma_1(T)$ a.s. That is,
 $\sigma(|A|)\nsubseteq\sigma(|T|)$ a.s.  By Corollary \ref{L:SpecDomina}, $A\ntriangleleft T$ a.s.
In the remaining we assume that $\Sigma_1(A)\subset\Gamma$. We shall show that $A\lhd T$ a.s.
Assume that $\{e_i\}_{i=1}^{\infty}$ is an orthonormal basis of $\cH$ and $Te_i=X_i e_i$ for $i\geq 1$.
Assume that $\{f_i\}_{i\in\bZ}$ is an orthonormal basis of $\cK$ and $Af_i=\mu_i f_{i+1}$ for $i\in\bZ$. By the hypothesis, we have $\{\mu_i: i\in\bZ\}\subset\Gamma$.
Choose a denumerable dense subset $\{\alpha_i: i\in\Lambda\}$ of $\Gamma$ and set $$\Omega'=\bigcap_{k\geq 1}\bigcap_{(i_1,\cdots,i_k)\subset\Lambda}\Big\{\liminf_{n\to\infty} \|(X_{n+j})_{j=1}^k-(\alpha_{i_j})_{j=1}^k\|_{\text{max}}=0\Big\}.$$  By Lemma \ref{L:constantCase}, we have $\bP(\Omega')=1$.
Now choose $\omega\in\Omega'$ and denote $\lambda_j=X_j(\omega)$ for each $j\geq 1$.
Thus $T(\omega)$ is a unilateral weighted shift with weights $\{\lambda_j\}$. Since $\{\alpha_i: i\in\Lambda\}$ is dense in $\essran X_1$, for any $k$-tuple $(\beta_1,\cdots,\beta_k)$ in $\Gamma$, we have
\begin{equation}\label{(4.44)}
\liminf_{n\to\infty} \|(\lambda_{n+i})_{i=1}^k-(\beta_{i})_{i=1}^k\|_{\max}=0.
\end{equation}
We denote $B=T(\omega)$.
It suffices to prove $A\lhd B$.
 Fix a polynomial $p(x,y)$ in two free variables and assume that degree $p(x,y)=l$.
For each $n\geq 1$, denote by $Q_n$ the orthogonal projection of $\cK$ onto $\vee\{f_i: -n\leq i\leq n\}$.
Note that $\|p(A^*,A)Q_n\|\rightarrow \|p(A^*,A)\|$ as $n\rightarrow\infty$.
Fix  $n\geq 1$. By the hypothesis, $\{\mu_i: -n-l\leq i\leq n+l\}\subset\Gamma$. In view of (\ref{(4.44)}), for any $\eps>0$, we can find $m_0>n+l$ such that
\[\|(\mu_i)_{i=-n-l}^{n+l}-(\lambda_j)_{j=m_0-n-l}^{m_0+n+l}\|_{\max}<\eps.\]
Then there exists a finite-rank operator $K$ on $\cH$ with $\|K\|<\eps$ such that $B+K$ is a
unilateral weighted shift satisfying $(B+K)e_i=\lambda_i'e_{i+1}$ for $i\geq 1$ and
\[(\lambda_{m_0-n-l}',\lambda_{m_0-n-l+1}',\cdots,\lambda_{m_0+n+l}')=(\mu_{-n-l},\mu_{-n-l+1},\cdots,\mu_{n+l}).\]
By Lemma \ref{L:BilatTruncated} (i), we have $p(A^*,A)Q_n\cong p\big(B^*+K^*,B+K\big)Q_n'$, where $Q_n'$ is the orthogonal projection of $\cH$ onto $\vee\{e_j: m_0-n\leq j\leq m_0+n\}$. It follows that
\[\|p(A^*,A)Q_n\|=\|p(B^*+K^*,B+K)Q_n'\|\leq \|p(B^*+K^*,B+K)\|.\]

   (ii)
By Lemma \ref{P:NSpectrum}, there exists $\Omega'\subset\Omega$ with $\bP(\Omega')=1$ such that $\Sigma_n(T(\omega))=\Gamma^n$ for all $n$ and  $\omega\in\Omega'$.
Since $0\notin\Gamma$, $T(\omega)$ is a unilateral weighted shift with positive weights for $\omega\in\Omega'$.
 For each $\omega\in\Omega'$ and  $n\geq 1$, we have $(0,X_1(\omega),\cdots,X_n(\omega))\in \{0\}\times \Gamma^n \subset\Sigma_{n+1}(A)$.
By Lemma \ref{P:URSDominated}, we have $T(\omega)\lhd A$. This proves the sufficiency.

  Conversely,  by Lemma \ref{P:URSDominated}, $T\lhd A$ a.s. implies that there exists $\Omega''\subset\Omega'$ with $\bP(\Omega'')=1$ such that $(0,X_1(\omega),\cdots,X_n(\omega))\in\Sigma_{n+1}(A)$ for $n\geq 1$ and  $\omega\in\Omega''$.
Hence \begin{equation}\label{(4.5)}\Sigma_{n+1}(A)\supset\{(0,X_1(\omega),\cdots,X_n(\omega)): \omega\in\Omega''\},\quad\forall n\geq 1.\end{equation}

\noindent Fix  $n\geq 1$. Arbitrarily choose $(\alpha_{i})_{i=1}^n\subset\essran X_1$. Then, given $\eps>0$, we have
\[\bP(\max_{1\leq i\leq n}|X_i-\alpha_i|<\eps)=\Pi_{i=1}^n\bP(|X_i-\alpha_i|<\eps)>0.\]
Thus there exists $\omega\in\Omega''$ such that $ |X_i(\omega)-\alpha_i|<\eps$ for $1\leq i\leq n$. Since $\eps$ is arbitrary, by (\ref{(4.5)}), we deduce that $(0,\alpha_1,\cdots\alpha_n)\in\Sigma_{n+1}(A)$.
By the discussion above, we have $\{0\}\times\Gamma^n\subset\Sigma_{n+1}(A)$ for  $n\geq 1$.
This ends the proof.

  (iii)
Choose an at most countable dense subset $\Delta=\{\alpha_i: i\in\Lambda\}$ of $\Gamma\setminus\{0\}$. Since $0$ is an accumulation point of $\essran X_1$, $\Delta$ is also dense in $\Gamma$.
Assume that $\eta_1,\eta_2,\eta_3,\cdots$ are all finite tuples with entries in $\Delta$.
For each $i\geq 1$, we assume that $\eta_i=(\alpha_{i,1}, \cdots, \alpha_{i,{n_i}})$ and write $E_{i}$ to denote the truncated weighted shift on $\bC^{n_i+1}$ with weights $(\alpha_{i,1}, \cdots, \alpha_{i,{n_i}})$.
Set $E=\oplus_{i=1}^\infty E_{i}$. By the proof of Theorem \ref{T:ApprUnitEqui} (iii),  $T\cong_a E^{(\infty)}$ a.s. Thus $T\lhd A$ a.s. if and only if $E\lhd A$.
Note that $E=\oplus_{i=1}^\infty E_{i}$. Then $E\lhd A$ if and only if $E_i\lhd A$ for $i\geq 1$.
In view of Lemma \ref{P:TruncWeigh}, the latter implies that
$(0, \alpha_{i,1}, \cdots, \alpha_{i,{n_i}},0)\in\Sigma_{n_i+2}(A)$ for $i\geq 1.$
Since $\{\eta_i: i\geq 1\}$ are all finite tuples with entries in $\Delta$ which is dense in $\Gamma$, we conclude that
$T\lhd A$ a.s. if and only if $\Gamma^n\subset\Sigma_n(A)$ for  $n\geq 1$.

  (iv) In the case that $\Gamma=\{0\}$, by Corollary \ref{C:Tep}, the result is clear. Next we deal with the case that $\card~\Gamma>1$.
We proceed as in the proof of Theorem \ref{T:UnilateralAUE} (iv).
Choose an at most countable dense subset $\Delta=\{\alpha_i: i\in\Lambda\}$ of $\essran X_1$.
Since $0$ is an isolated point of $\Gamma$, it follows that $0\in\Delta$ and $\Delta\setminus\{0\}$ is dense in $\Gamma_1$.
Suppose $\eta_1,\eta_2,\eta_3,\cdots$ are all finite tuples with entries in $\Delta$.
For each $i\geq 1$, we assume that $\eta_i=(\alpha_{i,1}, \cdots, \alpha_{i,{n_i}})$, and write $E_{i}$ to denote the truncated weighted shift on $\bC^{n_i+1}$ with weights $(\alpha_{i,1}, \cdots, \alpha_{i,{n_i}})$.
Set $E=\oplus_{i=1}^\infty E_{i}$. By the proof of Theorem \ref{T:ApprUnitEqui} (iii),  $T\cong_a E^{(\infty)}$ a.s. Thus $T\lhd A$ a.s. if and only if $E\lhd A$.
There are infinitely many $i$'s such that $(0,0)=(a_{i,s},a_{i,s+1})$ for some
$s$ with $1\leq s \leq n_i-1$. Thus $E$ can be rewritten as
$E=0^{(\infty)}\oplus (\oplus_{j=1}^\infty F_j^{(k_j)}),$
where each $F_j$ is a truncated weighted shift of order $\geq 2$ with weights in $\Delta\setminus\{0\}$.
For each $j\geq 1$, we assume that $F_j$'s weights are $\xi_j=(\beta_{j,1}, \cdots, \beta_{j,{m_j}})$.
Then  $\xi_1,\xi_2,\xi_3,\cdots$ are all finite tuples with entries in $\Delta\setminus\{0\}$.
 Set $F=0\oplus(\oplus_{j=1}^\infty F_j)$. Then $T\lhd A$ a.s. if and only if $F\lhd A$, i.e.,   $0\lhd A$ and $F_j\lhd A$ for $ j\geq 1.$
In view of Lemma \ref{P:TruncWeigh}, the latter implies that
\[(0,0)\in\Sigma_2(A), \quad(0,\beta_{j,1}, \cdots, \beta_{j,{m_j}},0)\in \Sigma_{m_j+2}(A),\quad \forall j\geq 1.\]
Since $\{\xi_j: j\geq 1\}$ are all finite tuples with entries in $\Delta\setminus\{0\}$ and $\Delta\setminus\{0\}$ is dense in $\Gamma_1$, we conclude that
$T\lhd A$ a.s. if and only if $\{0\}\times \Gamma_1^{n}\times \{0\}\subset\Sigma_{n+2}(A)$ for  $n\geq 0$. Thus the proof is complete.
\end{proof}


\subsection{Normal operators}

Recall that the original  von Neumann inequality may be interpreted as   comparing a contractive Hilbert space operator with a normal operator. This promotes much research activity to explore the (fruitful)  operator-function relationship.
The main result of this subsection is the following.

\begin{thm}\label{T:vonNeum-Normal}
Let  $T\sim\{X_n\}_{n=1}^\infty.$
If $N$ is a normal operator, then
\begin{enumerate}
\item[(i)] $\bP(N \lhd T)=\begin{cases}
1,& \textup{if} ~\sigma(N)\subset\{z\in\bC: |z|\in\essran X_1\},\\
0,& \textup{otherwise};
\end{cases}$
\item[(ii)] $\bP(T \lhd N)=\begin{cases}
1,& \textup{if} ~ \essran X_1=\{0\}\subset\sigma(N),\\
0,& \textup{otherwise}.
\end{cases} $
\end{enumerate}
\end{thm}

\begin{proof}
(i) By the Weyl-von Neumann Theorem, there exists a diagonal operator $D$ such that $D\cong_aN$.
Then $\sigma(N)=\sigma(D)$. Assume that $D=\textup{diag}\{\lambda_1,\lambda_2,\lambda_3,\cdots\}$.
Note that $\sigma(D)=\overline{\{\lambda_i: i\geq 1\}}$.
 If $\sigma(N)\nsubseteq\{z\in\bC: |z|\in\Gamma\}$, then there exists $i\geq 1$ such that $|\lambda_i|\notin\Gamma$.
By Lemma \ref{P:NSpectrum}, $\sigma(|T|)=\Sigma_1(T)=\Gamma$ a.s. By Corollary \ref{L:SpecDomina}, we have $\lambda_i\ntriangleleft T$ a.s. It follows that $D\ntriangleleft T$ a.s. and, equivalently, $N\ntriangleleft T$ a.s.
 Assume that $\sigma(N)\subset\{z\in\bC: |z|\in\Gamma\}$. Then $|\lambda_i|\in\Gamma$ for $i\geq 1$.
Fix  $i\geq 1$ and denote $\lambda_i=re^{\textup{i}\theta}$, where $r=|\lambda_i|$. By Theorem \ref{T:BilateralAUE} (i), $r U\lhd T$ a.s., where $U$ is the bilateral (unweighted) shift.
Since $r U$ is normal and $r\in\sigma(r U)$,
$\|p(r,r)\|\leq \|p(rU^*,r U)\|$ for each polynomial $p(x,y)$ in two free variables. That is, $r I\lhd rU$. This implies that $r I\lhd T$ a.s. Note that $T\cong e^{i\theta} T$. We deduce that $re^{\textup{i}\theta}I\lhd T$ a.s., that is, $\lambda_{i} I\lhd T$ a.s. Furthermore, we obtain $D\lhd T$ a.s. or,  equivalently, $N\lhd T$ a.s.

  (ii) By Lemma \ref{L:Homomor} and Corollary \ref{L:SpecDomina}, if an operator $A$ satisfies $A\lhd N$, then $A$ is normal and $\sigma(A)\subset\sigma(N)$.
If $\essran X_1\nsubseteq \{0\}$, then by Corollary \ref{C:EssNormal} in the next section, $T$ is almost surely not normal. So $\bP(T\lhd N)=0$.
 If $\essran X_1=\{0\}$, then $T=0$ a.s. Note that $0\lhd N$ if and only if $0\in\sigma(N)$. The proof is complete.
\end{proof}

\begin{cor}\label{C:BilateralAUE}
Let  $T\sim\{X_n\}_{n=1}^\infty.$ A unitary operator $U$ satisfies $U\lhd T$ a.s. if and only if $1\in\essran X_1$.
\end{cor}

  In view of Theorem \ref{T:vonNeum-Normal}, given a random weighted shift $T$ with non-degenerate weights and a normal operator $N$, one can not expect $T \vartriangleleft N$ a.s. For the weaker relation $\lhd_h$, we have the following result.

\begin{lem}
Let  $T\sim\{X_n\}_{n=1}^\infty$ with
$r<R=1.$ If $\|N\|\leq 1$, then $T \ntriangleleft_h N$ a.s.
\end{lem}

\begin{proof}
We choose $n\in\bN$ such that $r^3+r^{4(n-1)}<1$.
Set $p(x,y)=x{y}^2-x^{2n-1}{y}^{2n}$. Noting that $N$ is normal, we have
\[\|p(N,N^*)\|\leq \sup_{z\in B(0,1)}|z|^3(1-|z|^{4(n-1)}).\]
Note that $p(T,T^*)$ is a unilateral backward random weighted shift with weights $\{Z_k\}_{k=1}^{\infty}$, where
$Z_k=(1-\Pi_{j=2}^{2n-1}X_{k-j}^2)X_{k-1}^2X_{k},$ for $ k\geq 2n.$
Note that $\{Z_k: k=2n,4n,6n,\cdots\}$ are \iid. Then
$\|p(T,T^*)\|\geq 1-r^{4(n-1)} $ a.s.
We claim that $\sup_{z\in B(0,1)}|p(z,z^*)|<1-r^{4(n-1)}.$
If not, then we  choose $z\in \overline{B(0,1)}$ such that $1-r^{4(n-1)}\leq |z|^3(1-|z|^{4(n-1)})$. It follows that $|z|<r$. Hence
$
\frac{1-r^{4(n-1)}}{1-|z|^{4(n-1)}}\leq |z|^3< r^3,
$ which implies that $|z|<0,$ a contradiction. This proves the claim, and we deduce that
$\|p(N,N^*)\|<\|p(T,T^*)\|$ a.s.  Therefore $T \ntriangleleft_h N$ a.s.
\end{proof}

\subsection{Algebraical equivalence}\label{Subs:Al}

Armed with the results obtained so far, in this subsection we shall provide a classification of samples up to algebraical equivalence.
By definition, two operators $A,B$ on Hilbert spaces are algebraically equivalent if and only if $A\lhd B$ and $B\lhd A$.
The main result of this subsection is the following theorem.

\begin{thm}\label{T:AlgebraEquivalent}
Let  $T\sim\{X_n\}_{n=1}^\infty.$
Denote $\Gamma=\essran X_1$ and $\Gamma_1=\Gamma\setminus\{0\}$.
\begin{enumerate}
\item[(i)]  If $0\notin\Gamma$ and $\card~\Gamma>1$,
 then $\bP^2\big\{(\omega_1,\omega_2)\in\Omega^2: T(\omega_1)\approx T(\omega_2)\big\}=0;$ indeed,
$\bP(T\approx A)=0$ for any deterministic operator $A$.
\item[(ii)] If $0\in\Gamma$, then $\bP^2\big\{(\omega_1,\omega_2)\in\Omega^2: T(\omega_1)\approx T(\omega_2)\big\}=1;$ indeed, there exists a deterministic operator $A$ such that $T\approx A$ a.s.
\item[(iii)] If $0$ is an accumulation point of $\Gamma$ and $A$ is a (unilateral or bilateral) weighted shift, then
$T\approx A$ a.s. if and only if $\Sigma_n(A)=\Gamma^n$ for  $n\geq 1$.
\item[(iv)]  If $0$ is an isolated point of $\Gamma$  and $A$ is a unilateral weighted shift with nonnegative weights $\{a_i\}_{i=1}^\infty$, then
$T\approx A$ a.s. if and only if $\Sigma_1(A)\subset\Gamma$ and one of the following three holds:
\begin{itemize}
    \item[(a)] $a_1=0$ and $\{0\}\times\Gamma_1^n\times\{0\}\subset\Sigma_{n+2}(A)$ for  $n\geq 1;$ \item[(b)]   $a_1\ne0$ and $\{0\}\times\Gamma_1^n\times\{0\}\subset\Sigma_{n+2}(A)$ for  $n\geq 0;$
     \item[(c)]  $\exists m\geq 1$ s.t. $(a_1,\cdots, a_{m+1})\in\Gamma_1^{m}\times\{0\}$, $$\{0\}\times\Big[\Gamma_1^m\setminus\{(a_1,\cdots,a_m)\}\Big]\times\{0\}\subset\Sigma_{m+2}(A)$$ and
    \[\{0\}\times\Gamma_1^n\times\{0\}\subset\Sigma_{n+2}(A), \ \ \forall n\geq 0, n\ne m.\]
\end{itemize}
\item[(v)] If $0$ is an isolated point of $\Gamma$ and $A$ is a bilateral weighted shift, then
$T\approx A$ a.s. if and only if $\Sigma_1(A)\subset\Gamma$ and $\{0\}\times\Gamma_1^{n}\times\{0\}\subset\Sigma_{n+2}(A)$ for  $n\geq 0$.
\end{enumerate}
\end{thm}

  Statements  (iii), (iv) and (v) follow from Theorem \ref{T:UnilateralAUE} and Theorem \ref{T:BilateralAUE}. 
Now we  prove (i) and (ii).

%
%

\begin{proof}[Proof of Theorem \ref{T:AlgebraEquivalent}]
(i) Let $\{e_i\}_{i=1}^\infty$ be an orthonormal basis of $\cH$ and $Te_i=X_ie_{i+1}$ for  $i\geq 1$.
By Theorem \ref{T:compacts}, there exists $\Omega'\subset\Omega$ with $\bP(\Omega')=1$ such that $\KH\subset C^*(T(\omega))$ for  $\omega\in\Omega'$. In particular, $T(\omega)$ is irreducible for $\omega\in\Omega'$.

\medskip

\noindent{\it Claim.} For $\omega_1,\omega_2\in\Omega'$, $T(\omega_1)\approx T(\omega_2)$ if and only if $T(\omega_1)\cong T(\omega_2)$.

\medskip

\noindent  If the claim holds, then in view of Theorem \ref{T:ApprUnitEqui} (i), $\bP(T\approx T(\omega_0))=\bP(T\cong T(\omega_0))=0$ for any $\omega_0\in\Omega'$. It follows that $\bP^2\big\{(\omega_1,\omega_2)\in\Omega^2: T(\omega_1)\approx T(\omega_2)\big\}=0$. In particular, $\bP(T\approx A)=0$ for any operator $A$.
Now the proof of the claim.  The sufficiency is obvious. For necessity, let $A_i=T(\omega_i)$, $i=1,2.$ Since $A_1\approx A_2$, there exists a $*$-isomorphism $\varphi: C^*(A_2)\rightarrow C^*(A_1)$ such that $\varphi(A_2)=A_1$. Thus $\varphi|_{\KH}\ne 0$. By Corollary 5.41 in \cite{Douglas}, there exists unitary $V:\cH\rightarrow\cH$ such that $\varphi(Z)=V^*ZV$ for $Z\in C^*(A_2)$. Thus $A_1=V^*A_2V$, that is, $A_1\cong A_2$.
 (ii) By Theorem \ref{T:ApprUnitEqui} (iii), there exists a deterministic operator $W$ such that $W\cong_a T$ a.s.
Thus $W\approx T$ a.s.
\end{proof}

\begin{rmk}
Let  $T\sim\{X_n\}_{n=1}^\infty.$
If $X_1$ is degenerate and $\essran X_1=\{\lambda\}$, then $T$ is almost surely a  unilateral weighted shift with weight sequence $(\lambda,\lambda,\lambda,\cdots)$, that is, $T\cong\lambda S$ a.s., where $S$ is the unilateral shift.
If, in addition, $\lambda>0$, then $\KH\subset C^*(T)$. In view of the proof of Theorem \ref{T:AlgebraEquivalent} (i), an irreducible operator $A$ satisfies $T\approx A$ a.s. if and only if $A\cong \lambda S$.
\end{rmk}

\begin{cor}\label{C:AppUniEquIrred}
Let $T\sim \{X_n\}_{n=1}^{\infty}$.
If $X_1$ is non-degenerate, then the following are equivalent:
\begin{enumerate}
\item[(i)] There exists an irreducible unilateral weighted shift $W$ such that
$T\cong_a W$ a.s.
\item[(ii)] There exists an irreducible unilateral weighted shift $W$ such that
$T\approx W$ a.s.
\item[(iii)] $0$ is an accumulation point of $\essran X_1$.
\end{enumerate}
\end{cor}

\begin{proof}
In view of Theorem \ref{T:ApprUnitEqui} (ii) and Theorem \ref{T:AlgebraEquivalent} (i), we need only consider the case that $0\in\essran X_1$.
 {\it Case 1.} $0$ is an isolated point of $\essran X_1$.
 Choose a unilateral weighted shift $W$ such that $W\approx T$ a.s.
It follows from Corollary \ref{L:SpecDomina} that $\sigma(|W|)=\essran X_1$. Since $0$ is an isolated point of $\essran X_1$, it follows that $0$ appears infinitely many times in the weight sequence of $W$. So $W$ is reducible.
 {\it Case 2.} $0$ is an accumulation point of $\essran X_1$.
 Denote $\Gamma=\essran X_1$. Choose an at most countable dense subset $\Delta=\{\alpha_i: i\in\Lambda\}$ of $\Gamma\setminus\{0\}$. Since $0$ is an accumulation point of $\Gamma$, $\Delta$ is also dense in $\Gamma$.
We assume that $\eta_1,\eta_2,\eta_3,\cdots$ are all finite tuples with entries in $\Delta$. For each $i\geq 1$, assume that $\eta_i=(a_{i,1},\cdots, a_{i,n_i})$. We let $W$ be a unilateral weighted shift with weights
\[(a_{1,1},\cdots, a_{1,n_1},a_{2,1},\cdots, a_{2,n_2},a_{3,1},\cdots, a_{3,n_3},\cdots).\] Thus $\Sigma_n(W)=\Gamma^n$ for  $n\geq 1$ and $W$ is irreducible, since each $a_{i,j}$ is positive. By Theorem \ref{T:AlgebraEquivalent} (iii), $T\approx W$ a.s. By Theorem 3.2.1 in \cite{Donovan}, $C^*(W)$ contains no nonzero compact operators. Thus it follows from Theorem \ref{T:ApprUnitEqui} (iv) that $T\cong_a W$ a.s.
\end{proof}



\section{Invariant subspaces}\label{S:InvSubSpa}

  The  purpose of this section is to investigate the lattice of invariant subspaces of $T$.  A (closed) subspace $\mathcal{M}$ of $\cH$ is said to be {\it invariant}
for an operator $A\in\BH$ if $A(\mathcal{M})\subset \mathcal{M}$.
This lattice is intensely investigated for the three canonical weighted shifts (i.e., Hardy, Bergman, and Dirichlet), as well as many others. Examples considered in the past are usually essentially normal; in other words,  the weights don't oscillate asymptotically.

\subsection{Essential normality and subnormality}
This short subsection answers  two basic questions: (1) $T$ is almost surely not essentially normal, hence not like those familiar weighted shifts in the existing literature; (2) $T$ is almost surely not subnormal, hence not realizable as the multiplication operator on a weighted Hardy or Bergman-type space induced by a measure.

 Recall that an operator $A\in\BH$ is  {\it essentially normal} if $A^*A-AA^*$ is compact; {\it subnormal} if $A$ is unitarily equivalent to the restriction of a normal operator to an invariant subspace \cite{Conway}.

\begin{lem}\label{T:EssNormal}
Let $T\sim \{X_n\}_{n=1}^\infty.$ If $X_1$ is non-degenerate, then
  $$\bP\big(T~\textup{is essentially normal}\big)=0.$$
\end{lem}

\begin{proof}
Set $Y_n=X_{2n-1}-X_{2n}.$ Then $\{Y_n\}_{n=1}^{\infty}$ are i.i.d.   By the fact in the proof of Lemma \ref{L:quasiDiagonal},   almost surely,
\begin{equation*}
\liminf_{n\to\infty}Y_n=\min \essran Y_1,\ \limsup_{n\to\infty}Y_n=\max \essran Y_1.
\end{equation*}
Since  $\min \essran Y_1<0$ and $\max \essran Y_1>0,$ we know  $\lim\limits_{n\to\infty} (X_{i}-X_{i+1}) \neq 0$ a.s.
\end{proof}

\noindent The following corollary is straightforward and we record it  since it is used in the proof of Theorem \ref{T:vonNeum-Normal} in the previous section.

\begin{cor}\label{C:EssNormal}
Let $T\sim \{X_n\}_{n=1}^\infty.$
Then $$\bP\big(T ~\textup{is normal}\big)=\begin{cases}
1,& \essran  X_1=\{0\}, \\
0,& \textup{otherwise}.
\end{cases}$$
\end{cor}

\begin{lem}\label{P:hypoNormal}
Let $T\sim \{X_n\}_{n=1}^\infty.$
If $X_1$ is non-degenerate, then
  $$\bP\big(T~\textup{is hyponormal}\big)=0.$$
 \end{lem}

  An operator $A$ is {\it hyponormal} if $A^*A\geq AA^*$. It is well known that each subnormal operator is hyponormal \cite{Conway}.

\begin{proof}[Proof of Lemma \ref{P:hypoNormal}]
A direct calculation shows that $T$ is hyponormal if and only if $X_i\leq X_{i+1}$ for  $i\geq 1$. Observe that $$
\bP(X_i\leq X_{i+1}, \forall i\geq 1)
\leq \bP(X_{1}\leq X_{2})^n \leq c^n,$$ for some $c \in (0, 1)$ and for  $n \ge 1$
since $\{X_{2i-1}-X_{2i}\}_{i=1}^\infty$ are \iid.
The proof is complete now.  In particular, $T$ is almost surely not subnormal.
\end{proof}

\begin{rmk} To the best of our knowledge, there is no current example of non-essential normal weighted shift for which the lattice of invariant subspace is even reasonably understood. Next we show that $T=T(\omega)$
admits a rich lattice  by identifying it among the so-called $\mathbb{A}_{\aleph_0}$ class.
\end{rmk}

\subsection{$\mathbb{A}_{\aleph_0}$ and consequences}

In 1980s, Apostol, Bercovici, Foia\c{s}, and Pearcy \cite{Apostol} developed a beautiful theory on the structure of preduals of singly generated dual algebras. A dual algebra is a subalgebra $\mathcal{A}$ of $\BH$ that contains the identity $I$, and is closed in the weak*
topology on $\BH$. For dual algebras, they introduced properties $(\mathbb{A}_n), 1\leq n\leq\aleph_0$. For each $n=1,2,\cdots,\aleph_0$, denote by $\mathbb{A}_n$ the set of all operators $A\in\BH$ such that the dual algebras generated by $A$ has the property $(\mathbb{A}_n)$. Operators in $\mathbb{A}_{\aleph_0}$ have a rich lattice of invariant subspaces \cite[Theorem 3.1]{Apostol}. We reformulate the result as below.  Recall that an operator $A$ is  {\it reflexive} if $\{Z\in\BH: \Lat (A)\subset \Lat (Z)\}$ coincides with the (weakly closed) algebra generated by $A$. Reflexive operators represent another way to say that invariant subspaces are many \cite{Conway2, Sarason}.

\begin{thm}[\cite{Apostol}, Theorem 3.1]\label{T:ABFP}
If $A\in \mathbb{A}_{\aleph_0}$, then
\begin{enumerate}
\item[(i)] $A$ is reflexive;
\item[(ii)] $\Lat (A)$ contains a lattice that is isomorphic to the lattice of subspaces of $\cH$;
\item[(iii)] if $B$ is a strict contraction (that is, $\|B\|<1$), then there exist $\mathcal{N}, \mathcal{M}\in\Lat~A$ with $\mathcal{N}\subset \mathcal{M}$ such that the compression $A_{\mathcal{M}\ominus \mathcal{N}}\cong B$.
\end{enumerate}
\end{thm}
\noindent  Moreover, the above theorem is about abstract operators and it was a victory of the theory to discover that the Bergman shift  lies in this class \cite[Proposition 3.2]{Apostol}. Since then the investigation of the operator theory on  the Bergman space was significantly boosted by this fact.
Next we show that the random model provides a new member in $\mathbb{A}_{\aleph_0}$, which is apparently quite different from the Bergman shift.

\begin{thm}\label{P:Universal}
Let $T\sim \{X_n\}_{n=1}^\infty$ with $r<R=1$. Then, almost surely, $T\in\mathbb{A}_{\aleph_0}$.
\end{thm}

\begin{rmk} The degenerate case, that is, $r=R$,  is not in $\mathbb{A}_{\aleph_0}$.
\end{rmk}

%

\noindent  The proof of Theorem \ref{P:Universal} is rather short  in view of Theorem 3.6 in \cite{Apostol}, which states that a
weighted shift $A$ lies in $ \mathbb{A}_{\aleph_0}$  if it is a $C_{00}$ operator and $\gamma(A)=\|A\|=1$.  Here $C_{00}$ represents the set of all (completely nonunitary)
contractions $A\in\BH$ for which both sequences $\{A^n\}_{n=1}^\infty$ and $\{{A^*}^n\}_{n=1}^\infty$
converge to $0$ in the strong operator topology. Hence the  proof of Theorem \ref{P:Universal} is reduced to the following, and because of which we have no need to bother with the  technical, precise definition of $\mathbb{A}_{\aleph_0}$.

\begin{lem}\label{P:C00}
Let $T\sim \{X_n\}_{n=1}^\infty$ with $r<R=1$.
Then $ \bP(T\in C_{00})=
1.$
\end{lem}

\begin{proof}
By Lemma  \ref{T:NormSpectral}, $\|T\|=R=1$ a.s. The fact that  ${T^*}^n\overset{\sot}{\longrightarrow} 0$ as $n\rightarrow\infty$ is straightforward. The rest of the proof is divided into two cases.
{\it Case 1.} $\bP(X_1=0)>0$.
 By the proof of Lemma \ref{P:PointSpect}, $T$ is almost surely the direct sum of a family nilpotent operators. Thus $T^n\overset{\sot}{\longrightarrow} 0$ a.s. Thus $T\in C_{00}$ a.s.
{\it Case 2.} $\bP(X_1=0)=0$.
 For any $k\geq 1$, we claim that almost surely $\Pi_{i=1}^nX_{i+k}\rightarrow 0$  as $n\rightarrow\infty$.
Since  $\{\ln X_n\}_{n=1}^{\infty}$ are \iid random variables  and
$\bE(\ln X_1)<0,$ $\sum_{i=1}^n \ln X_{i+k}\rightarrow -\infty$ a.s.. Hence $\Pi_{i=1}^nX_{i+k}\rightarrow 0$ a.s. This proves the claim.
 Recall that $Te_n=X_ne_{n+1}$  for some orthonormal basis $\{e_n\}_{n=1}^\infty$  of $\cH.$ Thus $\|T^ne_k\|=\Pi_{i=1}^nX_{i+k-1}$
for $k\geq 1$ and $n\geq 1$. By the preceding claim, for each $k\geq1$, we have almost surely $\|T^ne_k \|\rightarrow 0$ as $n\rightarrow\infty$.
Thus $T\in C_{00}$ a.s.
\end{proof}



The {\it index} of  an invariant subspace $\cM\in
\Lat(A)$ of $A\in\BH$ is defined to be the dimension of $\cM\cap (A\cM)^\bot$.
Beurling \cite{Beurling} showed that the unilateral shift has the codimension-one property, that is, every nonzero invariant subspace of the unilateral shift has an index of one. This fact turns out to be one of the most seminal results in operator theory. A great deal of effort was put into determining the indices of invariant subspaces of weighted shifts and this line of research remains active  as of today. It came as a big surprise when it was first discovered that the indices can be arbitrary in the case of the Bergman space, also due to the powerful work on dual algebras mentioned above. The following is  an immediate consequence of the $\mathbb{A}_{\aleph_0}$ membership in Theorem \ref{P:Universal}.


\begin{cor}\label{C:Universal}
Let $T\sim \{X_n\}_{n=1}^\infty$ with $R=1$.
Then the following are equivalent:
\begin{enumerate}
\item[(i)] Almost surely, there is a nonzero $\cM\in\Lat(T)$ with index $\ge 2$;
\item[(ii)] For each $n=2, \cdots, \infty$, there is a nonzero $\cM\in\Lat(T)$ with index $n$;
\item[(iii)] $ X_1$ is non-degenerate.
\end{enumerate}
Consequently, $T$ almost surely has the ``codimension-one property" if and only if $ X_1$ is degenerate.
\end{cor}

\begin{rmk}
Suppose that $T\sim \{X_n\}_{n=1}^\infty.$
If $r>0$, then almost surely each $\cM\in\Lat(T)$ with $\dim\cH\ominus\cM<\infty$ satisfies $\cM\ominus T(\cM)=1$.
\end{rmk}


\subsection{Toward a Beurling-type theorem}\label{S:sub Beurling-type}

This subsection concerns our (clearly unsuccessful) effort toward a Beurling-type theorem for $T(\omega)$. We feel that the presentation here is still of some interests primarily because of  the questions it brings forth. (Accordingly, for simplicity, we skip most details of proofs.) Throughout this subsection we assume that  $A\in \BH$ is left-invertible and analytic, that is,  $\cap_{n=1}^{\infty}A^n\cH=\{0\}.$    We say that the {\it Beurling-type theorem} holds  for $A$ if $\cM=[\cM\ominus A\cM]$ for each $\cM\in \text{Lat}(A)$; in other words, $\cM$ is the smallest invariant subspace containing $\cM\ominus A\cM.$  The origin of this property is the celebrated Beurling-Halmos-Lax theorem \cite{Beurling,Halmos61,Lax} which, among other things, states that a $z$-invariant subspace $\cM$ of the vector-valued  Hardy space over the unit disk is generated by $\cM\ominus z\cM.$

  Although it has been more than half a century since \cite{Beurling}, this property is still a focus of current interests. In particular, an attractive open problem is to prove, or dis-prove, it for the entire $\D_\al$ family $(-\infty<\al<\infty),$ which has been introduced in Section \ref{S:random Hardy}. Recall that
\begin{equation}\label{E:D al}
\D_{\al}=\Big\{f(z)=\sum_{n=0}^{\infty}a_nz^n\in H(\dd):\ \|f\|_{\al}^2=\sum_{n=0}^{\infty}(n+1)^{\al}|a_n|^2<\infty\Big\},
\end{equation}
where $H(\dd)$ denotes the collection of analytic functions over $\dd.$ Note that $\D_{-1}, \D_0$ and $\D_1$ corresponds to the classical Bergman, Hardy and Dirichlet spaces, respectively.

  In 1988, S. Richter \cite{1988Richter Dirichlet} proved  the Beurling-type theorem  for  $\D_{\al}$,  $0< \al\leq 1$. A more general theorem was proved there indeed: If $A\in \BH$ is  analytic  and concave, that is,
\begin{equation}\label{E:concavity}
\|A^2x\|^2+\|x\|^2\leq 2\|Ax\|^2,\quad  \forall x\in\cH,
\end{equation}
then
$
\mathcal{H}=[\mathcal{H}\ominus A\mathcal{H}].
$
One can observe  two consequences of  the concavity condition (\ref{E:concavity}): First,  the operator is expansive, that is, $\|Ax\| \ge \|x\|$; second,  the {\it moment sequence} $\{\|A^nx\|^2\}_{n=0}^{\infty}$ grows at most linearly for each  $x\in\cH$:
\begin{equation}\label{E:linear}
\|A^nx\|^2\leq n(\|Ax\|^2-\|x\|^2)+\|x\|^2,\quad x\in \cH.
\end{equation}
These two facts  play  key roles in the proof in \cite{1988Richter Dirichlet}.
The following lemma is essentially a reformulation of some arguments in   \cite{1988Richter Dirichlet}.

\begin{lem}\label{L:Sufficy for WSP}
Let $A \in \BH$ be  left-invertible and analytic, let $S=A|_\cM$ be the restriction of $A$ to an invariant subspace $\cM\in\Lat (A)$,  and let $L$ be the standard left-inverse of $S,$ i.e. $L=(S^{\ast}S)^{-1}S^{\ast}.$  If for any $x \in \cM$,
\begin{equation}\label{E:y_n has finite c-point}
\liminf_{n\to\infty}\|S^nL^nx\|<\infty,
\end{equation}
then $\cM=[\cM\ominus S\cM].$
\end{lem}

The Bergman case  ($\al=-1$) was proved several years later, in 1996 by Aleman-Richter-Sundberg \cite{Bergman}, via sophisticated  arguments. An early motivation is Hedenmalm's \cite{Hed91}.  S. Shimorin \cite{2001 Shimorin Crell, 2003 Shimorin PAMS} extended it to  $-2\leq \al<0$ by much easier but elegant tricks. Indeed he  reduced the problem to arguments for $\al=1$. The Bergman case was reproved by Sun-Zheng in \cite{ShunZheng} by lifting the Bergman space to the bidisk Hardy space. This proof was simplified by K. Izuchi \cite{Izuchi}.   If $\al<-5,$ examples of  Hedenmalm-Zhu \cite{1992Zhu} disproves the property. Furthermore, Hedenmalm and Perdomo \cite{HedJMPA} gave counterexamples for $\alpha<-2.04$. As of today, the problem is still open for $-2.04\leq \alpha<-2$.
   %

On the other hand, almost nothing is known for $\al >1$. But this is probably the range more relevant to the present paper, only with more severe difficulty for the random model. Two new technical challenges are present for this range, say, when compared with the Dirichlet case, i.e., $\al=1$. First, the multiplication operator is strictly convex, as opposed to the concavity condition in (\ref{E:concavity}). To the best of our knowledge, there is no result  for convex operators in this area so far. Second,   the moment sequence exhibits super-linear growth, indeed of order $n^\alpha$ for $\D_\al$,  as opposed to (\ref{E:linear}).

It seems that the difficulty associated with $\alpha >1$ is only aggravated when it comes to the random model $T(\omega)$.   We argue that a Beurling-type theorem for $T(\omega)$, if eventually established,   needs to treat ``very convex" operators (Lemma \ref{L:m-convex}) and ``very fast-growing" moment sequences (Lemma \ref{P:LIL App}).

As a first step, we offer a simple result for convex operators, as well as a conjecture, to spur further interests. The following lemma is formulated in a way so that it is meaningful to compare it with Richter's result on concave operators. The proof is not hard for any serious reader by modifying the arguments of \cite{1988Richter Dirichlet}, or by Lemma \ref{L:Sufficy for WSP}.

\begin{lem}\label{L:convex}
If an analytic operator $A\in \BH$ is  expansive, convex, and the moment sequence  $\|A^nx\|^2$ satisfies linear growth, that is, there exists $c_1, c_2$ such that $\|A^nx\|^2 \leq (c_1 n+c_2)\|x\|^2$ for $x\in \mathcal{H},$
then the Beurling-type theorem holds for $A$.
\end{lem}

\begin{con}\label{Con:quadratic}
 The above result holds when the linear growth condition is relaxed to quadratic growth.
\end{con}

\noindent Eventually, in order to treat the random model, one needs to take care of half-exponential growth. Indeed, the precise growth rate is $e^{\lambda \sqrt{ n\ln\ln n}}$ for a specific $\lambda$ (Lemma \ref{P:LIL App}).

  Now we analyze the random model in details. The next lemma says that it is convex of any order in the mean. This suggests that the moment sequence is fast-growing. The proof is by direct calculation.

\begin{lem}\label{L:m-convex}
Let  $T\sim \{X_n\}_{n=1}^{\infty}.$ Then for any $m\geq 1,$
\begin{equation*}
\bE \Big(\sum_{l=0}^m(-1)^l\binom{m}{l}\|T^{m-l}x\|^2\Big)=\big(\bE (X_1^2)-1\big)^m\|x\|^2,\quad x\in H.
\end{equation*}
\end{lem}

  Note that, unless in the degenerate case, we  have $\bE (X_1^2) >1$ since the standard normalization is $e^{\bE(\ln X_1)}=1.$

  Next we study  the precise  growth rate of the moment sequence $\{\|T^nx\|^2\}_{n=0}^{\infty}.$  The proof of the following lemma is by direct applications of the law of iterated logarithm \cite{Hartman-Wintner}.

\begin{lem}\label{P:LIL App}
Let $T\sim \{X_n\}_{n=1}^{\infty}$,  $\{e_k\}_{k=1}^{\infty}$ be the associated orthonormal basis for $\cH$,  $\bE(\ln X_1)=0$ and $\bE(\ln X_1^2)^2 =\sig^2<\infty.$ Then for each $k,$
\begin{equation}\label{E:165}
\limsup_{n\to\infty} \frac{\ln \|T^ne_k\|^2}{\sqrt{2\sig^2n\ln\ln n}}=1, \qquad
\liminf_{n\to\infty} \frac{\ln \|T^ne_k\|^2}{\sqrt{2\sig^2n\ln\ln n}}=-1\ \ a.s.
\end{equation}

\end{lem}


Indeed, to be more specific,
for any $\ep>0,$ let
\begin{equation*}
A_n=\Big[\|T^ne_k\|^2<e^{(1+\ep)\sqrt{2\sig^2n\ln\ln n}}\Big],\quad
B_n=\Big[\|T^ne_k\|^2>e^{(-1-\ep)\sqrt{2\sig^2n\ln\ln n}}\Big].
\end{equation*}
Then
\begin{equation}\label{E:262}
\bP(\liminf_{n\to\infty}A_n)=\bP(\liminf_{n\to\infty}B_n)=1.
\end{equation}

  In order to treat the random model, one might need more estimates beyond the mere growth rate in the mean. Then the tail probability of $\|T^ne_k\|^2$ offered by, say,  the classical  Hoeffding's inequality \cite{Hoeffding inequality} or the like, can be helpful.


  For a general moment sequence $\{\|T^nx\|^2\}_{n=1}^{\infty}$, let $x=\sum_{k=1}^{\infty}a_ke_k\in \cH.$ Then
\begin{equation}\label{E:173}
\|T^nx\|^2=\sum_{k=1}^{\infty}|a_k|^2X_{k}^2\cdots X_{k+n-1}^2.
\end{equation}

\begin{rmk}This is a sum of  inter-correlated random variables from a stationary series.
For  sums of independent r.v.'s, plenty of techniques are available \cite{Petrov, Stout}.   Unfortunately, for dependent random variables, existing knowledge in probability does not readily yield satisfactory answers for us. For instance, one can apply the theory of orthogonal series to the above (\ref{E:173})
if $\mathbb{E}(X_1^2)=0$, which is  absurd for us.
As mentioned in the introduction, we plan to start a separate work to sort out the general theory of (\ref{E:powerseries}) and (\ref{E:powerseries2}) which are elegant examples of dependent summation.
\end{rmk}

\begin{con}\label{con:growth}
 Lemma \ref{P:LIL App}  holds for  $x\in \cH$ (instead of only $e_k \in \cH)$.
 \end{con}

\section{Algebras generated by $T$}\label{S:algebra}

  The aim of this section is to describe various algebras generated by a random weighted shift $T(\omega)$. The list includes
Banach algebras, weakly closed algebras, commutant algebras/multipliers, dual algebras and $C^*$-algebras. The proofs rest on some past results, especially from \cite{shields}.

\subsection{Non-selfadjoint algebras}

Let $A\in\BH$. We denote
$$\sA_A=\textup{the norm closure of polynomials in}\ A $$ and
$$\sW_A=\textup{the closure in the weak operator topology of polynomials in}\ A.$$ Thus $\sA_A$ is the Banach algebra generated by $A$.
By \cite[P. 477]{Dunford}, $\sW_A$ coincides with the closure, in the strong operator topology, of polynomials in $A$.
We denote by $\{A\}'$ the commutant of $A$, and by $\{A\}''$ the double commutant of $A$. It is easy to see that $\{A\}'$ and $\{A\}''$ are both closed in the weak operator topology and $\sW_A\subset\{A\}'$.

  Since $\BH$ is the dual space of the ideal of trace-class operators, it makes sense to define
$$\sO_A=\textup{the closure, in weak* topology, of polynomials in}\ A,$$ which is called the
{\it dual algebra} generated by $A$ \cite{Apostol}.


%

\begin{thm}\label{T:GenerAlgebra}
Let  $T\sim\{X_n\}_{n=1}^\infty$ with $R=1$. Then
\begin{enumerate}
\item[(i)] $\sA_T$ is isometrically isomorphic to the disk algebra $\mathcal{A}(\bD)$  almost surely.
\item[(ii)] $\sO_T=\sW_T$ a.s. Moreover, if   $\bP(X_1=0)=0$, then $\sO_T=\sW_T=\{T\}'=\{T\}''$ is isometrically isomorphic to $H^\infty(\bD)$  almost surely.
\end{enumerate}
\end{thm}

\begin{proof}[Proof of Theorem \ref{T:GenerAlgebra}]
We first observe that $\sigma(T)=\overline{\bD}$ a.s.  by Lemma  \ref{T:NormSpectral}.
(i)
For any polynomial $p(z)$, by the von Neumann inequality and the spectral mapping theorem,   $\gamma(p(T)=\|p\|_\infty$ a.s. Now we easily  obtain
$\|p(T)\|=\|p\|_{\infty}$ a.s. Since polynomials are dense in ${A}(\bD)$, the map $p\mapsto p(T)$ extends to an isometrically  isomorphism $\varphi$ of ${A}(\bD)$ onto $\sA_{T}$ a.s.
 (ii)  By Lemma  \ref{T:NormSpectral} and  Theorem 3.7 in \cite{Apostol}, one easily has $\sO_T=\sW_T$ a.s. Since  $X_i>0$ for  $i\geq 1$ a.s.,  $T$ is injective and $\|T\|=\gamma(T)=1$ a.s.   By Theorem 3 in \cite{shields} and its corollary, we have $\{T\}'=\{T\}''$ a.s. Furthermore, using Theorem 10 and Theorem 12 in \cite{shields}, we deduce that $\{T\}'$ is isometrically isomorphic to $H^\infty(\bD)$ a.s.
\end{proof}

\begin{rmk}
Take $T$ as in Theorem \ref{T:GenerAlgebra}. By Theorem \ref{T:GenerAlgebra} (ii), a functional calculus of $T$ with respect to functions in $H^\infty(\bD)$ can be defined. In particular, we have
(\ref{OpeFunc}). But we should keep in mind that the assumption $R=1$ is not well suited for function theory since, according to Lemma \ref{T:convergence radius}, functions in the random Hardy spaces are now not defined over the unit disk.
\end{rmk}

\begin{rmk}\label{R:multiplier alg}
If we realize $T$ as $M_z$ on $H^2_{\mu}$ (Section \ref{S:random Hardy}), then Theorem 3 in \cite{shields} also implies that the multiplier algebra of $H^2_{\mu}$  is isometrically  isomorphic to $H^\infty(\bD)$ a.s. under the assumption of (ii) in Theorem \ref{T:GenerAlgebra}.
\end{rmk}

\subsection{$C^*$-algebras}
In this subsection we give some applications of results obtained in Section  \ref{S:SampleClassify} and Section \ref{S:vonNeumanIneq}. The results obtained here may be viewed as the first step to understand the structures of the $C^*$-algebras generated by $T$. We answer two basic questions: when they are simple and when they are GCR. The results suggest  that the structure of $C^*(T)$ is  nontrivial.

 GCR algebras are thought to be among the most tractable $C^*$-algebras. A $C^*$-algebra $\cA$ is said to be CCR if, for every irreducible representation $\rho$ of $\cA$, $\rho(\cA)$ consists of compact operators. A GCR algebra is a $C^*$-algebra $\cA$ such that every nonzero quotient $C^*$-algebra of $\cA$
possesses a nonzero CCR ideal. A $C^*$-algebra is said to be NGCR if it has no nonzero CCR ideals.
An operator $A$ is said to be GCR if $C^*(A)$ is a GCR algebra (\cite{arveson}, Definition 1.5.6).
The class of all GCR operators contains the normal operators, the compact operators and the unilateral shift (\cite{arveson}, Exercise 1.5.D).
NGCR operators are similarly defined.
An irreducible operator $A$ is NGCR if and only if $C^*(A)$ contains no nonzero compact operators (\cite{arveson}, Exercise 1.5.C).
The reader is referred to Section 1.5 in \cite{arveson} for more details.

  The main result of this subsection is the following result. Recall that a $C^*$-algebra $\mathcal{A}$ is {\it simple} if it has no proper two-sided ideal.

\begin{thm}\label{T:GCR}
Let  $T\sim\{X_n\}_{n=1}^\infty.$
Then
\begin{enumerate}
\item[(i)] $\bP(C^*(T)\ \textup{is simple})=\begin{cases}
1,& \essran X_1=\{0\},\\
0,& \textup{otherwise};
\end{cases}$
\item[(ii)] $\bP(T\ \textup{is GCR})=\begin{cases}
1,& \card\big(\essran X_1\setminus\{0\}\big)\leq1,\\
0,& \textup{otherwise}.
\end{cases}$
\end{enumerate}
\end{thm}

\begin{rmk}
Let $T\sim \{X_n\}_{n=1}^\infty$ with $\bP(X_1=0)=0$. Then $T$ is almost surely irreducible.
If $0\in\essran X_1$, then by Theorem \ref{T:compacts}, $C^*(T)$ almost surely contains no nonzero
compact operators, which implies that $T$ is almost surely NGCR.
\end{rmk}

\begin{proof}[Proof of Theorem \ref{T:GCR}]  (i) If $\essran X_1=\{0\}$, then $C^*(T)=\bC I$ a.s. So $T$ is almost surely simple.
If $\essran X_1\nsubseteq\{0\}$, we can choose nonzero $\lambda\in\essran X_1$. By Theorem \ref{T:vonNeum-Normal}, $\lambda\lhd T$ a.s.
By Lemma \ref{L:Homomor}, almost surely, there exists a unital $*$-homomorphism $\varphi: C^*(T)\rightarrow\bC$
so that $\varphi(T)=\lambda$. Hence $\ker\varphi$ is a proper ideal of $C^*(T)$.

  (ii) First, we deal with the case that $\essran X_1$ contains at least two nonzero points. We choose distinct $\lambda_1,\lambda_2\in\essran X_1\cap(0,\infty)$. We claim that there exists a deterministic NGCR bilateral weighted shift $A$ with weights lying in $\{\lambda_1,\lambda_2\}$.
 Assume that $\eta_1,\eta_2,\eta_3,\cdots$ are all finite tuples with entries in $\{\lambda_1,\lambda_2\}$.
For each $i\geq 1$, assume that $\eta_i=(c_{i,1},\cdots, c_{i,n_i})$.
Then we can construct a bilateral sequence $\{d_i\}_{i\in\bZ}$ with entries in $\{a,b\}$ such that each $\eta_i$ appears in
it. That is, for each $i\geq 1$, there exists $j\in\bZ$ such that $(d_{j+1},\cdots, d_{j+n_i})=\eta_i$. For example, to construct $\{d_i\}$, we may put these $\eta_i$'s in the following order:
\[\cdots,\eta_6,\eta_4,\eta_2,\eta_1,\eta_3,\eta_5,\cdots \] Let $A$ be a bilateral weighted shift with weights $\{d_i\}_{i\in\bZ}$.
Thus it is easy to see that \begin{enumerate}
\item[(a)] $\Sigma_n(A)=\{\lambda_1,\lambda_2\}^n$ for  $n\geq 1$, and
\item[(b)] $\card\{k\in\bZ: (d_{i})_{i=-n}^n=(d_{k+i})_{i=-n}^n\}=\infty$ for any $n\geq 1$.
\end{enumerate}

\noindent By (a),   $\{d_n\}_{n\in\bZ}$ is not periodic. In fact, if $\{d_n\}_{n\in\bZ}$ is periodic with period $m$, then
$$\card\Sigma_k(A)=\card\{(d_{j+1}, \cdots, d_{j+k}): j\in\bZ\}\leq m,$$which contradicts (a).
Then $A$ is not periodic and, by Problem 159 in \cite{HilbertProblem}, $A$ is irreducible. On the other hand,
by Theorem 2.5.1 of \cite{Donovan}, condition (b) implies that $C^*(A)$ contains no nonzero compact operators. Thus $A$ satisfies all requirements.
This proves the  claim.


  In view of Theorem \ref{T:BilateralAUE} (i), we have $A\lhd T$ a.s. Then there  exists almost surely a $*$-homomorphism $\varphi: C^*(T)\rightarrow C^*(A)$ such that $\varphi(T)=A$. Since $C^*(A)$ contains no nonzero compact operators, by Proposition 1.54 in \cite{arveson}, $T$ is almost surely not GCR.

  Next we assume that $\essran X_1$ contains at most one nonzero point. Then there exists positive number $\lambda$ such that $\essran X_1\subset\{0,\lambda\}$.
 {\it Case 1.} $\essran X_1=\{0\}$.
 Then $C^*(T)=\bC I$ is almost surely a commutative $C^*$-algebra. Thus $T$ is almost surely GCR.
 {\it Case 2.} $\essran X_1=\{\lambda\}$.
 In this case, $T$ is almost surely unitarily equivalent to $\lambda S$, where $S$ is the unilateral shift.
Thus $C^*(T)$ is almost surely $*$-isomorphic to $C^*(S)$. By Exercise 1.5.D in \cite{arveson}, $T$ is almost surely GCR.
{\it Case 3.} $\essran X_1=\{0,\lambda\}$.
 Without loss of generality, we assume that $\lambda=1$. Then $\bP(X_1=0)>0$ and $\bP(X_1=1)>0$.
Then $T$ is almost surely a unilateral weighted shift with weights lying in $\{0,1\}$.
Then, almost surely, $T=\oplus_{i=1}^\infty T_i$, where each $T_i$ is either a truncated weighted shift with weights in $\{1\}$ or a degenerate truncated shift of order $1$ (namely $0$ acting a subspace of dimension $1$). Moreover, for each $n\geq 1$, it follows from Lemma \ref{L:constantCase} that the $(n+2)$-tuple \[(0,\underbrace{1,1,\cdots,1}_n,0)\] almost surely appears infinitely many times in the weight sequence of $T$.
Then, almost surely, $T\cong \oplus_{k=1}^\infty J_k^{(\infty)}$, where $J_k$ is a truncated weighted shift with weights in $\{1\}$ of order $k$.
By Theorem 4.1 in \cite{BunceDeddens74}, $\oplus_{k=1}^\infty J_k$ is GCR. So $T$ is almost surely GCR. The proof is complete now.
\end{proof}


\section{Dynamical properties}\label{S:Dynamic}


  It is natural to wonder whether random models  exhibit more chaotic behaviors. This is indeed the case. The main result of this section is the following theorem which addresses the most common questions one might ask from the outset. Notations will be explained in subsequent subsections as we proceed to the proofs.
The reader is referred to \cite{DynamicLinOper} for the basic theory of dynamics of linear operators which indeed covers most technical tools we need here.

\begin{thm}\label{T:Dynamics}
Let  $T\sim\{X_n\}_{n=1}^\infty$ with $X_1$ being non-degenerate and $\bP(X_1=0)=0$.
Then\begin{enumerate}
\item[(i)] $ T^*$ is almost surely supercyclic;
\item[(ii)] $\bP(T^* ~\textup{is hypercyclic})=\begin{cases}
1,& \bE(\ln X_1)\geq 0,\\
0,& \bE(\ln X_1)<0;
\end{cases}$
\item[(iii)] $
\bP(T^* ~\textup{is Li-Yorke chaotic})=\begin{cases}
1,& R> 1,\\
0,& R\leq 1;
\end{cases}  $
\item[(iv)] $
\bP(T^* ~\textup{is toplogically mixing})
=\begin{cases}
1,& \bE(\ln X_1)> 0,\\
0,& \bE(\ln X_1)\leq 0;
\end{cases}   $
\item[(v)] $\bP(T^* ~\textup{is chaotic})=\bP(T^* ~\textup{is frequently hypercyclic})
=\begin{cases}
1,& \bE(\ln X_1)> 0,\\
0,& \bE(\ln X_1)\leq 0.
\end{cases}  $
\end{enumerate}\end{thm}

\subsection{Supercyclicity and hypercyclicity}

Let $A\in\BH$. Recall that $A$ is said to be {\it supercyclic} if there exists a vector $x\in\cH$
such that the set $\{\lambda A^nx: \lambda\in\bC, n\geq 0\}$ is a dense subset of $\cH$.
If there exists $y\in\cH$ such that
the set $\{A^ny: n\geq 0\}$ is a dense subset of $\cH$, then $A$ is said to be {\it hypercyclic}.
Note that there exists no hypercyclic operator on a finite dimensional space.

\begin{lem}\label{L:supeCyclic}
Let $A$ be a backward unilateral weighted shift with weights $\{\lambda_n\}_{n=1}^\infty$.
Then $A$ is supercyclic if and only if $\lambda_i\ne 0$ for  $i\geq 1$.
\end{lem}

\begin{proof}
If there exists $i_0$ such that $\lambda_{i_0}=0$, then $0\in\sigma_p(A^*)$. By Proposition 1.26 in \cite{DynamicLinOper}, $A$ is not supercyclic.
If $\lambda_i\ne 0$ for  $i\geq 1$, then by Example 1.15 in \cite{DynamicLinOper}, $A$ is supercyclic.
\end{proof}

  By Proposition 1.17 in \cite{DynamicLinOper}, each hypercyclic operator $A$ satisfies that $\sigma_p(A^*)=\emptyset$.
Then a forward unilateral weighted shift is never hypercyclic. When considering hypercyclic backward weighted shifts, one needs only consider those  with positive weights.


%

%

\begin{lem}[\cite{Salas}, Theorem 2.8]\label{L:salas}
If $A$ is a backward unilateral weighted shift with positive weights $\{\lambda_i\}_{i=1}^\infty$, then $A$ is hypercyclic if and only if \[\sup_{n\geq 1} \prod_{i=1}^n\lambda_i=\infty.\]
\end{lem}

\begin{proof}[Proof of Theorem \ref{T:Dynamics} (i) \& (ii)]
(i) By the hypothesis, $T^*$ is almost surely a backward unilateral weighted shift with positive weights. By Lemma \ref{L:supeCyclic},
$T^*$ is almost surely supercyclic.
 (ii) The proof depends on a straightforward consequence (\cite{Durrett}, Theorem 4.1.2) of the strong law of large numbers.  {\it Fact:} Let $\{Y_n\}_{n=1}^\infty$ be positive, bounded, i.i.d., non-degenerate random variables on $(\Omega,\mathcal{F},\bP)$.
\begin{enumerate}
\item[(a)] If $\bE(\ln Y_1)>0$, then $\lim\limits_{n\to\infty}\big(\Pi_{i=1}^n Y_i\big)=\infty$ a.s.
\item[(b)] If $\bE(\ln Y_1)<0$, then $\lim\limits_{n\to\infty}\big(\Pi_{i=1}^n Y_i\big)=0$ a.s.
\item[(c)] If $\bE(\ln Y_1)=0$, then
$\liminf\limits_{n\to\infty}\big(\Pi_{i=1}^n Y_i\big)=0$ a.s. and $\limsup\limits_{n\to\infty}\big(\Pi_{i=1}^n Y_i\big)=\infty$ a.s.
\end{enumerate}
In view of Lemma \ref{L:salas}, the result is clear.
\end{proof}

\subsection{Li-Yorke chaoticity and topologically mixing property}

Let $A\in\BH$. If $\{x,y\}\subset\cH$ and
\[\limsup_{n\to\infty}\|A^n(x-y)\|>0,\quad \liminf_{n\to\infty}\|A^n(x-y)\|=0,\] then $\{x,y\}$ is called a {\it Li-Yorke chaotic pair} for $A$.
Furthermore, $A$ is called {\it Li-Yorke chaotic} \cite{LiYorke}, if there exists an uncountable subset
$G\subset\cH$ such that each pair of two distinct points in $G$ is a Li-Yorke chaotic pair for $A$.
 An operator $A\in\BH$ is said to  have {\it sensitive dependence on
initial conditions} (or simply $A$ is {\it sensitive}) if there exists $\delta > 0$ such that,
for any $x\in \cH$ and every neighborhood $G$ of $x$, one can find $y\in G$ and an
integer $n\geq 0$ such that $\|A^nx-A^ny\| \geq \delta$.

\begin{lem}\label{L:Sensive}
If $A$ is a backward unilateral weighted shift with positive weights, then $A$ is Li-Yorke chaotic if and only if $\sup_n\|A^n\|=\infty$.
\end{lem}

\begin{proof}
In view of Corollary 3.7 in \cite{hou}, a unilateral backward weighed shift $A$
is Li-Yorke chaotic if and only if $A$ is sensitive.
By Proposition 2.2 in \cite{Feldman},   $A$ is Li-Yorke chaotic if and only if $\sup_n\|A^n\|=\infty$.
\end{proof}

  An operator $A\in\BH$ is called {\it topologically
mixing} (Definition 2.1, \cite{DynamicLinOper}), if for any nonempty open subsets $G_1$ and $G_2$ , there exists a positive integer
$n$ such that $A^m(G_1)\cap G_2\ne\emptyset$ for every $m\geq n$.

\begin{lem}[\cite{Costakis},Theorem 1.2]\label{L:Mixing}
Let $A\in\BH$ be a backward weighted shift with positive weights $\{\lambda_i\}_{i=1}^\infty$. Then $A$ is topologically mixing if and only if $$\lim_{n\to\infty} (\Pi_{i=1}^n\lambda_i)=\infty.$$
\end{lem}

\begin{proof}[Proof of Theorem \ref{T:Dynamics} (iii) \& (iv)]
(iii) Recall that  $R=\max\essran X_1$. By the proof of Lemma \ref{T:NormSpectral},
$\|T^n\|=\|X_1\|_\infty^n=R^n$ a.s.
If $R\leq 1$, then $\sup_n\|T^n\|\leq 1$ a.s. By Lemma \ref{L:Sensive}, $T$ is almost surely not Li-Yorke chaotic.
If $R> 1$, then $\sup_n\|T^n\|=\infty$ a.s. By Lemma \ref{L:Sensive}, $T$ is almost surely  Li-Yorke chaotic.
(iv) By the Fact in the proof of Theorem \ref{T:Dynamics} (ii) and Lemma \ref{L:Mixing}, the result is clear.
\end{proof}

\subsection{Chaoticity and frequent hypercyclicity}

Chaoticity and frequent hypercyclicity are two qualitative strengthenings of hypercyclicity.
 An operator $A\in\BH$ is said to be {\it chaotic} \cite{Devaney} if
\begin{enumerate}
\item[(i)] for each pair of nonempty open subsets
$G_1, G_2$ of $\cH$ there exists $n\in\bN$ such that $A^n(G_1) \cap G_2 \ne\emptyset$;
 \item[(ii)] $A$ has a dense set of periodic points $(x\in \cH$ is a periodic point of $A$ if $A^kx = x$
for some $k\geq 1)$;
 \item[(iii)] $A$ is sensitive.
\end{enumerate}

  We recall that the {\it lower density} of a set of natural numbers $\Lambda$ is defined by
$$\textup{dens}(\Lambda)= \liminf_{n\rightarrow\infty}\frac{\card(\Lambda \cap [1,n])}{n}.$$
An operator $A\in\BH$ is said to be {\it frequently hypercyclic} if there exists some vector $x\in\cH$
such that $\{k\in\bN: A^kx\in G\}$ has positive lower density for every nonempty open set $G\subset\cH$.

  If an operator $A$ is chaotic or frequently hypercyclic, then $A$ is always hypercyclic  \cite{DynamicLinOper}. Thus, by Proposition 1.17 in \cite{DynamicLinOper}, it is necessary that $\sigma_p(A^*)=\emptyset$. So when considering the chaoticity and the frequent hypercyclicity of backward weighted shifts, one needs only to consider   positive weights.

\begin{lem}[\cite{DynamicLinOper}, Theorem 6.12 or \cite{Erdmann}]\label{L:Chaotic}
If $A$ is a weighted shift with positive weights $\{\lambda_i\}_{i=1}^\infty$, then $A^*$ is chaotic if and only if $\sum_{n=1}^\infty(\lambda_1\cdots\lambda_n)^{-2}<\infty$.
\end{lem}

\begin{lem}[\cite{Bayart}, Theorem 4]\label{L:FreHyp}
If $A$ is a weighted shift with positive weights $\{\lambda_i\}_{i=1}^\infty$, then $A^*$ is frequently hypercyclic
 if and only if $\sum_{n=1}^\infty$ $(\lambda_1\cdots\lambda_n)^{-2}$ $<\infty$.
\end{lem}

\begin{proof}[Proof of Theorem \ref{T:Dynamics} (v)]
In view of Claim in the proof of Lemma \ref{P:PointSpect}, Lemma \ref{L:Chaotic} and Lemma \ref{L:FreHyp}, the result is clear.
\end{proof}


\section{Regularization of weights}\label{S:Aluthge}

  Since the weights of the random model oscillates, in this (short) section we study a way to regularize them.
 The Aluthge transform of an operator originally arose in the study for $p$-hyponormal operators \cite{Aluthge,Aluth96}.
For $A\in\BH$ with polar decomposition $A=U|A|$, the Aluthge transform of $A$ is $\Delta(A)=|A|^{1/2}U|A|^{1/2}$.
For a weighted shift, the Aluthge transform is still a weighted shift, but with more regular weights.
Jung-Ko-Pearcy \cite{Jung} proved that an operator $A$ has a nontrivial
invariant subspace if and only if $\Delta(A)$ does.
 We write $\Delta^n(A)$ for the  iterated Aluthge transform.
 For each $A\in\BH$, Jung-Ko-Pearcy  \cite{Jung3} conjectured that the sequence of iterates $\{\Delta^n(A)\}_{n=1}^\infty$ converges in the strong operator topology (\sot). Antezana-Pujals-Stojanoff \cite{Antezana} proved that
the iterated Aluthge transforms of a matrix always converge. In \cite{Jung2}, Ch\={o}-Jung-Lee gave an example based on a weighted shift where the sequence of iterates does not converge
even with respect to the weak operator topology.

\begin{thm}\label{T:AluthgeTrans}
Let  $T\sim \{X_n\}_{n=1}^\infty$.
Then almost surely
\begin{equation*}\Delta^n(T)\overset{\sot}{\longrightarrow} e^{\bE(\ln X_1)}S\qquad \text{as} \quad n\rightarrow\infty,\end{equation*}
where $S$ is the unilateral (unweighted) shift. Moreover, the above convergence does not hold in the norm topology except that $ X_1$ is degenerate.
\end{thm}


%
%

\begin{proof}
Let $A$ be a unilateral weighted shift with weights $\{\lambda_i\}_{i=1}^\infty$.
It is easy to check that the Aluthge transform of $A$ is still a weighted shift, whose weights are
$\{\sqrt{\lambda_i\lambda_{i+1}}\}_{i=1}^\infty$. In particular, $\Delta^n(T)$ is a random weighted shift with weights $
Y_k(n):=\left(\prod_{i=0}^{n}X_{k+i}^{C_n^i}\right)^{1/2^n}.$
By Lemma \ref{L:iidRV}, $\|Y_k(n)\|_\infty=\|X_1\|_\infty.$
Since $\{Y_{1+s(n+1)}(n)\}_{s=0}^\infty$ are \iid, it follows from the Fact in the proof of Lemma \ref{L:quasiDiagonal}  that
$\|\Delta^n(T)\|=\|X_1\|_\infty$ a.s.
If $\bP(X_1=0)>0$, then by Lemma \ref{P:PointSpect}, $T$ is almost surely a direct sum of countably many truncated weighted shifts.
It is easy to check that the iterated Aluthge transforms of a truncated weighted shift always converges to $0$ in the norm topology.
Thus,  $\Delta^n(T)\overset{\sot}{\longrightarrow} 0=e^{\bE(\ln X_1)}S$ a.s.
Since $\|\Delta^n(T)\|=\|X_1\|_\infty$ a.s. for each $n\geq 1$, we have $\Delta^n(T)\overset{\|\cdot\|}{\longrightarrow} 0$ a.s. if and only if $\|X_1\|_\infty=0$ or equivalently $X_1$ is degenerate.
 We assume $\bP(X_1=0)=0$. Fix  $k\geq 1$.
By Theorem 2 in  \cite{Stout1968},
$\ln Y_k(n)=\frac{\sum_{i=0}^{n} C_n^i \ln X_{k+i} }{2^n}\longrightarrow \bE(\ln X_1)$ a.s. It follows that $\{Y_k(n)\}_{n=1}^\infty$ converges to $ e^{\bE(\ln X_1)}$ a.s. For all $k$ and $n$, since $Y_k(n)\leq\|X_1\|_\infty$ a.s., so almost surely,
\begin{equation}\label{(8.1)}
\Delta^n(T)\overset{\sot}{\longrightarrow} e^{\bE(\ln X_1)}S\qquad \textup{as}\quad n\rightarrow\infty. \end{equation}
\noindent
If $ X_1$ is non-degenerate, then $\bE(\ln X_1)<\ln(\|X_1\|_\infty)$ and for each $n \ge 1$,
$$\|e^{\bE(\ln X_1)} S\|=e^{\bE(\ln X_1)}<\|X_1\|_\infty=\|\Delta^n(T)\|\quad a.s.$$
%
%
Now we may conclude the proof.
\end{proof}

\section*{Acknowledgements}

G. Cheng is supported by NSFC (11871482).  X. Fang is supported by MOST
of Taiwan (106-2115-M-008-001-MY2) and NSFC (11571248) during his visits
to Soochow University in China. S. Zhu is supported by NSFC (11671167). Part of this work was done during Zhu's visit to NCU of Taiwan in the Fall semester of 2016.

\end{document}